\documentclass[phd,tocprelim]{cornell}
\usepackage[utf8]{inputenc}	
\usepackage[T1]{fontenc}	
\usepackage[USenglish]{babel}
\usepackage{amsmath}
\usepackage{amssymb}
\usepackage{enumerate}
\usepackage{graphicx}
\usepackage{aliascnt}		
\usepackage[hyperref,thmmarks,amsthm,amsmath,framed]{ntheorem}		
\usepackage{tikz}			
\usepackage{tikz-3dplot}
\usepackage{pgfplots}
\tdplotsetmaincoords{60}{110}
\usetikzlibrary{matrix,arrows,patterns,intersections,calc,decorations.pathmorphing}
\usetikzlibrary{calc,decorations.markings}
\usetikzlibrary{shapes,snakes}
\usepackage[pdftitle={Horocycle invariant closed subsets of strata: a few examples},pdfauthor={Lucien Clavier},pdfborder={0 0 0}]{hyperref}   
\usepackage[figure]{hypcap}		
\usepackage{mathtools}		
\usepackage{float}

\theoremstyle{break}
\theorembodyfont{\upshape}

\theoremstyle{remark}

\newaliascnt{rem}{definition}
\newtheorem{rem}[rem]{Remark}
\aliascntresetthe{rem}

\newaliascnt{exa}{definition}

\aliascntresetthe{exa}

\newaliascnt{cor}{definition}

\aliascntresetthe{cor}

\newaliascnt{lem}{definition}
\newtheorem{lem}[lem]{Lemma}
\aliascntresetthe{lem}

\newaliascnt{conj}{definition}

\aliascntresetthe{conj}

\newaliascnt{quest}{definition}

\aliascntresetthe{quest}

\theoremstyle{break}
\theorembodyfont{\itshape}
\newtheorem{thm}{Theorem}

\newaliascnt{prop}{definition}
\newtheorem{prop}[prop]{Proposition}
\aliascntresetthe{prop}

\DeclareMathOperator{\marked}{marked}
\DeclareMathOperator{\Aut}{Aut}
\DeclareMathOperator{\bound}{bound}
\DeclareMathOperator{\lbl}{label}

\DeclareMathOperator{\Mod}{Mod}
\DeclareMathOperator{\rel}{rel}

\DeclareMathOperator{\Der}{Der}
\DeclareMathOperator{\dev}{dev}

\DeclareMathOperator{\SL}{SL}
\DeclareMathOperator{\RE}{Re}
\DeclareMathOperator{\IM}{Im}

\DeclareMathOperator{\Push}{P}

\DeclareMathOperator{\Dir}{Dir}
\DeclareMathOperator{\WMS}{WMS}

\newcommand\myeq{\stackrel{\mathclap{\scriptsize\mbox{up to Dehn twist}}}{=\joinrel=}}

\newcommand{\RR}{\mathbb{R}}
\newcommand{\QQ}{\mathbb{Q}}
\newcommand{\CC}{\mathbb{C}}
\newcommand{\NN}{\mathbb{N}}
\newcommand{\ZZ}{\mathbb{Z}}
\newcommand{\MM}{\mathcal{M}}
\newcommand{\HH}{\mathcal{H}}
\newcommand{\Cyl}{\mathcal{C}}
\newcommand{\Bound}{\mathcal{B}}
\newcommand{\Acc}{\mathcal{A}}
\newcommand{\TM}{\widetilde{M}}
\newcommand{\TN}{\widetilde{N}}

\title{Non-affine horocycle orbit closures on strata of translation surfaces: new examples}
\author{Lucien Pierre Odilon Clavier}
\conferraldate{February}{2016}
\degreefield{Mathematics}

\begin{document}
 
\maketitle
\clearpage\mbox{}\clearpage

\begin{abstract}
We investigate specific examples of locally-defined real vector-fields on strata of translation surfaces. Integrating $\SL(2,\RR)$-loci of Veech surfaces along these vector-fields yield interesting new examples of horocycle-invariant ergodic measures. These measures are supported on closed immersed manifolds with boundary that approach the boundary of strata in a novel way.
\end{abstract}



\begin{acknowledgements}
I warmly thank my thesis advisor John Smillie for his great contribution of knowledge, ideas, funding and time. I also want to greatly thank Chenxi Wu for many interesting conversations and lots of help.
\end{acknowledgements}

\contentspage
\figurelistpage

\normalspacing \setcounter{page}{1} \pagenumbering{arabic}
\pagestyle{cornell} \addtolength{\parskip}{0.5\baselineskip}

\tikzset{middlearrow/.style={
    decoration={markings, mark= at position 0.5 with {\arrow{#1}},
    },
    postaction={decorate}
  }
}

\chapter{Introduction}
Translation surfaces and their strata have been extensively studied and represent a rich theory with many developments. See for instance the expository papers \cite{expo:Wright} and \cite{expo:Zorich}, and the seminal paper \cite{hm79}.
The interesting action of $\SL(2,\RR)$ on strata is related to the intriguing behavior of translation surfaces and presents many similarities with the action of a connected subgroup $H$ of a Lie group $G$ that is generated by unipotent elements, acting on a homogeneous quotient $X$ of $G$.

For homogeneous spaces, the careful study of flows that normalize the horocycle flow led to the celebrated Ratner's theorems \cite{ratner123}, \cite{ratner5}, \cite{ratner4} and \cite{ratner6} for the measure classification theorem which asserts that ergodic invariant measures are homogeneous, and \cite{ratner7} for the orbit equidistribution theorem which states that orbits equidistribute in their closure. 
We refer the interested reader to \cite{dwm} for an introduction to these theorems. It includes an extensive history of the study of unipotent flows on homogeneous spaces.

Strata of translation surfaces have many features in common with homogeneous spaces. 
Namely, $\SL(2,\RR)$ is generated by unipotent elements, every stratum has an affine (or affine orbifold) structure, and a canonical invariant probability measure that can be thought of as a Haar measure. 
The two settings present some critical differences as well, the most striking being that for genus higher than 1, there are non-complete horocycle-invariant vector-field on strata (see Chapter \ref{sec:background}).
In the case of strata of translation surfaces, the study of horocycle-invariant ergodic measures is a powerful tool to establish broader results about translation surfaces, see for instance \cite{cw10}, \cite{emwm06}, \cite{ems03} and \cite{sw04}.
Moreover, some results that are very similar to Ratner's theorems were proven in the case of $\SL_2(\RR)$-invariant measures and orbits in the groundbreaking papers \cite{EM} and \cite{EMM}. In particular, $\SL_2(\RR)$ orbit closures are immersed invariant affine submanifold of strata, and equidistribution theorems similar to the homogeneous case hold.
Equidistribution theorems were also obtained in \cite{bsw15} for horocycle-invariant measures in the stratum $\HH(1,1)$.

On the other hand, the careful analysis of real-rel flows in \cite{sw15} revealed examples of horocycle-invariant ergodic measures supported on a manifold with boundary and infinitely generated fundamental group, thus emphatically not homogeneous or affine.
In order to develop Ratner-type tools in strata, it is crucial to study the behavior of flows and vector-fields that normalize the horocycle flow.
Because there are few invariant flows on strata that normalize the horocycle flow, new tools need to be introduced. A fundamental idea in \cite{sw15} is to replace them by invariant vector-fields that can be used to create new horocycle-invariant measures. 
The first example of a vector-field investigated in \cite{sw15} is real-rel, which implies that directions and lengths of loops stay unchanged under the integration of these vector-fields. This case is studied in depth, and the obstruction to integrate vector-fields for all times comes from a saddle connection shrinking to a point.  See also \cite{bsw15} where these vector-fields are examined in detail. 
Moreover, the generalized trajectories obtained by approaching non-complete trajectories by nearby-defined trajectories (see \cite{sw15}, Section 7) hit the boundary of strata for a discrete set of times.  

In this thesis, we exhibit new behaviors of horocycle-invariant ergodic measures on strata that come from non-rel real (locally-defined) vector-fields, and we investigate what happens when the obstruction to integrate such vector-fields comes from a whole subsurface with boundary degenerating to a point, rather than just a saddle connection. In that case, we will see examples that show that generalized trajectories can hit boundaries of strata in a non-discrete set of times.

After some background and preliminaries, we start in Section \ref{sec:WMSpush} by studying the example of the locus of the Eierlegende Wollmilchsau, a translation surface taking its name from the mythical creature “egg-laying wool-milk-sow” that produces anything one might need.
It has extraordinary properties and produced many counterexamples, see \cite{fm14}, \cite{hs08} and \cite{ss13}.
In Section \ref{sec:WMSpush}, by ``pushing'' this locus along a particular vector-field, we obtain a space $\Push(\MM_{\WMS})$ which is the support of a horocycle-invariant ergodic measure and that presents a new behavior that we describe as ``infinite catastrophes in finite time''.
See Theorem \ref{thm:balls} which gives a specific homeomorphism between a slice of an open ball in $\MM_{\WMS}$ and its image in $\Push(\MM_{\WMS})$.
In this specific case, the ``infinite catastrophes'' phenomenon is due to a cylinder getting more and more horizontal. This is sketched in Figure \ref{fig:Nhalf} and studied in detail in Chapter \ref{chap:BS}.
In Chapter \ref{chap:generalCase} we look at the more general case of Veech surfaces $M$ for which the vector $v$ to push along has some finiteness property, namely $\Dir(M, v)$ is finite, and we show that in that case as well, ``infinite catastrophes'' exactly happen when a cylinder gets more and more horizontal, see Theorem \ref{thm:Acc}.
In Chapter \ref{chap:BS}, we also investigate how $\Push(\MM_{\WMS})$ approaches a Borel-Serre type of blow-up of the limit point in a simple case, see Theorem \ref{thm:spiraling}.

\chapter{Background}
\label{sec:background}
\section{Strata of translation surfaces}

A translation surface is a connected, orientable, compact surface $M$ with a finite set of points (called singular points) $\Sigma \subset M$ such that transition maps coming from charts on $M\setminus \Sigma$ are translations.

The Euclidean metric on $\RR^2$ induces a metric on a translation surface $M$ with set of singular points $\Sigma$ via pull-backs of charts on $M\setminus \Sigma$.
Under this metric, points in $M\setminus \Sigma$ have neighborhoods isometric to Euclidean disks, and singular points have neighborhoods isometric to cyclic covers of Euclidean disks. 
We define the order of a singular point to be $n-1$ where $n$ is the order of such a cyclic cover.
The order of a singular point is 0 when it admits a neighborhood isometric to a Euclidean disk.
By the Gauss-Bonnet Theorem, the sum of the orders of singular points is $2g-2$ with $g$ the genus of $M$.

Given two translation surfaces $M$ and $N$ with sets of singular points $\Sigma_M$ and $\Sigma_N$ respectively, 
we define an affine isomorphism to be an orientation-preserving homeomorphism $\phi: M \to N$ 
such that $\phi(\Sigma_M) = \phi(\Sigma_N)$ and which is affine on charts of $M\setminus \Sigma_M$.
We define the derivative $\Der(\phi)$ of $\phi$ to be the linear part of $\phi$. If $\Der(\phi)$ is the identity, then we say that $\phi$ is an equivalence of translation surfaces, and we say $M$ and $N$ are equivalent.

We define a marked translation surface $(S, \psi, M)$ as a triple where $S$ is a connected, orientable, compact surface, with a finite set of points $\Sigma_S \subset S$, $M$ is translation surface with set of singular points $\Sigma_M$ and $f: S \to M$ a homeomorphism with $f(\Sigma_S)=\Sigma_M$. 
We say that two marked translation surfaces $(S, \psi_1, M)$ and $(S, \psi_2, M)$ are equivalent if $\psi_1$ and $\psi_2$ are isotopic rel $\Sigma_M$.
Note that $S$ and $M$ must have the same genus $g$. Fixing an order on $\Sigma_S$ once and for all induces an order of $\Sigma_M$ via $f$, so we think of $\Sigma_M$ as an ordered set $(x_1, \ldots, x_n)$. Write $\kappa=(k_1, \ldots, k_n)$ for the $n$-tuple of orders of the $x_i$. This is a partition of $2g-2$.

Given a genus $g$ and a partition $\kappa$ of $2g-2$, define the stratum $\HH^{\marked}(\kappa)$ to be the set of all marked translation surfaces with partition $\kappa$.
Integrating paths in $M$ gives rise to a well-defined map $\dev: \HH^{\marked}(\kappa)\to H^1(M, \Sigma_M; \CC)\simeq \CC^{2g+n-1}$, called the developing map. This map was first introduced in \cite{dh75} under the name ``period map''. Pulling back the Euclidean structure of $\CC^{2g+n-1}$ makes $\HH^{\marked}(\kappa)$ a complex manifold which is complete exactly when $n=1$.
As in \cite{bsw15}, we want to define a version of strata for which singular points have labels that are well defined.
Thus, define two surfaces of type $\kappa$ to be label-preserving translation equivalent when they are translation equivalent via a label-preserving map.
We then define the stratum $\HH^{\lbl}(\kappa)$ or simply $\HH(\kappa)$ as the set of translation surfaces $M$ with partition $\kappa$, up to label-preserving translation equivalence. 
The surjective map $(S, \psi, M) \mapsto M$ makes $\HH(\kappa)$ into a complex orbifold $\HH(\kappa) = \HH^{\marked}(\kappa)/\Mod^{\lbl}(S,\Sigma_S)$, with $\Mod^{\lbl}(S, \Sigma_S)$ the mapping class group of homeomorphisms of $S$ that preserve $\Sigma_S$, with label. 

Pulling back the Lebesgue measure on $\CC^{2g+n-1}$ gives a measure on $\HH^{\marked}(\kappa)$, which is foliated by affine leaves corresponding to the surfaces of same area. 
We restrict to the leaf of area-1 marked translation surfaces $\HH^{\marked}_1(\kappa)$ and disintegrate that measure to endow $\HH^{\marked}_1(\kappa)$ with a measure $\mu$. 
The push-forward of $\mu$ on the space $\HH_1(\kappa)$ of area-1 labeled translation surfaces gives it the structure of a measured space with finite volume. 

Finally, $\SL(2,\RR)$ acts on both $\HH^{\marked}_1(\kappa)$ and $\HH_1(\kappa)$ by post-composition of the charts of $M$.
$\HH_1(\kappa)$ has thus many of the features of homogeneous spaces $G/\Lambda$ with $G$ a Lie group containing a subgroup $H\simeq \SL(2, \RR)$ and $\Lambda$ a lattice of $G$. Namely, complex orbifold structure replaces homogeneity, $\mu$ replaces the Haar measure on $G/\Lambda$ and the $\SL(2, \RR)$ action replaces the action of $H$. An important difference is that horocycle-invariant vector-fields $\HH_1(\kappa)$ is not complete in general.

\section{Mirzakhani-Wright partial compactification of strata}
\label{sec:MirWri}
In this section, we recall briefly the Mirzakhani-Wright partial compactification of strata developed in \cite{mzw15} where several examples are carefully worked out.
From this viewpoint, saddle connections and sub-surfaces are allowed to degenerate to points or 1-dimensional trees, and the way this degeneration happens does not affect the limit.
This is in contrast with Borel-Serre types of compactifications where all the degenerating subsurfaces are scaled back to have area 1 and contribute in the limit. 
A Borel-Serre type description will be discussed in Chapter \ref{chap:BS} in a simple case.

Following \cite{mzw15}, as we want to allow a surface to degenerate into a disjoint union of several surfaces, we define strata of multicomponent area-1 marked (resp. labeled) surfaces as $\HH^{\marked}_1(\kappa_1) \times \ldots \times \HH^{\marked}_1 (\kappa_k)$ (resp. $\HH_1(\kappa_1) \times \ldots \times \HH_1(\kappa_k)$) for some partitions $\kappa_i$, where each component surface $M^i$ is scaled by $\frac1k$ so that the multicomponent surface $(M^1, \ldots , M^k)$ has area 1.

%
Let $\HH$ and $\HH'$ be two strata of multicomponent area-1 surfaces. 
Let $M_n$ be a sequence of surfaces in $\HH$ and $M$ a surface in $\HH'$, with tuples of singular points $\Sigma_n$ and $\Sigma$ respectively.
We say that $M_n$ converges to $M$ if there exist decreasing neighborhoods $U_n$ of $\Sigma$ in $M$ with $\cap U_n = \Sigma$, and maps $g_n: M \setminus U_n \to M_n$ that are diffeomorphisms onto their image and that verify:
\begin{enumerate}
	\item given any embedded triangle $T$ of $M\setminus \Sigma$, $g_n(T)$ is an embedded triangle $T_n$ of $M_n\setminus \Sigma_n$ for $n$ large enough, and for any side $s_i$ of $T$, $\dev(M_n, g_n(s_i)) \to \dev(M, s_i)$, i.e. the sides $g_n(s_i)$ of $T_n$ have lengths and directions converging to that of~$s_i$.
	\item the injectivity radii of points that do not lie in the image of $g_n$ converge to 0 uniformly.
\end{enumerate}



Let $M_n$ be a converging sequence of translation surfaces. If for each $n$ there is a saddle connection $\sigma_n\subset M_n$ (resp. a cylinder $\Cyl_n \subset M_n$) with length $l_n$ verifying $l_n\to 0$ (resp. with the supremum of injectivity radii $r_n$ verifying $r_n\to 0$), we say that a saddle connection (resp. a cylinder) degenerates as $n\to \infty$.
In Lemmas \ref{lem:ex1}, \ref{lem:ex2} and \ref{lem:ex3} below we investigate explicit examples of convergence and divergence in the sense of Mirzakhani-Wright. 
These form a catalog of behaviors that we will encounter below. 
\begin{lem}
\label{lem:ex1}
Let $M$ be a marked (resp. labeled) translation surface with a cylinder $\Cyl\subset M$. 
For any $h>0$, define $M_h$ to be the surface obtained from $M$ by letting $M\setminus \Cyl$ fix while the height of $\Cyl$ is set to $h$ and no extra shear happens inside $\Cyl$ 
(otherwise put, the twist $\widetilde{t} \in \RR$ (resp. $t\in S^1= [0,1]/(0\sim 1)$) stays constant), and by then rescaling so that $M_h$ has area 1.
Then the limit of $M_h$ in the sense of Mirzakhani-Wright as $h\to 0$ is the one you expect, namely the surface obtained from $M$ by cutting out $\Cyl$ and identifying the two end-loops $\gamma_i$ along the Interval Exchange Map corresponding to flowing inside $\Cyl$ from $\gamma_1$ to $\gamma_2$ in the direction normal to $\gamma_1$, then rescaled to have area 1. 
If the flow inside $\Cyl$ in the direction normal to $\gamma_1$ connects two singular points, then a saddle connection degenerates as $n\to \infty$.
Otherwise, there is no saddle connection degenerating as $n\to \infty$.
\end{lem}
\begin{proof}
See the first example in the Example section of \cite{mzw15}.
\end{proof}

\begin{lem}
\label{lem:ex2}
Let $M_n$ be a sequence of marked (resp. labeled) translation surfaces in the same stratum and such that $M_n$ converges to some surface $M$ in the sense of Mirzakhani-Wright. 
Assume that for each $n$ there is a saddle connection $\sigma_n$ in $M_n$ and that the lengths $|\sigma_n|$ converge to 0.
For each $n$, define $\widetilde{M_n}$ as the surface (rescaled to have area 1) obtained from $M_n$ by cutting along $\sigma_n$ and attaching a cylinder $\Cyl_n$ (see Figure \ref{fig:crossCyl} for the example for which $M_n$ is a torus with two singular points of order 0) of height $h_n <1$ with end-loops at each of the two open slits. 
Then $\widetilde{M_n}$ converges to $M$ in the sense of Mirzakhani-Wright, and there is a cylinder degenerating as $n\to \infty$
\end{lem}
\begin{proof}
Let $U_n \subset M$ and $g_n: M \setminus U_n \to M_n$ as in the definition of Mirzakhani-Wright convergence. 
Let $\widetilde{U_n} = U_n \cup \sigma$, $\iota_n: M_n \setminus \sigma_n \to \widetilde{M_n}$ the inclusion map (with scaling) and $\widetilde{g_n} = \iota_n \circ (g_n |_{M_n\setminus \sigma_n})$. 
Then since the injectivity radii of points in $\Cyl_n$ is less than or equal to $|\sigma_n|$, we see that $\widetilde{M_n}$, $M$, $\widetilde{U_n}$ and $\widetilde{g_n}$ verify the Mirzakhani-Wright convergence definition. We use the fact that the scaling factor $\frac1{1+|\sigma_n|h_n}$ tends to 1 to show (i).
\end{proof}

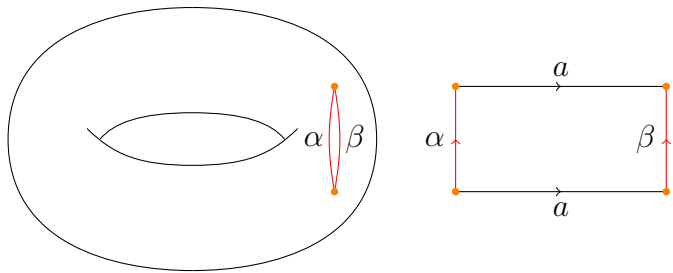
\begin{figure}
\begin{center}
\begin{tikzpicture}[scale=.7]

\begin{scope} [shift={(-5,0)}]
\draw (-3.5,0) .. controls (-3.5,2) and (-1.5,2.5) .. (0,2.5);
\draw[xscale=-1] (-3.5,0) .. controls (-3.5,2) and (-1.5,2.5) .. (0,2.5);
\draw[rotate=180] (-3.5,0) .. controls (-3.5,2) and (-1.5,2.5) .. (0,2.5);
\draw[yscale=-1] (-3.5,0) .. controls (-3.5,2) and (-1.5,2.5) .. (0,2.5);

\draw (-2,.2) .. controls (-1.5,-0.3) and (-1,-0.5) .. (0,-.5) .. controls (1,-0.5) and (1.5,-0.3) .. (2,0.2);

\draw (-1.75,0) .. controls (-1.5,0.3) and (-1,0.5) .. (0,.5) .. controls (1,0.5) and (1.5,0.3) .. (1.75,0);

  \draw[color=red] (2.7,-1) to [bend left=10] (2.7,1);
    \node[left] at (2.7,0) {$\alpha$};
  \draw[color=red] (2.7,-1) to [bend right=10] (2.7,1);
    \node[right] at (2.7,0) {$\beta$};
  \fill[orange] (2.7,1) circle (2pt);
  \fill[orange] (2.7,-1) circle (2pt);
\end{scope}

\draw[middlearrow={>},red] (0, -1)--(0, 1);
    \node[left] at (0,0) {$\alpha$};
\draw[middlearrow={>},red] (4, -1)--(4, 1);
    \node[left] at (4,0) {$\beta$};
\draw[middlearrow={>}] (0, 1)--(4, 1);
    \node[above] at (2,1) {$a$};
\draw[middlearrow={>}] (0, -1)--(4, -1);
    \node[below] at (2,-1) {$a$};

  \fill[orange] (0,1) circle (2pt);
  \fill[orange] (4,1) circle (2pt);
  \fill[orange] (0,-1) circle (2pt);
  \fill[orange] (4,-1) circle (2pt);

\end{tikzpicture}
\end{center}
\caption{Left: a torus with a slit. Right: a horizontal cylinder. We glue these two surfaces with boundary along $\alpha$ and $\beta$. The new singular point (in orange) of the resulting surface has order 2. The sequence of such surfaces $M_n$ with slits of length $\frac1n$ has a Mirzakhani-Wright limit which is a torus with a singular point of order 0.}
\label{fig:crossCyl}
\end{figure}


\begin{lem}
\label{lem:ex3}
For any positive number $h$, let $M_h$ as in Lemma \ref{lem:ex1}. 
Let $h_n$ be any sequence of positive numbers with $h_n \to \infty$. 
Then $M_{h_n}$ does not converge in the sense of Mirzakhani-Wright.
\end{lem}
\begin{proof}
	Assume by contradiction that $M_{h_n} \to M'$ with $M'$ in some stratum $\HH'$.
	We will prove that the injectivity radii of points in $M_{h_n}$ converge to 0, thus (i) cannot be satisfied.
	Let $x\in M$.
	If $x \in M\setminus \Cyl$, let $\gamma_x$ be a path in $M \setminus \Cyl$ from $x$ to a singular point of $M$. 
	Otherwise, let $\gamma_x$ be a loop around $\Cyl$ in a direction parallel to its end-loops. This is a saddle connection of length $l$ with $l$ the circumference of $\Cyl$. 
	Since the scaling factor from $M$ to $M_{h_n}$ is $\frac1{1+lh_n}$, then the injectivity radius of $x$ in $M_{h_n}$ is at most $\frac{l}{1+lh_n}$.
\end{proof}

\chapter{Pushes of loci}
\label{chap:push}
%
\section{Preliminaries}
Define the geodesic flow 
\[
g_t = \left( \begin{array}{cc}
e^{t/2} & 0 \\
0 & e^{-t/2} \end{array} \right) 
\] 
and the horocycle flow
\[
h_t = \left( \begin{array}{cc}
1 & t \\
0 & 1 \end{array} \right).
\]
The flow $h_t$ is the flow of interest here, and $g_t$ is interesting because it normalizes~$h_t$: 
$$g_t h_s g_{-t} = h_{s.e^t}.$$
We say that a measure $\mu$ (resp. a closed set $F$) is horocycle-invariant if the pushforwards $\mu(h_t^{-1}(\cdot))$ of that measure (resp. the closed sets $h_t^{-1}(F)$) are all equal.

    Throughout this section, we consider a stratum of marked area-1 translation surfaces $\HH^{\marked} =\HH^{\marked}_1(\kappa)$ and we fix a marked translation surface $(S, \psi, \underline{M})\in \HH^{\marked}$. We write $\MM^{\marked}$ for the closure of the $\SL(2,\RR)$-orbit of $\underline{M}$ in $\HH^{\marked}$. We also write $\HH^{\lbl}$ for $\HH^{\lbl}_1(\kappa)$ and $\MM^{\lbl} \subset \HH^{\lbl}$ for the image of $\MM^{\marked}$ under the map $(S, \psi, M)\to M$.

Let $v$ be a vector in $ \dev(\HH^{\marked})\subset \CC^{2g+n-1}$. By pulling back the constant vector-field equal to $v$ on $\dev(\HH^{\marked})$ via the developing map, we obtain a well-defined vector-field on $\HH^{\marked}$.
    If $v\in \RR^{2g+n-1}\subset \CC^{2g+n-1}$ then the induced vector-field on $\HH^{\marked}$ commutes with the horocycle flow and we call it a real vector-field. If its image in $H^1(\underline{M}; \CC)$ is trivial, then we say $v$ is real-rel.
    We would like to define such vector-fields on $\HH^{\lbl}$. 
A possible obstruction here is that the image of a vector $v \in T\HH^{\marked}$ in $T\HH^{\lbl}$ is not necessarily finite, so we will work in a finite cover instead.

Therefore, consider the monodromy action of $\pi_1(\HH^{\lbl}, \underline{M})$ on the fiber of~$M$ $\Mod^{\lbl}(S, \Sigma_S). (S, \psi, \underline{M})$. 
It induces an action on $T_{\underline{M}}\HH^{\lbl}$.
By restricting to $\pi_1(\MM^{\lbl}, \underline{M})$, we obtain an action of $\pi_1(\MM^{\lbl}, \underline{M})$ on $T_{\underline{M}}\HH^{\lbl}$. 
Given a vector $v\in T_{\underline{M}}\HH^{\lbl}$, we write $\Dir(\underline{M}, v) $ for the orbit of a vector $v$ under this action.
Thanks to \cite{w15} where the computation of $\Dir(M, v)$ is carefully detailed in the case where $M$ is an abelian cover of the pillow-case, we know that infinitely many examples of pairs $(M,v)$ can be constructed for which $\Dir(M,v)$ is a finite set. See the table of examples in \cite{w15}[Section 6].  We will study an explicit example of such a pair in Section \ref{sec:WMSpush}.

Letting $k=|\Dir(\underline{M}, v)|$, we now define $\widehat{\HH}$ to be the $k$-fold cover of $\HH^{\lbl}$ corresponding to the kernel of $\pi_1(\MM^{\lbl}, \underline{M})\to \Aut(T_{\underline{M}}\HH^{\lbl})$, a subgroup of $\pi_1(\MM^{\lbl}, \underline{M})$ of index $k$.
We also define $\widehat{\MM}\subset \widehat{\HH}$ to be the pre-image of $\MM^{\lbl}$ under the covering projection. It is a $k$-fold cover of $\MM^{\lbl}$.
Now by construction, $\pi_1(\widehat{\MM}, \underline{M})$ acts trivially on $T_{\underline{M}}\widehat{\HH}$ and the vectorfield defined by $v$ on $\HH^{\marked}$ descends to a well-defined vector-field on $\widehat{\HH}$.
Moreover, the action of $\SL(2,\RR)$ on $\HH^{\lbl}$ lifts to $\widehat{\HH}$, the pullback $\widehat{\mu}$ of the Masur-Veech measure $\mu$ on $\HH^{\lbl}$ is still $\SL(2, \RR)$-invariant and it is finite because $\widehat{\MM}$ is a finite cover of $\MM^{\lbl}$.

We define the push $\Push(M,v)$ of a translation surface $M$ in $\MM^{\marked}$ (resp. in $\widehat{\MM}$) along the vector $v$ as the surface obtained by flowing $M$ along the vector-field induced by $v$ on $\HH^{\marked}$ (resp. on $\widehat{\HH}$) for time 1, when defined.
We show in Proposition \ref{prop:normal} that, when the $\SL(2,\RR)$-orbit of $\underline{M}$ is closed, then $\Push(M,v) $ is well-defined on a set of positive measure of $\widehat{\MM}$, and thus by a standard ergodicity argument, it is well-defined for $\widehat{\mu}$-almost every $M$ in $\widehat{\MM}$.
See also \cite{mw14}[Theorem 1.2] which implies that in that case, $\Push(M,v) $ is well-defined as soon as it does not have a horizontal saddle connection. Again, this implies that the push of $M$ along $v$ is well-defined for almost every $M$.

We now define the push of $\MM^{\marked}$ (resp. of $\widehat{\MM}$) along the vector $v$ as the closure of the space of all pushes $\Push(M, v)$ with $M\in \MM^{\marked}$ (resp. $M\in \widehat{\MM}$). We write this push $\Push(\MM^{\marked}, v)$ (resp. $\Push(\widehat{\MM}, v)$).
Because the vector-field induced by $v$ commutes with $h_t$, the closed set $\Push(\MM^{\marked}, v)$ (resp. $\Push(\widehat{\MM},v)$) is horocycle-invariant.
Moreover, the push-forward $\Push_\ast \nu$ of the natural measure on the locus $\widehat{\MM}$ under the map $\Push: (M,v)\mapsto \Push(M,v)$ is a horocycle-invariant measure on $\Push(\widehat{\MM},v)$ of finite volume.
Finally, we obtain an invariant closed set $\Push(\MM^{\lbl},v)$ with invariant measure by taking the image of $\Push(\widehat{\MM}, v)$ under the covering projection to $\HH^{\lbl}$.

It will be useful to define the subvector-space $F \subset T_{\underline{M}} \HH^{\lbl}$ to be the set of real vectors $v$ for which $|\Dir(\underline{M}, v)|$ is finite. 
The set $\RR-\rel$ of real-rel vectors is a subspace of $F$. 
At this point, it is not clear how the vector-space structure of $F$ translates in terms of pushes of $\underline{M}$.
For instance, if we can describe pushes of $\underline{M}$ along vectors $v_1$ and $v_2$, it does not follow that we understand pushes along $v_1 + v_2$.
Recall that the computation of $\Dir(\underline{M}, v)$ was completely described in \cite{w15} in the case of abelian covers of the pillow-case.

\section{A push of the Wollmilchsau locus}
\label{sec:WMSpush}

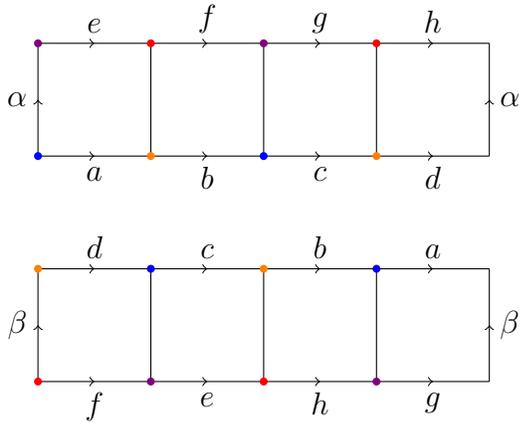
\begin{figure}
\begin{center}
\begin{tikzpicture}[scale=1.5]
\foreach \i in {0,...,3}{%
    \foreach \j in {0,...,3}{%
        \draw[middlearrow={>}] (\i,\j)--(\i+1,\j) ;}}
\foreach \j in {0,2}{%
    \foreach \i in {0,4}{%
        \draw[middlearrow={>}] (\i,\j)--(\i,\j+1) ;}
    \foreach \i in {1,...,3}{%
        \draw[-] (\i,\j)--(\i,\j+1) ;}}
\foreach \i in {0, 2}{%
    \fill[red] (\i,0) circle (1pt);
    \fill[red] (\i+1,3) circle (1pt);
    \fill[orange] (\i,1) circle (1pt);
    \fill[orange] (\i+1,2) circle (1pt);
    \fill[blue] (\i,2) circle (1pt);
    \fill[blue] (\i+1,1) circle (1pt);
    \fill[violet] (\i,3) circle (1pt);
    \fill[violet] (\i+1,0) circle (1pt);}

\node[left] at (0,2.5) {$\alpha$};
\node[right] at (4,2.5) {$\alpha$};
\node[left] at (0,0.5) {$\beta$};
\node[right] at (4,0.5) {$\beta$};

\node[above] at (0.5,3) {$e$};
\node[above] at (1.5,3) {$f$};
\node[above] at (2.5,3) {$g$};
\node[above] at (3.5,3) {$h$};
\node[below] at (0.5,2) {$a$};
\node[below] at (1.5,2) {$b$};
\node[below] at (2.5,2) {$c$};
\node[below] at (3.5,2) {$d$};

\node[above] at (0.5,1) {$d$};
\node[above] at (1.5,1) {$c$};
\node[above] at (2.5,1) {$b$};
\node[above] at (3.5,1) {$a$};
\node[below] at (0.5,0) {$f$};
\node[below] at (1.5,0) {$e$};
\node[below] at (2.5,0) {$h$};
\node[below] at (3.5,0) {$g$};

\end{tikzpicture}
\end{center}
\caption{The WMS as a square-tiled marked translation surface with gluings.}
\label{fig:WMS}
\end{figure}

We now investigate a specific example using the locus of the Eierlegende Wollmilchsau.
The Eierlegende Wollmilchsau (WMS) $M_{\WMS}$ is the square-tiled marked (resp. labeled) translation surface represented as the union of squares with gluings in Figure \ref{fig:WMS}.
It has genus 3 and 4 singular points of order 1, i.e. its stratum is $\HH_{\WMS} =\HH_1(1,1,1,1)$. Since $M_{\WMS}$ is square-tiled, its $\SL_2(\RR)$-orbit $\MM_{\WMS}$ is closed.
 
Each lettered label used for gluings in Figure \ref{fig:WMS} is a path $[0, 1]\to S$ with $S$ the connected, orientable, compact surface of genus 3.
When the context is clear, we think of them as elements of $H_1(S, \Sigma)$.
Recall that we write $\dev(M, x) \in \CC$ for the integral of $dz$ along the edge labeled $x$ in $M$.
When the context is clear, we will just write $\vec{x}$.

Recall that $F$ is the subvector-space of real vectors $v$ for which the set of directions $|\Dir(M_{\WMS}, v)|$ is finite. 
Here, $F$ is of dimension 7 and consists of the real vectors $v$ verifying
\[ v(a+b+c+d)=v(e+f+g+h)=v(2a+b-d-e-g-2\alpha-2\beta)=0. \]
These conditions are equivalent to preserving area, i.e. $\Push(M_{\WMS}, tv)$ always has area 1 for any $t$ such that this push is defined.
This is a particularity of the WMS.
Moreover $\dim(\RR-\rel)=3$ and for any $v$ in $F\setminus \RR-\rel$, $\Dir(M_{\WMS}, v)$ spans a subvector-space of dimension 2.
See \cite{my10} or \cite{w15}. 

In the remainder of this section and in Chapter \ref{chap:BS}, we will study the pushes of the WMS locus under the vector $v_0$ defined by
    $v_0(a)=-v_0(c)=1$, $v(b)=v(c)=v(e)=v(f)=v(g)=v(h)=0$ and $v_0(\alpha) = v_0(\beta) = \frac{1}{2}$.

Let us now explicitly describe $\Dir(M_{\WMS}, v_0)$. 
\begin{lem}

$\Dir(M_{\WMS}, v_0)$ has 8 elements:
\begin{itemize}
    \item $v_1$ defined by $v_1(b)=-v_1(d)=1$, $v_1(\alpha) = v_1(\beta) = \frac{1}{2}$ and other labels are mapped to 0
    \item $v_2$ defined by $v_2(e)=-v_2(g)=1$, $-v_2(\alpha) = v_2(\beta) = \frac{1}{2}$ and other labels are mapped to 0
    \item $v_3$ defined by $v_3(f)=-v_3(h)=1$, $-v_3(\alpha) = v_3(\beta) = \frac{1}{2}$ and other labels are mapped to 0
\end{itemize}
as well as $v_0$, $-v_0$, $-v_1$, $-v_2$ and $-v_3$.
\end{lem}
\begin{proof}
This can be done by direct computation or can also be derived from \cite{w15}.
Note that the $\Push(M, \pm v_i)$ are either all non-defined or differ from each other by label-preserving translation equivalences. This follows from a standard cut-and-paste argument.
\end{proof}

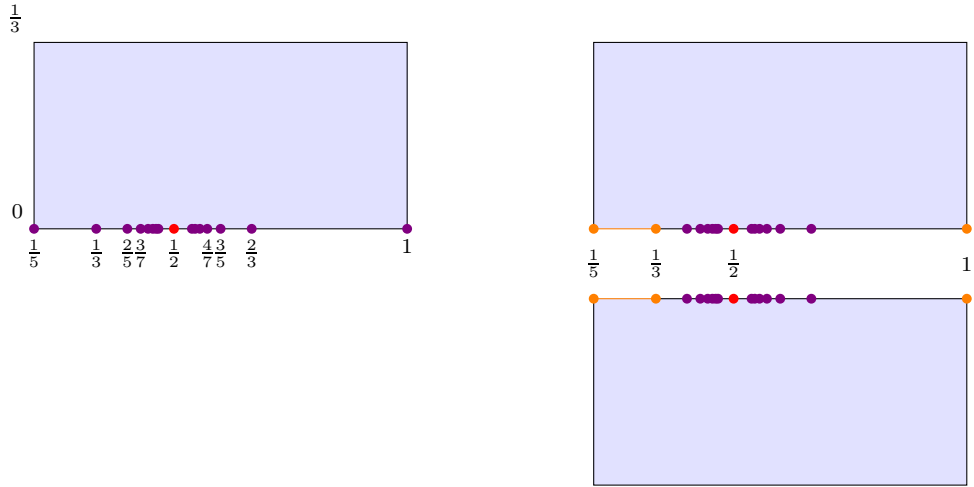
\begin{figure}
\begin{center}
\begin{tikzpicture}[scale=6.2]

    \fill[blue!12!white] (.2,0)--(1,0)--(1,.4)--(.2,.4)--(.2,0);
    \draw[-] (.2,0)--(1,0)--(1,.4)--(.2,.4)--(.2,0);
    \draw (.2,0) node[draw=none,fill=none,font=\scriptsize, above left] {$0$};
    \draw (1,0) node[draw=none,fill=none,font=\scriptsize,below] {1};
    \draw (.2,.4) node[draw=none,fill=none,font=\scriptsize,above left] {$\frac13$};


    \draw (.5,0) node[draw=none,fill=none,font=\scriptsize,below] {$\frac{1}{2}$};
        \fill[red] (.5,0) circle (0.3pt);
    \draw (.2,0) node[draw=none,fill=none,font=\scriptsize,below] {$\frac15$};
        \fill[violet] (1,0) circle (0.3pt);
        \fill[violet] (.2,0) circle (0.3pt);
    \draw (.666667,0) node[draw=none,fill=none,font=\scriptsize,below] {$\frac{2}{3}$};
        \fill[violet] (.6666667,0) circle (0.3pt);
    \draw (.6,0) node[draw=none,fill=none,font=\scriptsize,below] {$\frac{3}{5}$};
        \fill[violet] (.6,0) circle (0.3pt);
    \draw (.571429,0) node[draw=none,fill=none,font=\scriptsize,below] {$\frac{4}{7}$};
        \fill[violet] (.571429,0) circle (0.3pt);
    \draw (.333333,0) node[draw=none,fill=none,font=\scriptsize,below] {$\frac{1}{3}$};
        \fill[violet] (.333333,0) circle (0.3pt);
    \draw (.4,0) node[draw=none,fill=none,font=\scriptsize,below] {$\frac{2}{5}$};
        \fill[violet] (.4,0) circle (0.3pt);
    \draw (.428571,0) node[draw=none,fill=none,font=\scriptsize,below] {$\frac{3}{7}$};
        \fill[violet] (.428571,0) circle (0.3pt);
        \fill[violet] (.444444,0) circle (0.3pt);
        \fill[violet] (.555556,0) circle (0.3pt);
        \fill[violet] (.454545,0) circle (0.3pt);
        \fill[violet] (.545454,0) circle (0.3pt);
        \fill[violet] (.461538,0) circle (0.3pt);
        \fill[violet] (.538462,0) circle (0.3pt);
        \fill[violet] (.466667,0) circle (0.3pt);

\begin{scope} [shift={(1.2,0)}]

    \fill[blue!12!white] (.2,0)--(1,0)--(1,.4)--(.2,.4)--(.2,0);
    \draw[-] (.333333,0)--(1,0)--(1,.4)--(.2,.4)--(.2,0);
    \draw[-, orange] (.333333,0)--(.2,0);

%

    \draw (.5,-.075) node[draw=none,fill=none,font=\scriptsize] {$\frac{1}{2}$};
    \draw (.2,-.075) node[draw=none,fill=none,font=\scriptsize] {$\frac15$};
    \draw (.333333,-.075) node[draw=none,fill=none,font=\scriptsize] {$\frac{1}{3}$};
    \draw (1,-.075) node[draw=none,fill=none,font=\scriptsize] {$1$};
        \fill[red] (.5,0) circle (0.3pt);
        \fill[orange] (1,0) circle (0.3pt);
        \fill[orange] (.2,0) circle (0.3pt);
        \fill[violet] (.6666667,0) circle (0.3pt);
        \fill[violet] (.6,0) circle (0.3pt);
        \fill[violet] (.571429,0) circle (0.3pt);
        \fill[orange] (.333333,0) circle (0.3pt);
        \fill[violet] (.4,0) circle (0.3pt);
        \fill[violet] (.428571,0) circle (0.3pt);
        \fill[violet] (.444444,0) circle (0.3pt);
        \fill[violet] (.555556,0) circle (0.3pt);
        \fill[violet] (.454545,0) circle (0.3pt);
        \fill[violet] (.545454,0) circle (0.3pt);
        \fill[violet] (.461538,0) circle (0.3pt);
        \fill[violet] (.538462,0) circle (0.3pt);
        \fill[violet] (.466667,0) circle (0.3pt);
\end{scope}

\begin{scope} [shift={(1.2,.1)}]
    \fill[blue!12!white] (.2,-.25)--(1,-.25)--(1,-.65)--(.2,-.65)--(.2,-.25);
    \draw[-] (.333333,-.25)--(1,-.25)--(1,-.65)--(.2,-.65)--(.2,-.25);
    \draw[-, orange] (.333333,-.25)--(.2,-.25);
        \fill[red] (.5,-.25) circle (0.3pt);
        \fill[orange] (1,-.25) circle (0.3pt);
        \fill[orange] (.2,-.25) circle (0.3pt);
        \fill[violet] (.6666667,-.25) circle (0.3pt);
        \fill[violet] (.6,-.25) circle (0.3pt);
        \fill[violet] (.571429,-.25) circle (0.3pt);
        \fill[orange] (.333333,-.25) circle (0.3pt);
        \fill[violet] (.4,-.25) circle (0.3pt);
        \fill[violet] (.428571,-.25) circle (0.3pt);
        \fill[violet] (.444444,-.25) circle (0.3pt);
        \fill[violet] (.555556,-.25) circle (0.3pt);
        \fill[violet] (.454545,-.25) circle (0.3pt);
        \fill[violet] (.545454,-.25) circle (0.3pt);
        \fill[violet] (.461538,-.25) circle (0.3pt);
        \fill[violet] (.538462,-.25) circle (0.3pt);
        \fill[violet] (.466667,-.25) circle (0.3pt);
\end{scope}
\end{tikzpicture}
\end{center}
\caption{Left: $R$ as a rectangle and $U$ as a subset of $R$ (in violet). $(\frac12, 0)$ is the only non-isolated element of $U$. Right: the pair $(\widetilde{R}, \widetilde{U})$ as two copies of $(R,U)$ with the points corresponding to $(\frac12, 0)$, $[\frac15, \frac13]\times \{0\}$ and $(1,0)$ identified (in red and orange respectively).}
\label{fig:U}
\end{figure}

We are going to parametrize the part of the WMS locus $\MM_{\WMS}$ that consists of the surfaces $g. M_{\WMS}$ with $g\in \SL(2, \RR)$ a matrix with positive upper-left entries in the following way: for $A>0$ and $B, C \in \RR$, set
\[
    M_{A, B, C} =\left( \begin{array}{cc}
A & B \\
C & \frac{1+BC}{A} \end{array} \right) .M_{\WMS}.
\]
When defined, set also
\[
N_{A, B, C} = \Push(M_{A, B,C},v_0).
\]
The ``infinite catastrophes in finite time'' phenomenon happens around $A=\frac{1}{2}$. $A$ (resp. $C$) represents horizontal scaling (resp. vertical shear) of horizontal edges, while $B$ parametrizes a horizontal shear on vertical edges. $B$ will only play a small role in our arguments. 
%
In order to lighten the notation in Theorem \ref{thm:balls}, let us define $R$ as the rectangle $[\frac15,1]\times [0,\frac13]$.
and $U =\{(\frac15,0)\}\cup \{(\frac12,0)\}\cup\{(\frac{k}{2k\pm 1},0);\, k \in \NN^{>1}\} \subset R$ (see Figure \ref{fig:U} on the left).
We also define $\widetilde{R}$ as the disjoint union of two copies of $R$ with the points corresponding to $(\frac12, 0)$, $[\frac15,\frac13]\times \{0\}$ and $(1,0)$ identified in the canonical way, and similarly for $\widetilde{U}$ (see Figure \ref{fig:U} on the right).

\begin{thm}[Infinite catastrophes in finite time]
\label{thm:balls}
There is a precompact contractible open subset of the locus $\MM_{\WMS}$ whose push along $v_0$ intersects the boundary of the stratum $\HH_{\WMS}$ in infinitely many connected components.
Explicitly, consider $W = \{M_{A, B, C}\,; \, \frac15<A<1,\, -\frac{A^2}3<B<\frac{A^2}3 ,\, -\frac13<C<\frac13\} \subset \HH_{\WMS}$, which is an open neighborhood of the surface $M_{\frac{1}{2},0,0}$ in $\MM_{\WMS}$. 
Let $\overline{\Push} = \overline{\Push(W, v_0)}\subset \overline{\HH_{\WMS}}$ be the closure of the push of $W$ along $v_0$, in the Mirzakhani-Wright partial compactification of $\HH_{\WMS}$. Then $\overline{P}$ is homeomorphic to $\widetilde{R} \times [-1,1]$.
Letting $\Push^{\bound}=\overline{\Push}\cap \left( \overline{\HH_{\WMS}} \setminus  \HH_{\WMS} \right)$, the pair $(\overline{\Push}, \Push^{\bound})$ is homeomorphic to $(\widetilde{R}\times [-1, 1], \widetilde{U}\times [-1, 1])$.
Further, $(\frac12, 0)\in \widetilde{U}$ corresponds to limiting surfaces in $\Push^{\bound}$ where a cylinder degenerates, while every other point in $\widetilde{U}$ corresponds to limiting surfaces where only a saddle connection degenerates (recall definitions from Section \ref{sec:MirWri}).
See Figure \ref{fig:U} for a picture of $\widetilde{R}$ and $\widetilde{U}$.
\end{thm}

The strategy for proving Theorem \ref{thm:balls} is to describe the translation surfaces $N_{A, B, C}$ with $A$ decreasing from 2 to $\frac15$ and with $B $ and $C$ small.
The proof of Theorem \ref{thm:balls} will take us until the end of Chapter \ref{chap:push} and will be broken down into Lemmas \ref{lem:thm1part1} and \ref{lem:thm1part2}. The former investigates the surfaces $N_{A,B,C}$ when $C\neq 0$ and the latter investigates the limits of $N_{A,B,C}$ as $C\to 0$.

\begin{rem}
\label{rem:thm1}
In order to understand the first obstruction to pushing $M_{A,B,0}$ along $v_0$, note that the push $N_{A,B, 0}$ is easily seen to exist for $A>1$, and as $A$ decreases to $1$ the edge labeled $c$ shrinks to a point, thus $N_{A,B, 0}$ escapes to the boundary of $\HH_{WMS}$ as $A\to 1^+$. See Figure \ref{fig:omega_0}.
\end{rem}

\begin{lem}
\label{lem:thm1part1}
The map $(A,B,C)\mapsto N_{A,B,C}$ is a well-defined homeomorphism from $\{(A,B,C);\, A>0,\, B\in (-\frac{A^2}2,\frac{A^2}2), \, C\in(-\frac12,0)\cup (0,\frac12)\}$ into its image in $\HH_{\WMS}$.
\end{lem}
\begin{proof}
We will first assume $C>0$, the case $C<0$ being similar. 
We start by proving that for any $(A, B, C)$ such that $A>0$, $B>-\frac12$ and $0<C<2$, the surface $N_{A,B,C}=\Push(M_{A,B,C} ,v_0)$ is well-defined, by explicitly describing that surface as a union of polygons with gluings. 
The injectivity claim then directly follows.
See Figure \ref{fig:omega_s} for a representation of $M_{A, B, C}$ with $\vec{a} = \ldots = \vec{h} = (A, C)$ and $\vec{\alpha} = \vec{\beta} = (B, \frac{1+BC}{A})$.
See Figures \ref{fig:eta_A_bigger_1}, \ref{fig:eta_A_equal_1} and \ref{fig:omega_s_plus_tv} for representations of $N_{A, B, C}$ with $C>0$ and several values of $A$. 
In Figure \ref{fig:omega_s_plus_tv}, note that $\vec{a}= (A+1, C)$, $\vec{c} = (A-1, C)$ and $\vec{b} = \vec{d} = (A, C)$, so the two light-red quadrilaterals are always well-defined and non-degenerate because $C>0$. 
Moreover, $\vec{\alpha} = \vec{\beta} = (B+\frac{1}{2}, \frac{1+BC}{A})$ and $\vec{a} + \ldots + \vec{d} = \vec{e} + \ldots+ \vec{h} = 4(A, C)$, so the two light-yellow parallelograms are always well-defined and non-degenerate because $0<C<2$ and $B>-\frac{1}{2}$. 
Indeed, these conditions together imply that $\frac{1+BC}{A}>0$ and that the slope of their left side is greater than the slope of their lower side.
This finishes the proof of Lemma \ref{lem:thm1part1}.
\end{proof}

\begin{lem}
\label{lem:thm1part2}
For any $(A,B)$ such that $\frac15<A<1$, the two limiting surfaces $N^+_{A,B} = \lim_{C\to 0^+} N_{A,B,C}$ and $N^-_{A,B} = \lim_{C\to 0^-} N_{A,B,C}$ exist and are as described in Figures \ref{fig:omega_t_final}, \ref{fig:omega_t_final2}, \ref{fig:omega_t_minus} and \ref{fig:Nhalf}. In particular, $N^+_{A,B}=N^-_{A,B}$ exactly when $\frac15\leq A\leq\frac13$, $A=\frac12$ or $A=1$ and $N^\pm_{A,B}$ hits the boundary of the stratum $\HH_{\WMS}$ exactly when $A\in \{\frac15\} \cup \{\frac12\} \cup \{\frac{k}{2k\pm 1};\, k\in \NN^{>1}\}$. 
\end{lem}
\begin{proof}
First assume $A>\frac{1}{2}$ so that $\vec{b}+\vec{c} $ points to the right. In order to obtain a more convenient representation of $N_{A,B,C}$, cut vertical slits above the endpoints of $a$ and $d$ and remember the gluings. See Figure \ref{fig:omega_s_plus_tv_test} for the resulting decomposition of $N_{A, B, C}$ into glued polygons. Our aim is to describe the limiting surface $N^+_{A, B} = \lim_{C\to 0^+} N_{A, B, C}$.
As $C\to 0$, the four convex polygons in Figure \ref{fig:omega_s_plus_tv_test}, marked (1), (3), (4) and (6), converge to polygons for the Hausdorff limit of compact subsets of $\RR^2$. In contrast, the remaining polygons, marked (2) and (5) do not converge to polygons, thus we need to perform additional cutting and gluing in order to understand the limiting surface.
After gluing the polygons (2) and (5) along $b$ (see Figure \ref{fig:omega_s_plus_tv_test_bis}), cutting vertical slits below the orange singular point and gluing the resulting polygons along $c$ (see \ref{fig:convex_quad}), we get a decomposition of $N_{A,B,C}$ that admits a limit (in the sense of Mirzakhani-Wright) as $C\to 0$. See Figure \ref{fig:omega_t_final} for the final representation of $N^+_{A,B}$ as a pair of parallelograms with gluings. Note that the lengths $l$ and $m$ verify $l = \RE(\vec{b}+\vec{c})=2A-1$ and $m = \RE(-\vec{c}) = 1-A$. From the cutting and gluing we performed, we also see that
\[
\vec{b'} = (m \mod l, 0) = (m + \lfloor-\frac{m}{l}\rfloor,0 )
\]
which degenerates exactly when $m/l$ is an integer, i.e. when $A=\frac{k+1}{2k+1}$, $k\in \NN$.

Assume now that $1/5<A<\frac{1}{2}$, as in Figure \ref{fig:omega_s_plus_tv2}. 
We cut along vertical segments above and below the endpoints of $c$ (in red in Figure \ref{fig:omega_s_plus_tv2}). Then the edge labeled $a$ is split in 3 smaller edges $a_i$ with $\vec{a_1} = \vec{a_3} = (1-2A, C)$ and $\vec{a_2}=(5A-1,C)$, and we are left with 8 polygonal pieces. The pieces delimited by edges $a_2$ or $d$ have limits as $C\to 0$. After gluing the remaining ones along $a_1$, $a_3$ and $c$, we are left with a figure similar to Figure \ref{fig:convex_quad}. 
In a similar way to the case $A>\frac{1}{2}$, we get the resulting representation of $N^+_{A,B}$ in Figure \ref{fig:omega_t_final2}, with $c'$ coming from the $a_2$ label and $a'$, $b'$ coming from a standard cut and paste. More precisely, $\vec{c'}= (5A-1,0)$ and $\vec{b'} = (A \mod 1-2A,0)$.
$b'$ is degenerated exactly when $A = \frac{k}{2k+1}$, $k>1$ integer.

The case where $C<0$ is similar, and the resulting surfaces $N^-_{A,B}$ are drawn in Figure \ref{fig:omega_t_minus}. Note that for $\frac15<A<1$, $N^+_{A,B}=N^-_{A,B}$ exactly when $A=\frac12$ or $\frac15<A<\frac13$.
\end{proof}

\begin{proof}[Proof of Theorem \ref{thm:balls}]
Note that the limits $N^\pm_{A,B}$ obtained in Lemma \ref{lem:thm1part2} are the result of approaching $(A,B,0)$ radially with an angle $\pm \frac{\pi}2$.
In order to complete the proof of Theorem \ref{thm:balls}, we only are left to show that these limits do not depend on the way we approach $(A,B,0)$.
More precisely, given any sequence $(A_n, B_n, C_n)$ with $A_n\to A\in [\frac15,1]$, $B_n\to B$, $C_n> 0$ (resp. $C_n<0$) and $C_n\to 0$, then the Mirzakhani-Wright limit of $N_{A_n, B_n, C_n}$ exists and is equal to $N^+_{A,B}$ (resp. $N^-_{A,B}$).
But this results from small changes in the proof of Lemma \ref{lem:thm1part2}, therefore the proof of Theorem \ref{thm:balls} is complete.
\end{proof}

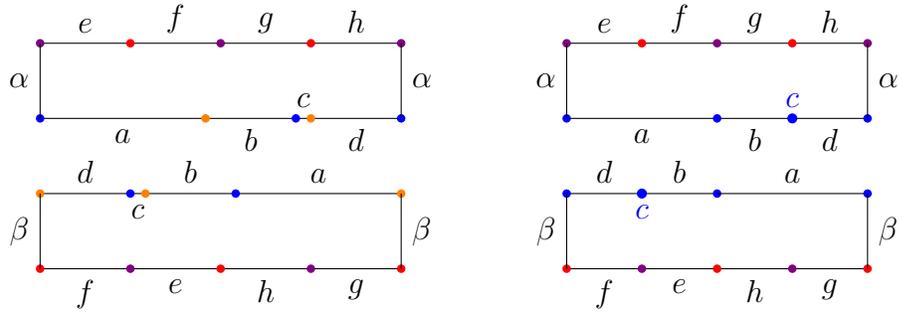
\begin{figure}
\begin{center}
\begin{tikzpicture}[scale=1]
\draw[-] (0,0)--(2.2,0) node[midway, below] {$a$};
    \fill[blue] (0,0) circle (1.6pt);
\draw[-] (2.2,0)--(3.4,0) node[midway, below] {$b$};
    \fill[orange] (2.2,0) circle (1.6pt);
\draw[-] (3.4,0)--(3.6,0) node[midway, above] {$c$};
    \fill[blue] (3.4,0) circle (1.6pt);
\draw[-] (3.6,0)--(4.8,0) node[midway, below] {$d$};
    \fill[orange] (3.6,0) circle (1.6pt);
    \fill[blue] (4.8,0) circle (1.6pt);

\draw[-] (0,1)--(1.2,1) node[midway, above] {$e$};
    \fill[violet] (0,1) circle (1.6pt);
\draw[-] (1.2,1)--(2.4,1) node[midway, above] {$f$};
    \fill[red] (1.2,1) circle (1.6pt);
\draw[-] (2.4,1)--(3.6,1) node[midway, above] {$g$};
    \fill[violet] (2.4,1) circle (1.6pt);
\draw[-] (3.6,1)--(4.8,1) node[midway, above] {$h$};
    \fill[red] (3.6,1) circle (1.6pt);
    \fill[violet] (4.8,1) circle (1.6pt);

\draw[-] (0,0)--(0,1) node[midway, left] {$\alpha$};
\draw[-] (4.8,0)--(4.8,1) node[midway, right] {$\alpha$};

\draw[-] (0,-2)--(1.2,-2) node[midway, below] {$f$};
    \fill[red] (0,-2) circle (1.6pt);
\draw[-] (1.2,-2)--(2.4,-2) node[midway, below] {$e$};
    \fill[violet] (1.2,-2) circle (1.6pt);
\draw[-] (2.4,-2)--(3.6,-2) node[midway, below] {$h$};
    \fill[red] (2.4,-2) circle (1.6pt);
\draw[-] (3.6,-2)--(4.8,-2) node[midway, below] {$g$};
    \fill[violet] (3.6,-2) circle (1.6pt);
    \fill[red] (4.8,-2) circle (1.6pt);

\draw[-] (0,-1)--(1.2,-1) node[midway, above] {$d$};
    \fill[orange] (0,-1) circle (1.6pt);
\draw[-] (1.2,-1)--(1.4,-1) node[midway, below] {$c$};
    \fill[blue] (1.2,-1) circle (1.6pt);
\draw[-] (1.4,-1)--(2.6,-1) node[midway, above] {$b$};
    \fill[orange] (1.4,-1) circle (1.6pt);
\draw[-] (2.6,-1)--(4.8,-1) node[midway, above] {$a$};
    \fill[blue] (2.6,-1) circle (1.6pt);
    \fill[orange] (4.8,-1) circle (1.6pt);

\draw[-] (0,-2)--(0,-1) node[midway, left] {$\beta$};
\draw[-] (4.8,-2)--(4.8,-1) node[midway, right] {$\beta$};

\draw[-] (7,0)--(9,0) node[midway, below] {$a$};
    \fill[blue] (7,0) circle (1.6pt);
\draw[-] (9,0)--(10,0) node[midway, below] {$b$};
    \fill[blue] (9,0) circle (1.6pt);
\draw[-] (10,0)--(11,0) node[midway, below] {$d$};
    \fill[blue] (10,0) circle (1.9pt) node[above, blue] {$c$};
    \fill[blue] (11,0) circle (1.6pt);

\draw[-] (7,1)--(8,1) node[midway, above] {$e$};
    \fill[violet] (7,1) circle (1.6pt);
\draw[-] (8,1)--(9,1) node[midway, above] {$f$};
    \fill[red] (8,1) circle (1.6pt);
\draw[-] (9,1)--(10,1) node[midway, above] {$g$};
    \fill[violet] (9,1) circle (1.6pt);
\draw[-] (10,1)--(11,1) node[midway, above] {$h$};
    \fill[red] (10,1) circle (1.6pt);
    \fill[violet] (11,1) circle (1.6pt);

\draw[-] (7,0)--(7,1) node[midway, left] {$\alpha$};
\draw[-] (11,0)--(11,1) node[midway, right] {$\alpha$};

\draw[-] (7,-2)--(8,-2) node[midway, below] {$f$};
    \fill[red] (7,-2) circle (1.6pt);
\draw[-] (8,-2)--(9,-2) node[midway, below] {$e$};
    \fill[violet] (8,-2) circle (1.6pt);
\draw[-] (9,-2)--(10,-2) node[midway, below] {$h$};
    \fill[red] (9,-2) circle (1.6pt);
\draw[-] (10,-2)--(11,-2) node[midway, below] {$g$};
    \fill[violet] (10,-2) circle (1.6pt);
    \fill[red] (11,-2) circle (1.6pt);

\draw[-] (7,-1)--(8,-1) node[midway, above] {$d$};
    \fill[blue] (7,-1) circle (1.6pt);
\draw[-] (8,-1)--(9,-1) node[midway, above] {$b$};
    \fill[blue] (9,-1) circle (1.6pt);
\draw[-] (9,-1)--(11,-1) node[midway, above] {$a$};
    \fill[blue] (8,-1) circle (1.9pt) node[below, blue] {$c$};
    \fill[blue] (11,-1) circle (1.6pt);

\draw[-] (7,-2)--(7,-1) node[midway, left] {$\beta$};
\draw[-] (11,-2)--(11,-1) node[midway, right] {$\beta$};

\end{tikzpicture}
\end{center}
\caption{Left: the surface $N_{1+\epsilon, 0, 0}$ with $\epsilon >0$ (actually here $\epsilon = 0.2$), with $\vec{a} = (2+\epsilon,0)$, $\vec{c} = (\epsilon,0)$, $\vec{\alpha} = \vec{\beta} = (\frac12, \frac1{1+\epsilon})$ and $\vec{b}=\vec{d}=\vec{e} = \ldots = \vec{h}=(1+\epsilon,0)$. Note that as $\epsilon \to 0$, the length of the segment $c$ tends to 0, so that $\lim_{\epsilon \to 0} N_{1+\epsilon,0,0} = \Push (M_{\WMS}, v_0)$ is not well-defined as a surface in $\HH_{\WMS}$. Right: the result of collapsing $c$ to a point is shown in blue. The Mirzakhani-Wright limit $N_{1,0,0}$ exists and lies in $\HH_1(1,1,2)$.}
\label{fig:omega_0}
\end{figure}


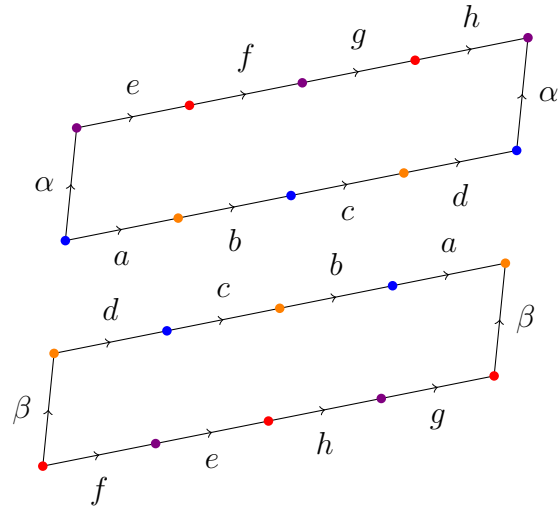
\begin{figure}
\begin{center}
\begin{tikzpicture}[scale=1.5]

\foreach \i in {0,...,3}{%
    \foreach \j in {0,...,3}{%
        \draw[middlearrow={>}] (\i+.1*\j,\j+.2*\i)--(\i+1+.1*\j,\j+.2*\i+.2) ;}}
\foreach \j in {0,2}{%
    \foreach \i in {0,4}{%
        \draw[middlearrow={>}] (\i+.1*\j,\j+0.2*\i)--(\i+.1*\j+.1,\j+1+.2*\i) ;}}

\node[left] at (0.2,2.5) {$\alpha$};
\node[right] at (4.3,3.3) {$\alpha$};
\node[left] at (0,0.5) {$\beta$};
\node[right] at (4.1,1.3) {$\beta$};

\node[above] at (0.8,3.2) {$e$};
    \fill[violet] (.3,3) circle (1.2pt);
\node[above] at (1.8,3.4) {$f$};
    \fill[red] (1.3,3.2) circle (1.2pt);
\node[above] at (2.8,3.6) {$g$};
    \fill[violet] (2.3,3.4) circle (1.2pt);
\node[above] at (3.8,3.8) {$h$};
    \fill[red] (3.3,3.6) circle (1.2pt);
    \fill[violet] (4.3,3.8) circle (1.2pt);

\node[below] at (0.7,2) {$a$};
    \fill[blue] (.2,2) circle (1.2pt);
\node[below] at (1.7,2.2) {$b$};
    \fill[orange] (1.2,2.2) circle (1.2pt);
\node[below] at (2.7,2.4) {$c$};
    \fill[blue] (2.2,2.4) circle (1.2pt);
\node[below] at (3.7,2.6) {$d$};
    \fill[orange] (3.2,2.6) circle (1.2pt);
    \fill[blue] (4.2,2.8) circle (1.2pt);

\node[above] at (0.6,1.2) {$d$};
    \fill[orange] (.1,1) circle (1.2pt);
\node[above] at (1.6,1.4) {$c$};
    \fill[blue] (1.1,1.2) circle (1.2pt);
\node[above] at (2.6,1.6) {$b$};
    \fill[orange] (2.1,1.4) circle (1.2pt);
\node[above] at (3.6,1.8) {$a$};
    \fill[blue] (3.1,1.6) circle (1.2pt);
    \fill[orange] (4.1,1.8) circle (1.2pt);

\node[below] at (0.5,0) {$f$};
    \fill[red] (0,0) circle (1.2pt);
\node[below] at (1.5,0.2) {$e$};
    \fill[violet] (1,.2) circle (1.2pt);
\node[below] at (2.5,0.4) {$h$};
    \fill[red] (2,.4) circle (1.2pt);
\node[below] at (3.5,0.6) {$g$};
    \fill[violet] (3,.6) circle (1.2pt);
    \fill[red] (4,.8) circle (1.2pt);

\end{tikzpicture}
\end{center}
\caption{The surface $M_{A, B, C}$. Here $A=1.2$ and $B=C=0.1$.}
\label{fig:omega_s}
\end{figure}


\begin{figure}
 \begin{center}
\begin{tikzpicture}[scale= 2]
\draw[-] (0,1)--(1.8,1.1)--(2.8,1.2)--(3,1.3)--(4,1.4)--(4.6,2.4)--(0.6,2)--(0,1);
  \fill[blue] (0,1) circle (0.9pt);
  \fill[orange] (1.8,1.1) circle (0.9pt);
  \fill[blue] (2.8,1.2) circle (0.9pt);
  \fill[orange] (3,1.3) circle (0.9pt);
  \fill[blue] (4,1.4) circle (0.9pt);

\node at (.9,1) {$a$};
\node at (2.5,1.25) {$b$};
\node at (2.9,1.35) {$c$};
\node at (3.5,1.45) {$d$};

\draw[-] (0,0)--(1,0.1)--(1.2,0.2)--(2.2,0.3)--(4,0.4)--(3.4,-0.6)--(-0.6,-1)--(0,0);
  \fill[orange] (0,0) circle (0.9pt);
  \fill[blue] (1,0.1) circle (0.9pt);
  \fill[orange] (1.2,0.2) circle (0.9pt);
  \fill[blue] (2.2,0.3) circle (0.9pt);
  \fill[orange] (4,0.4) circle (0.9pt);

\node at (0.5,0) {$d$};
\node at (1.1,0.1) {$c$};
\node at (1.7,0.2) {$b$};
\node at (2.9,0.3) {$a$};

\node at (1.1,2.1) {$e$};
\node at (2.1,2.2) {$f$};
\node at (3.1,2.3) {$g$};
\node at (4.1,2.4) {$h$};

\node at (-.1,-1) {$f$};
\node at (.9,-.9) {$e$};
\node at (1.9,-.8) {$h$};
\node at (2.9,-.7) {$g$};


  \fill[violet] (.6,2) circle (0.9pt);
  \fill[red] (1.6,2.1) circle (0.9pt);
  \fill[violet] (2.6,2.2) circle (0.9pt);
  \fill[red] (3.6,2.3) circle (0.9pt);
  \fill[violet] (4.6,2.4) circle (0.9pt);

  \fill[red] (-.6,-1) circle (0.9pt);
  \fill[violet] (.4,-.9) circle (0.9pt);
  \fill[red] (1.4,-.8) circle (0.9pt);
  \fill[violet] (2.4,-.7) circle (0.9pt);
  \fill[red] (3.4,-.6) circle (0.9pt);

\node at (0.2,1.5) {$\alpha$};
\node at (4.4,1.9) {$\alpha$};
\node at (-0.4,-0.5) {$\beta$};
\node at (3.8,-.1) {$\beta$};

\end{tikzpicture}
\end{center}
\caption{The surface $N_{A, B, C}$ with $A>1$. Here $A=1.2$ and $B=C=0.1$.}
\label{fig:eta_A_bigger_1}
\end{figure}
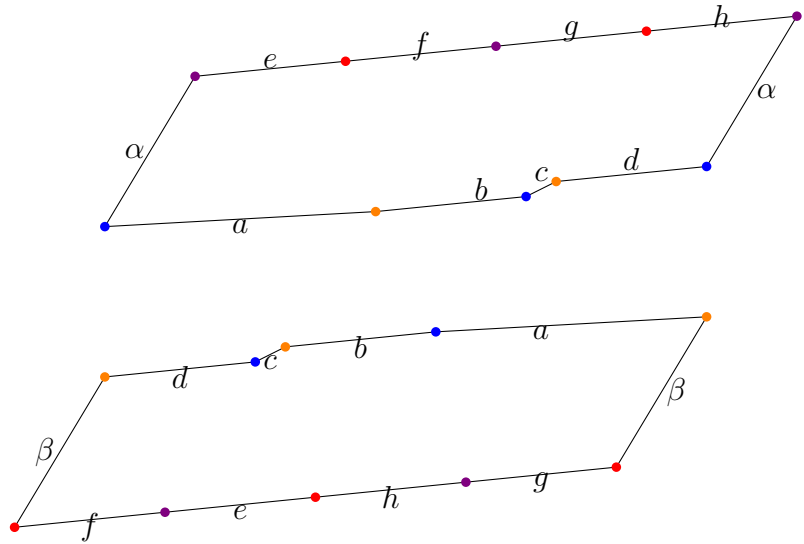


\begin{figure}
 \begin{center}
\begin{tikzpicture}[scale= 2]
\draw[-] (0,1)--(2,1.1)--(3,1.2)--(3,1.3)--(4,1.4)--(4.6,2.4)--(0.6,2)--(0,1);
  \fill[blue] (0,1) circle (0.9pt);
  \fill[orange] (2,1.1) circle (0.9pt);
  \fill[blue] (3,1.2) circle (0.9pt);
  \fill[orange] (3,1.3) circle (0.9pt);
  \fill[blue] (4,1.4) circle (0.9pt);

\node at (1,1) {$a$};
\node at (2.5,1.25) {$b$};
\node at (3.07,1.25) {$c$};
\node at (3.5,1.45) {$d$};

\draw[-] (0,0)--(1,0.1)--(1,0.2)--(2,0.3)--(4,0.4)--(3.4,-0.6)--(-0.6,-1)--(0,0);
  \fill[orange] (0,0) circle (0.9pt);
  \fill[blue] (1,0.1) circle (0.9pt);
  \fill[orange] (1,0.2) circle (0.9pt);
  \fill[blue] (2,0.3) circle (0.9pt);
  \fill[orange] (4,0.4) circle (0.9pt);

\node at (0.5,0) {$d$};
\node at (1.07,0.15) {$c$};
\node at (1.5,0.2) {$b$};
\node at (3,0.3) {$a$};

\node at (1.1,2.1) {$e$};
\node at (2.1,2.2) {$f$};
\node at (3.1,2.3) {$g$};
\node at (4.1,2.4) {$h$};

\node at (-.1,-1) {$f$};
\node at (.9,-.9) {$e$};
\node at (1.9,-.8) {$h$};
\node at (2.9,-.7) {$g$};


  \fill[violet] (.6,2) circle (0.9pt);
  \fill[red] (1.6,2.1) circle (0.9pt);
  \fill[violet] (2.6,2.2) circle (0.9pt);
  \fill[red] (3.6,2.3) circle (0.9pt);
  \fill[violet] (4.6,2.4) circle (0.9pt);

  \fill[red] (-.6,-1) circle (0.9pt);
  \fill[violet] (.4,-.9) circle (0.9pt);
  \fill[red] (1.4,-.8) circle (0.9pt);
  \fill[violet] (2.4,-.7) circle (0.9pt);
  \fill[red] (3.4,-.6) circle (0.9pt);

\node at (0.2,1.5) {$\alpha$};
\node at (4.4,1.9) {$\alpha$};
\node at (-0.4,-0.5) {$\beta$};
\node at (3.8,-.1) {$\beta$};

\end{tikzpicture}
\end{center}
\caption{The surface $N_{1, B, C}$ is well-defined when $C>0$, thus we can ``push past that point'', i.e. $N_{A, B, C}$ is well-defined for any $A>0$. Here $A=1$, see also Figure \ref{fig:omega_s_plus_tv} for the case $A<1$. Compare with Figure \ref{fig:omega_0} (case $C=0$) where the edge labeled $c$ degenerates as $A$ approaches 1.}
\label{fig:eta_A_equal_1}
\end{figure}
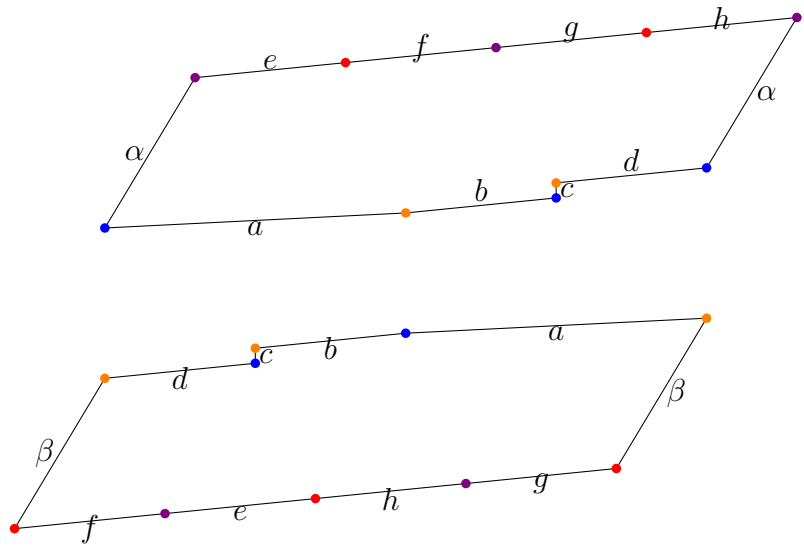


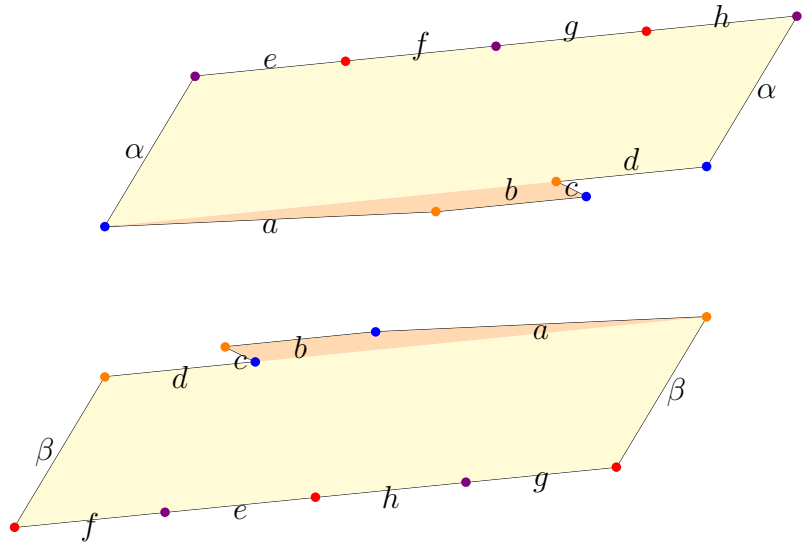
\begin{figure}
 \begin{center}
\begin{tikzpicture}[scale= 2]
\draw[-] (0,1)--(2.2,1.1)--(3.2,1.2)--(3,1.3)--(4,1.4)--(4.6,2.4)--(0.6,2)--(0,1);
\fill[yellow!20!white] (0,1)--(4,1.4)--(4.6,2.4)--(0.6,2)--(0,1);
\fill[orange!30!white] (0,1)--(2.2,1.1)--(3.2,1.2)--(3,1.3)--(0,1);
  \fill[blue] (0,1) circle (0.9pt);
  \fill[orange] (2.2,1.1) circle (0.9pt);
  \fill[blue] (3.2,1.2) circle (0.9pt);
  \fill[orange] (3,1.3) circle (0.9pt);
  \fill[blue] (4,1.4) circle (0.9pt);

\node at (1.1,1) {$a$};
\node at (2.7,1.25) {$b$};
\node at (3.1,1.25) {$c$};
\node at (3.5,1.45) {$d$};

\draw[-] (0,0)--(1,0.1)--(0.8,0.2)--(1.8,0.3)--(4,0.4)--(3.4,-0.6)--(-0.6,-1)--(0,0);
\fill[yellow!20!white] (0,0)--(4,0.4)--(3.4,-.6)--(-0.6,-1)--(0,0);
\fill[orange!30!white] (1,0.1)--(0.8,0.2)--(1.8,0.3)--(4,0.4)--(1,.1);
  \fill[orange] (0,0) circle (0.9pt);
  \fill[blue] (1,0.1) circle (0.9pt);
  \fill[orange] (0.8,0.2) circle (0.9pt);
  \fill[blue] (1.8,0.3) circle (0.9pt);
  \fill[orange] (4,0.4) circle (0.9pt);

\node at (0.5,0) {$d$};
\node at (0.9,0.1) {$c$};
\node at (1.3,0.2) {$b$};
\node at (2.9,0.3) {$a$};

\node at (1.1,2.1) {$e$};
\node at (2.1,2.2) {$f$};
\node at (3.1,2.3) {$g$};
\node at (4.1,2.4) {$h$};

\node at (-.1,-1) {$f$};
\node at (.9,-.9) {$e$};
\node at (1.9,-.8) {$h$};
\node at (2.9,-.7) {$g$};


  \fill[violet] (.6,2) circle (0.9pt);
  \fill[red] (1.6,2.1) circle (0.9pt);
  \fill[violet] (2.6,2.2) circle (0.9pt);
  \fill[red] (3.6,2.3) circle (0.9pt);
  \fill[violet] (4.6,2.4) circle (0.9pt);

  \fill[red] (-.6,-1) circle (0.9pt);
  \fill[violet] (.4,-.9) circle (0.9pt);
  \fill[red] (1.4,-.8) circle (0.9pt);
  \fill[violet] (2.4,-.7) circle (0.9pt);
  \fill[red] (3.4,-.6) circle (0.9pt);

\node at (0.2,1.5) {$\alpha$};
\node at (4.4,1.9) {$\alpha$};
\node at (-0.4,-0.5) {$\beta$};
\node at (3.8,-.1) {$\beta$};

\end{tikzpicture}
\end{center}
\caption{The surface $N_{A, B, C}$ with $C>0$ and $A<1$. Recall that $\vec{a}= (A+1, C)$, $\vec{c} = (A-1, C)$, $\vec{\alpha} = \vec{\beta} = (B+\frac{1}{2}, \frac{1+BC}{A})$ and $\vec{b} = \vec{d} = \vec{e} = \ldots = \vec{h} = (A,C)$. Computations in Theorem \ref{lem:thm1part1} show that this representation makes sense for any $A>0$ as long as $B>-\frac12$ and $0<C<2$.}
\label{fig:omega_s_plus_tv}
\end{figure}


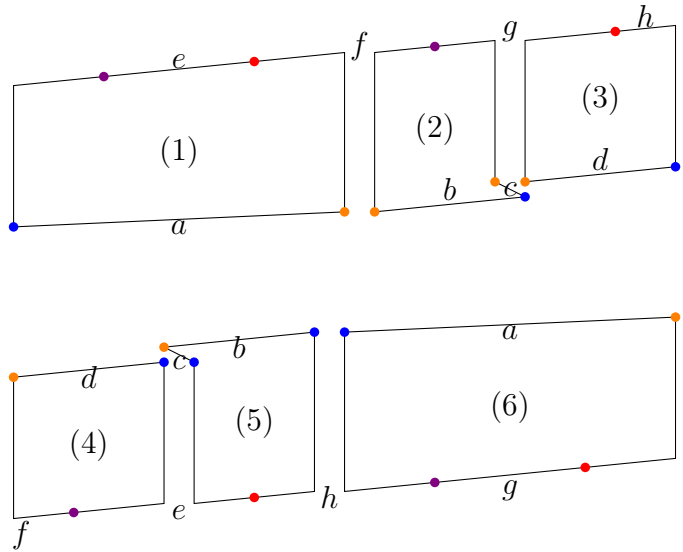
\begin{figure} 
 \begin{center} 
\begin{tikzpicture}[scale= 2] 
\draw[-] (-0.2,1)--(2,1.1)--(2,2.16)--(-0.2,1.94)--(-0.2,1); 
    \node at (.9,1.5) {(1)};
\draw[-] (2.2,1.1)--(3.2,1.2)--(3,1.3)--(3,2.24)--(2.2,2.16)--(2.2,1.1); 
    \node at (2.6,1.65) {(2)};
\draw[-] (3.2,1.3)--(4.2,1.4)--(4.2,2.34)--(3.2,2.24)--(3.2,1.3); 
    \node at (3.7,1.85) {(3)};
  \fill[blue] (-.2,1) circle (0.9pt); 
  \fill[orange] (2,1.1) circle (0.9pt); 
  \fill[orange] (2.2,1.1) circle (0.9pt); 
  \fill[blue] (3.2,1.2) circle (0.9pt); 
  \fill[orange] (3,1.3) circle (0.9pt); 
  \fill[orange] (3.2,1.3) circle (0.9pt); 
  \fill[blue] (4.2,1.4) circle (0.9pt); 
  
  \fill[violet] (.4,2) circle (0.9pt); 
  \fill[red] (1.4,2.1) circle (0.9pt); 
  \fill[violet] (2.6,2.2) circle (0.9pt); 
  \fill[red] (3.8,2.3) circle (0.9pt); 
 
\node at (.9,1) {$a$}; 
\node at (2.7,1.25) {$b$}; 
\node at (3.1,1.25) {$c$}; 
\node at (3.7,1.45) {$d$}; 
 
\draw[-] (-.2,0)--(.8,0.1)--(.8,-.84)--(-.2,-.94)--(-.2,0); 
    \node at (.3,-.45) {(4)};
\draw[-] (1,0.1)--(0.8,0.2)--(1.8,0.3)--(1.8,-.76)--(1,-.84)--(1,0.1); 
    \node at (1.4,-.3) {(5)};
\draw[-] (2,0.3)--(4.2,0.4)--(4.2,-.54)--(2,-.76)--(2,0.3); 
    \node at (3.1,-.2) {(6)};
  \fill[orange] (-.2,0) circle (0.9pt); 
  \fill[blue] (.8,0.1) circle (0.9pt); 
  \fill[blue] (1,0.1) circle (0.9pt); 
  \fill[orange] (0.8,0.2) circle (0.9pt); 
  \fill[blue] (1.8,0.3) circle (0.9pt); 
  \fill[blue] (2,0.3) circle (0.9pt); 
  \fill[orange] (4.2,0.4) circle (0.9pt); 

  \fill[violet] (.2,-.9) circle (0.9pt); 
  \fill[red] (1.4,-.8) circle (0.9pt); 
  \fill[violet] (2.6,-.7) circle (0.9pt); 
  \fill[red] (3.6,-.6) circle (0.9pt); 

\node at (0.3,0) {$d$};
\node at (0.9,0.1) {$c$};
\node at (1.3,0.2) {$b$};
\node at (3.1,0.3) {$a$};

\node at (.9,2.1) {$e$};
\node at (2.1,2.2) {$f$};
\node at (3.1,2.3) {$g$};
\node at (4,2.4) {$h$};

\node at (-.15,-1.05) {$f$};
\node at (.9,-.9) {$e$};
\node at (1.9,-.8) {$h$};
\node at (3.1,-.75) {$g$};

\end{tikzpicture}
\end{center}
\caption{Another representation of $N_{A, B, C}$ for $\frac{1}{2}<A<1$, $B>-\frac{1}{2}$ and $0<C<2$. It was obtained from Figure \ref{fig:omega_s_plus_tv} by cutting vertical slits above the endpoints of $a$ and $d$.}
\label{fig:omega_s_plus_tv_test}
\end{figure}


\begin{figure} 
 \begin{center} 
\begin{tikzpicture}[scale= 2] 
\draw[-] (-0.2,1)--(2,1.1)--(2,2.16)--(-0.2,1.94)--(-0.2,1); 
    \node at (.9,1.5) {(1)};
\draw[-] (3.2,.75)--(3,.85)--(3,1.79)--(2.2,1.71)--(2.2,.65); 
    \node at (2.6,1.25) {(2)};
\draw[-] (3.9,1.3)--(4.9,1.4)--(4.9,2.34)--(3.9,2.24)--(3.9,1.3); 
    \node at (4.4,1.75) {(3)};

\node at (.9,1) {$a$}; 
\node at (2.7,.8) {$b$}; 
\node at (3.13,.85) {$c$}; 
\node at (4.4,1.45) {$d$}; 
 
\draw[-] (.3,0)--(1.3,0.1)--(1.3,-.84)--(.3,-.94)--(.3,0); 
    \node at (.8,-.45) {(4)};
\draw[-] (2.4,0.55)--(2.2,0.65)--(3.2,0.75)--(3.2,-.31)--(2.4,-.39)--(2.4,0.55); 
    \node at (2.8,.2) {(5)};
\draw[-] (3.4,0.3)--(5.6,0.4)--(5.6,-.54)--(3.4,-.76)--(3.4,0.3); 
    \node at (4.5,-.2) {(6)};
  \fill[blue] (-.2,1) circle (0.8pt); 
  \fill[orange] (2,1.1) circle (0.8pt); 
  \fill[orange] (2.2,.65) circle (0.8pt); 
  \fill[blue] (3.2,.75) circle (0.8pt); 
  \fill[orange] (3,.85) circle (0.8pt); 
  \fill[orange] (3.9,1.3) circle (0.8pt); 
  \fill[blue] (4.9,1.4) circle (0.8pt); 

  \fill[violet] (.4,2) circle (0.8pt); 
  \fill[red] (1.4,2.1) circle (0.8pt); 
  \fill[violet] (2.6,1.75) circle (0.8pt); 
  \fill[red] (4.5,2.3) circle (0.8pt); 
 
  \fill[orange] (.3,0) circle (0.8pt); 
  \fill[blue] (1.3,0.1) circle (0.8pt); 
  \fill[blue] (2.4,0.55) circle (0.8pt); 
  \fill[blue] (3.4,0.3) circle (0.8pt); 
  \fill[orange] (5.6,0.4) circle (0.8pt); 

  \fill[violet] (.9,-.9) circle (0.8pt); 
  \fill[red] (2.8,-.35) circle (0.8pt); 
  \fill[violet] (4,-.7) circle (0.8pt); 
  \fill[red] (5,-.6) circle (0.8pt); 

\node at (0.8,.13) {$d$};
\node at (2.27,0.51) {$c$};
\node at (4.5,0.4) {$a$};



\end{tikzpicture}
\end{center}
\caption{Representation of $N_{A, B, C}$ for $\frac{1}{2}<A<1$, $B>-\frac{1}{2}$ and $0<C<2$, obtained from Figure \ref{fig:omega_s_plus_tv_test} by gluing the center pieces along $b$. Labels $e$ through $h$ were omitted. Note that as $A \to \frac{1}{2}$, $|b|/|c| \to 1$. See left of Figure \ref{fig:convex_quad} for a sketch of the center piece with $A$ closer to $\frac{1}{2}$.}
\label{fig:omega_s_plus_tv_test_bis}
\end{figure}
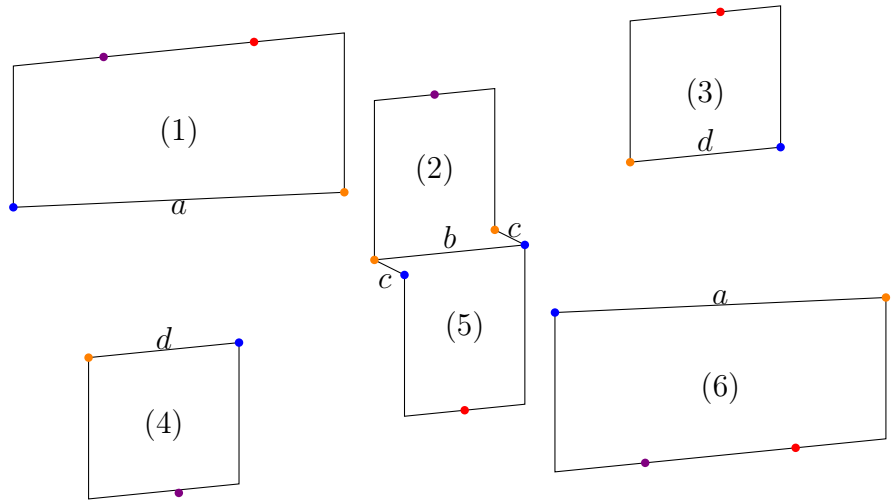


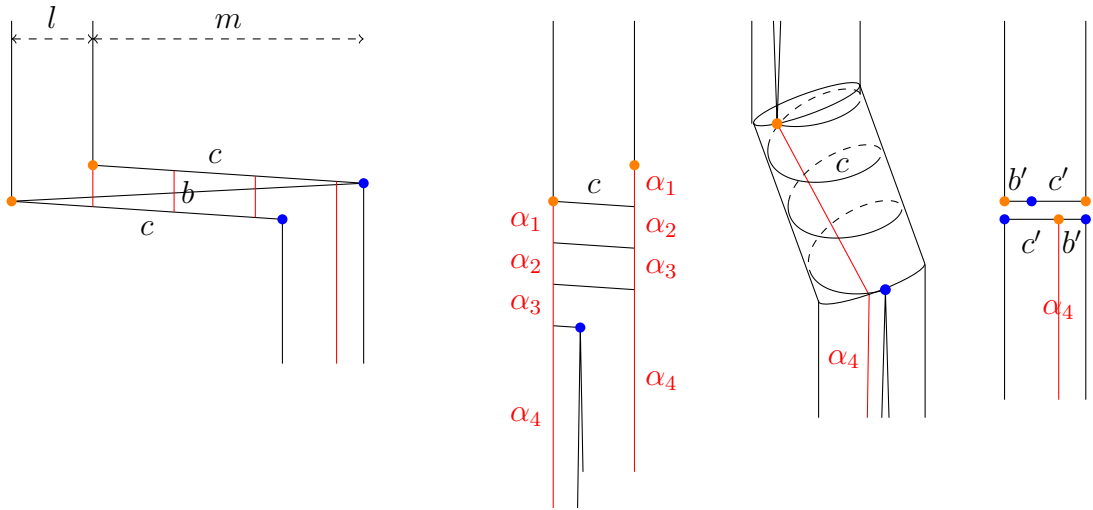
\begin{figure}
 \begin{center}
\begin{tikzpicture}[scale= .24]

\draw[-] (0,10)--(0,0)--(15,-1)--(15,-9);
\draw[-] (4.5,10)--(4.5,2)--(19.5,1)--(19.5,-9);
\draw[-] (0,0)--(19.5,1);
\draw[-,red] (4.5,2)--(4.5,-.3);
\draw[-,red] (9,1.7)--(9,-.6);
\draw[-,red] (13.5,1.4)--(13.5,-.9);
\draw[-,red] (18,1.1)--(18,-9);
\fill[orange] (0,0) circle (.28);
\fill[blue] (15,-1) circle (.28);
\fill[orange] (4.5,2) circle (.28);
\fill[blue] (19.5,1) circle (.28);
\node[above] at (11.25,1.5) {$c$};
\node[below] at (7.5,-.5) {$c$};
\node at (9.75,.5) {$b$};
\draw[<->,dashed] (0,9)--(4.5,9);
\node[above] at (2.25,9) {$l$};
\draw[<->,dashed] (4.5,9)--(19.5,9);
\node[above] at (12,9) {$m$};

\draw[-] (30,10)--(30,0);
\draw[-] (34.5,10)--(34.5,2);
\draw[-] (31.35,-17)--(31.5,-7)--(31.65,-15);
\draw[-] (30,0)--(34.5,-.3) node [above, midway] {$c$};
\draw[-] (30,-2.3)--(34.5,-2.6);
\draw[-] (30,-4.6)--(34.5,-4.9);
\draw[-] (30,-6.9)--(31.5,-7);
\draw[-,red] (30,0)--(30,-17);
\draw[-,red] (34.5,2)--(34.5,-15);
\fill[orange] (30,0) circle (.28);
\fill[orange] (34.5,2) circle (.28);
\fill[blue] (31.5,-7) circle (.28);
\node[left, red] at (30,-1.15) {$\alpha_1$};
\node[left, red] at (30,-3.45) {$\alpha_2$};
\node[left, red] at (30,-5.75) {$\alpha_3$};
\node[left, red] at (30,-11.95) {$\alpha_4$};
\node[right, red] at (34.5,.85) {$\alpha_1$};
\node[right, red] at (34.5,-1.45) {$\alpha_2$};
\node[right, red] at (34.5,-3.75) {$\alpha_3$};
\node[right, red] at (34.5,-9.95) {$\alpha_4$};

\draw[-] (42.2, 10)--(42.4, 4.3)--(42.6, 10);
\draw[-] (41, 10)--(41,4.3);
\draw[-] (47, 10)--(47, 5.9);

\draw[-] (48.2, -12)--(48.4, -5)--(48.6, -12);
\draw[-] (44.7, -12)--(44.7,-5.5);
\draw[-] (50.6, -12)--(50.6, -3.5);

\draw[-, red] (42.4, 4.3)--(47.5,-5.2)--(47.4, -12);
\node[red] at (46.1, -8.75) {$\alpha_4$};

\begin{scope} [shift={(46,0)}, tdplot_main_coords]

\node [cylinder,rotate = 110,draw,minimum width=1.5cm, minimum height=2.8cm](C){};

 \begin{scope}[rotate = 20]
  \pgfplothandlerlineto
  \pgfplotfunction{\t}{-570,-569,...,-480}
       {\pgfpointxyz {-3 * cos(\t)}{3 * sin(\t)}{1.45 + 3.8 * \t/360}} 
       \pgfusepath{stroke}
  \end{scope}

 \begin{scope}[rotate = 20, dashed]
  \pgfplothandlerlineto
  \pgfplotfunction{\t}{-480,-479,...,-270}
       {\pgfpointxyz {-3 * cos(\t)}{3 * sin(\t)}{1.45 + 3.8 * \t/360}} 
       \pgfusepath{stroke}
  \end{scope}

 \begin{scope}[rotate = 20]
  \pgfplothandlerlineto
  \pgfplotfunction{\t}{-270,-269,...,-90}
       {\pgfpointxyz {-3 * cos(\t)}{3 * sin(\t)}{1.45 + 3.8 * \t/360}} 
       \pgfusepath{stroke}
  \end{scope}

 \begin{scope}[rotate = 20, dashed]
  \pgfplothandlerlineto
  \pgfplotfunction{\t}{-90,-89,...,90}
       {\pgfpointxyz {-3 * cos(\t)}{3 * sin(\t)}{1.45 + 3.8 * \t/360}} 
       \pgfusepath{stroke}
  \end{scope}

 \begin{scope}[rotate = 20]
  \pgfplothandlerlineto
  \pgfplotfunction{\t}{90,91,...,270}
       {\pgfpointxyz {-3 * cos(\t)}{3 * sin(\t)}{1.45 + 3.8 * \t/360}} 
       \pgfusepath{stroke}
  \end{scope}

 \begin{scope}[rotate = 20, dashed]
  \pgfplothandlerlineto
  \pgfplotfunction{\t}{270,271,...,450}
       {\pgfpointxyz {-3 * cos(\t)}{3 * sin(\t)}{1.45 + 3.8 * \t/360}} 
       \pgfusepath{stroke}
  \end{scope}

 \begin{scope}[rotate = 20]
  \pgfplothandlerlineto
  \pgfplotfunction{\t}{450,451,...,560}
       {\pgfpointxyz {-3 * cos(\t)}{3 * sin(\t)}{1.45 + 3.8 * \t/360}} 
       \pgfusepath{stroke}
  \end{scope}
\end{scope}

\fill[orange] (42.4, 4.3) circle (0.3);
\fill[blue] (48.4, -4.9) circle (0.3);
\node at (46, 2) {$c$};

\begin{scope} [shift={(10,0)}]
\draw[-] (45,10)--(45,0)--(49.5,0)--(49.5,10);
\draw[-] (45,-11)--(45,-1)--(49.5,-1)--(49.5,-11);
\draw[-, red] (48,-11)--(48,-1);
\fill[orange] (45,0) circle (.28);
\fill[blue] (46.5,0) circle (.28);
\fill[orange] (49.5,0) circle (.28);
\fill[blue] (45,-1) circle (.28);
\fill[orange] (48,-1) circle (.28);
\fill[blue] (49.5,-1) circle (.28);
\node[red] at (48,-6) {$\alpha_4$};
\node[above] at (45.75,0) {$b'$};
\node[above] at (48,0) {$c'$};
\node[below] at (46.5,-1) {$c'$};
\node[below] at (48.75,-1) {$b'$};
\end{scope} 

\end{tikzpicture}
\end{center}
\caption{The leftmost picture is a sketch of the center piece of Figure \ref{fig:omega_s_plus_tv_test_bis}. Cutting vertical slits below the left endpoint of $c$ and regluing along $c$ yields the second picture. Gluing along $\alpha_i$, we get the representation in the third picture. Letting $C\to 0$, the edges $\alpha_1$, $\alpha_2$ and $\alpha_3$ are reduced to a point, as is shown in the right picture, with $|b'| = m \mod l$.}
\label{fig:convex_quad}
\end{figure}


\begin{figure}
 \begin{center}
\begin{tikzpicture}[scale= 2]
\draw[-] (0,1)--(4,1)--(4.6,2)--(.6,2)--(0,1);

  \fill[blue] (0,1) circle (0.9pt);
  \fill[orange] (2.2,1) circle (0.9pt);
  \fill[blue] (2.8,1) circle (0.9pt);
  \fill[orange] (3,1) circle (0.9pt);
  \fill[blue] (4,1) circle (0.9pt);

\node at (1.1,.9) {$a$};
\node at (2.5,.9) {$b'$};
\node at (2.9,.9) {$c'$};
\node at (3.5,.9) {$d$};
\node[left] at (.3,1.5) {$\alpha$};

  \fill[violet] (0.6,2) circle (0.9pt);
  \fill[red] (1.6,2) circle (0.9pt);
  \fill[violet] (2.6,2) circle (0.9pt);
  \fill[red] (3.6,2) circle (0.9pt);
  \fill[violet] (4.6,2) circle (0.9pt);

\node at (1.1,2.1) {$e$};
\node at (2.1,2.1) {$f$};
\node at (3.1,2.1) {$g$};
\node at (4.1,2.1) {$h$};

\draw[-] (0,.4)--(4,0.4)--(3.4,-.6)--(-.6,-.6)--(0,.4);

  \fill[orange] (0,.4) circle (0.9pt);
  \fill[blue] (1,0.4) circle (0.9pt);
  \fill[orange] (1.2,0.4) circle (0.9pt);
  \fill[blue] (1.8,0.4) circle (0.9pt);
  \fill[orange] (4,0.4) circle (0.9pt);

\node at (0.5,0.5) {$d$};
\node at (1.1,0.5) {$c'$};
\node at (1.5,0.5) {$b'$};
\node at (2.9,0.5) {$a$};
\node[left] at (-.3,-.1) {$\beta$};

  \fill[red] (-.6,-.6) circle (0.9pt);
  \fill[violet] (.4,-.6) circle (0.9pt);
  \fill[red] (1.4,-.6) circle (0.9pt);
  \fill[violet] (2.4,-.6) circle (0.9pt);
  \fill[red] (3.4,-.6) circle (0.9pt);

\node at (-.1,-.7) {$f$};
\node at (.9,-.7) {$e$};
\node at (1.9,-.7) {$h$};
\node at (2.9,-.7) {$g$};

\end{tikzpicture}
\end{center}
\caption{Final representation of $N^+_{A,B}$ for $\frac{1}{2}<A<1$, with $\vec{\alpha} = \vec{\beta} = (B+\frac12, \frac1A)$, $|\vec{b'}| = 1-A \mod 2A-1$ and $|\vec{c'}| = 2A-1 - |\vec{b'}|$. It was obtained by combining Figures \ref{fig:omega_s_plus_tv_test} and \ref{fig:convex_quad}.}
\label{fig:omega_t_final}
\end{figure}
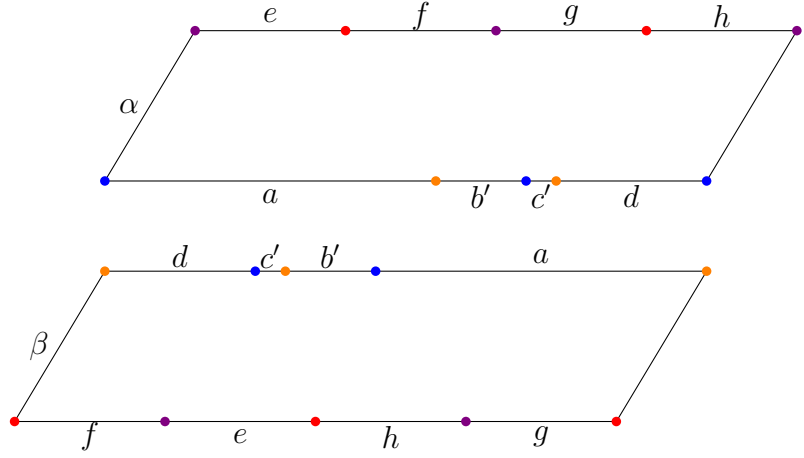

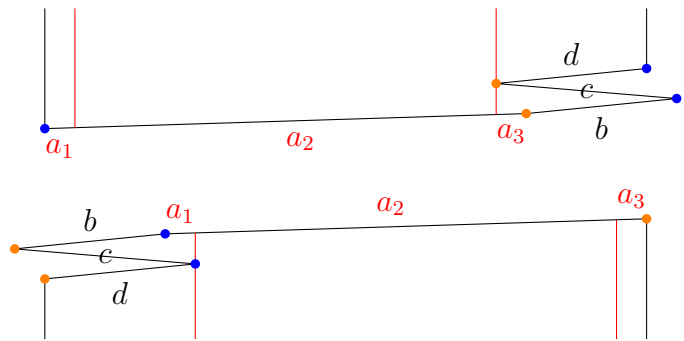
\begin{figure}
\begin{center}
\begin{tikzpicture}[scale= 2]
\draw[-] (0,1.8)--(0,1)--(3.2,1.1)--(4.2,1.2)--(3,1.3)--(4,1.4)--(4,1.8);
\draw[-, red] (.2, 1.8)--(.2, 1.00625);
\draw[-, red] (3, 1.8)--(3, 1.09375);

  \fill[blue] (0,1) circle (0.9pt);
  \fill[orange] (3.2,1.1) circle (0.9pt);
  \fill[blue] (4.2,1.2) circle (0.9pt);
  \fill[orange] (3,1.3) circle (0.9pt);
  \fill[blue] (4,1.4) circle (0.9pt);

\node[red, below] at (.1,1) {$a_1$};
\node[red, below] at (1.7,1.05) {$a_2$};
\node[red, below] at (3.1,1.1) {$a_3$};
\node[below] at (3.7,1.15) {$b$};
\node at (3.6,1.25) {$c$};
\node[above] at (3.5,1.35) {$d$};

\draw[-] (0,-0.4)--(0,0)--(1,0.1)--(-0.2,0.2)--(0.8,0.3)--(4,0.4)--(4,-0.4);
\draw[-, red] (1, -.4)--(1, .30625);
\draw[-, red] (3.8, -.4)--(3.8, .39375);

  \fill[orange] (0,0) circle (0.9pt);
  \fill[blue] (1,0.1) circle (0.9pt);
  \fill[orange] (-0.2,0.2) circle (0.9pt);
  \fill[blue] (0.8,0.3) circle (0.9pt);
  \fill[orange] (4,0.4) circle (0.9pt);

\node[below] at (0.5,.05) {$d$};
\node at (0.4,.15) {$c$};
\node[above] at (0.3,.25) {$b$};
\node[red, above] at (0.9,.3) {$a_1$};
\node[red, above] at (2.3,.35) {$a_2$};
\node[red, above] at (3.9,.4) {$a_3$};
\end{tikzpicture}
\end{center}
\caption{Sketch for $N_{A,B,C}$ with $1/5<A<\frac{1}{2}$ and $C>0$. Here $\RE(\dev(a_1)) = \RE(\dev(a_3)) = 1-2A$ and $\RE(\dev(a_2)) = 5A-1$.}
\label{fig:omega_s_plus_tv2}
\end{figure}


\begin{figure}
 \begin{center}
\begin{tikzpicture}[scale= 2]
\draw[-] (0,1)--(4,1)--(5.1,2)--(1.1,2)--(0,1);

  \fill[blue] (0,1) circle (0.9pt);
  \fill[orange] (0.1,1) circle (0.9pt);
  \fill[blue] (0.3,1) circle (0.9pt);
  \fill[orange] (3,1) circle (0.9pt);
  \fill[blue] (4,1) circle (0.9pt);

\node at (0.05,.9) {$a'$};
\node at (0.2,.9) {$b'$};
\node at (1.65,.9) {$c'$};
\node at (3.5,.9) {$d$};
  
  \fill[violet] (1.1,2) circle (0.9pt);
  \fill[red] (2.1,2) circle (0.9pt);
  \fill[violet] (3.1,2) circle (0.9pt);
  \fill[red] (4.1,2) circle (0.9pt);
  \fill[violet] (5.1,2) circle (0.9pt);

\node at (1.6,2.1) {$e$};
\node at (2.6,2.1) {$f$};
\node at (3.6,2.1) {$g$};
\node at (4.6,2.1) {$h$};

\draw[-] (0,.4)--(4,0.4)--(2.9,-.6)--(-1.1,-.6)--(0,.4);

  \fill[orange] (0,.4) circle (0.9pt);
  \fill[blue] (1,0.4) circle (0.9pt);
  \fill[orange] (3.7,0.4) circle (0.9pt);
  \fill[blue] (3.9,0.4) circle (0.9pt);
  \fill[orange] (4,0.4) circle (0.9pt);

\node at (0.5,0.53) {$d'$};
\node at (2.35,0.53) {$c'$};
\node at (3.8,0.53) {$b'$};
\node at (3.95,0.53) {$a'$};
  
  \fill[red] (-1.1,-.6) circle (0.9pt);
  \fill[violet] (-.1,-.6) circle (0.9pt);
  \fill[red] (.9,-.6) circle (0.9pt);
  \fill[violet] (1.9,-.6) circle (0.9pt);
  \fill[red] (2.9,-.6) circle (0.9pt);

\node at (-.6,-.7) {$f$};
\node at (.4,-.7) {$e$};
\node at (1.4,-.7) {$h$};
\node at (2.4,-.7) {$g$};
\end{tikzpicture}
\end{center}
\caption{Representation of $N^+_{A, B} $ for $1/5<A<\frac{1}{2}$, with $|\vec{c'}| = 5A-1$, $|\vec{b'}|=A \mod 1-2A$ and $|\vec{a'}| = 1-2A-|\vec{b'}|$.}
\label{fig:omega_t_final2}
\end{figure}
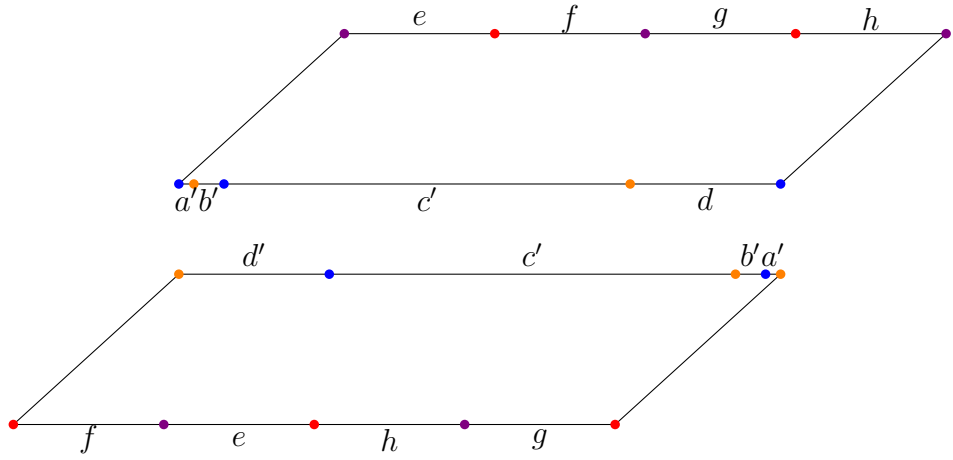


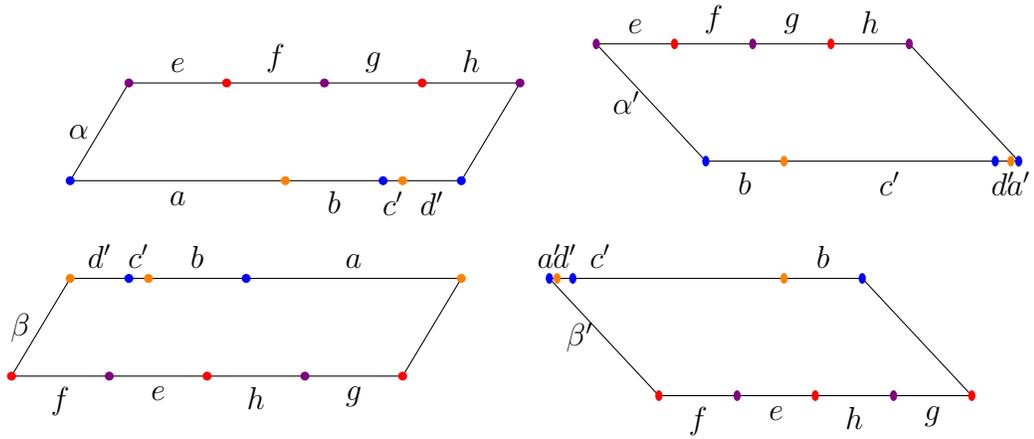
\begin{figure}
 \begin{center}
\begin{tikzpicture}[scale= 1.3]

\draw[-] (0,1)--(4,1)--(4.6,2)--(.6,2)--(0,1);

  \fill[blue] (0,1) circle (1.3pt);
  \fill[orange] (2.2,1) circle (1.3pt);
  \fill[blue] (3.2,1) circle (1.3pt);
  \fill[orange] (3.4,1) circle (1.3pt);
  \fill[blue] (4,1) circle (1.3pt);

\node[below] at (1.1,1) {$a$};
\node[below] at (2.7,1) {$b$};
\node[below] at (3.3,1) {$c'$};
\node[below] at (3.7,1) {$d'$};
\node[left] at (.3,1.5) {$\alpha$};

  \fill[violet] (0.6,2) circle (1.3pt);
  \fill[red] (1.6,2) circle (1.3pt);
  \fill[violet] (2.6,2) circle (1.3pt);
  \fill[red] (3.6,2) circle (1.3pt);
  \fill[violet] (4.6,2) circle (1.3pt);

\node[above] at (1.1,2) {$e$};
\node[above] at (2.1,2) {$f$};
\node[above] at (3.1,2) {$g$};
\node[above] at (4.1,2) {$h$};

\draw[-] (0,0)--(4,0)--(3.4,-1)--(-.6,-1)--(0,0);

  \fill[orange] (0,0) circle (1.3pt);
  \fill[blue] (.6,0) circle (1.3pt);
  \fill[orange] (.8,0) circle (1.3pt);
  \fill[blue] (1.8,0) circle (1.3pt);
  \fill[orange] (4,0) circle (1.3pt);

\node[above] at (0.3,0) {$d'$};
\node[above] at (.7,0) {$c'$};
\node[above] at (1.3,0) {$b$};
\node[above] at (2.9,0) {$a$};
\node[left] at (-.3,-.5) {$\beta$};

  \fill[red] (-.6,-1) circle (1.3pt);
  \fill[violet] (.4,-1) circle (1.3pt);
  \fill[red] (1.4,-1) circle (1.3pt);
  \fill[violet] (2.4,-1) circle (1.3pt);
  \fill[red] (3.4,-1) circle (1.3pt);

\node[below] at (-.1,-1) {$f$};
\node[below] at (.9,-1) {$e$};
\node[below] at (1.9,-1) {$h$};
\node[below] at (2.9,-1) {$g$};

\begin{scope}  [shift={(.1,0)}, xscale = .8, yscale = 1.2]
\draw[-] (8,1)--(12,1)--(10.6,2)--(6.6,2)--(8,1);

  \fill[blue] (8,1) circle (1.3pt);
  \fill[orange] (9,1) circle (1.3pt);
  \fill[blue] (11.7,1) circle (1.3pt);
  \fill[orange] (11.9,1) circle (1.3pt);
  \fill[blue] (12,1) circle (1.3pt);

\node[below] at (8.5,1) {$b$};
\node[below] at (10.35,1) {$c'$};
\node[below] at (11.8,1) {$d'$};
\node[below] at (12,1) {$a'$};
\node[left] at (7.3,1.5) {$\alpha'$};
  
  \fill[violet] (6.6,2) circle (1.3pt);
  \fill[red] (7.6,2) circle (1.3pt);
  \fill[violet] (8.6,2) circle (1.3pt);
  \fill[red] (9.6,2) circle (1.3pt);
  \fill[violet] (10.6,2) circle (1.3pt);

\node[above] at (7.1,2) {$e$};
\node[above] at (8.1,2) {$f$};
\node[above] at (9.1,2) {$g$};
\node[above] at (10.1,2) {$h$};

\draw[-] (6,0)--(10,0)--(11.4,-1)--(7.4,-1)--(6,0);

  \fill[blue] (6,0) circle (1.3pt);
  \fill[orange] (6.1,0) circle (1.3pt);
  \fill[blue] (6.3,0) circle (1.3pt);
  \fill[orange] (9,0) circle (1.3pt);
  \fill[blue] (10,0) circle (1.3pt);

\node[above] at (6,0) {$a'$};
\node[above] at (6.2,0) {$d'$};
\node[above] at (6.65,0) {$c'$};
\node[above] at (9.5,0) {$b$};
\node[left] at (6.7,-.5) {$\beta'$};
  
  \fill[red] (7.4,-1) circle (1.3pt);
  \fill[violet] (8.4,-1) circle (1.3pt);
  \fill[red] (9.4,-1) circle (1.3pt);
  \fill[violet] (10.4,-1) circle (1.3pt);
  \fill[red] (11.4,-1) circle (1.3pt);

\node[below] at (7.9,-1) {$f$};
\node[below] at (8.9,-1) {$e$};
\node[below] at (9.9,-1) {$h$};
\node[below] at (10.9,-1) {$g$};

\end{scope}
\end{tikzpicture}
\end{center}
\caption{Left: $N^-_{A,B}$ for $\frac{1}{2}<A<1$, with $\vec{\alpha} = \vec{\beta} = (B+\frac12, \frac1A)$, $|\vec{d'}| = 1-A \mod 2A-1$ and $|\vec{c'}| = 2A-1 - |\vec{b'}|$. Right: $N^-_{A,B}$ for $1/5<A<\frac{1}{2}$, with $\vec{\alpha'} = \vec{\beta'} = (B-(A+1)+\frac12, \frac1A)$, $|\vec{c'}| = 5A-1$, $|\vec{d'}|=A \mod 1-2A$ and $|\vec{a'}| = 1-2A-|\vec{d'}|$.} 
\label{fig:omega_t_minus}
\end{figure}


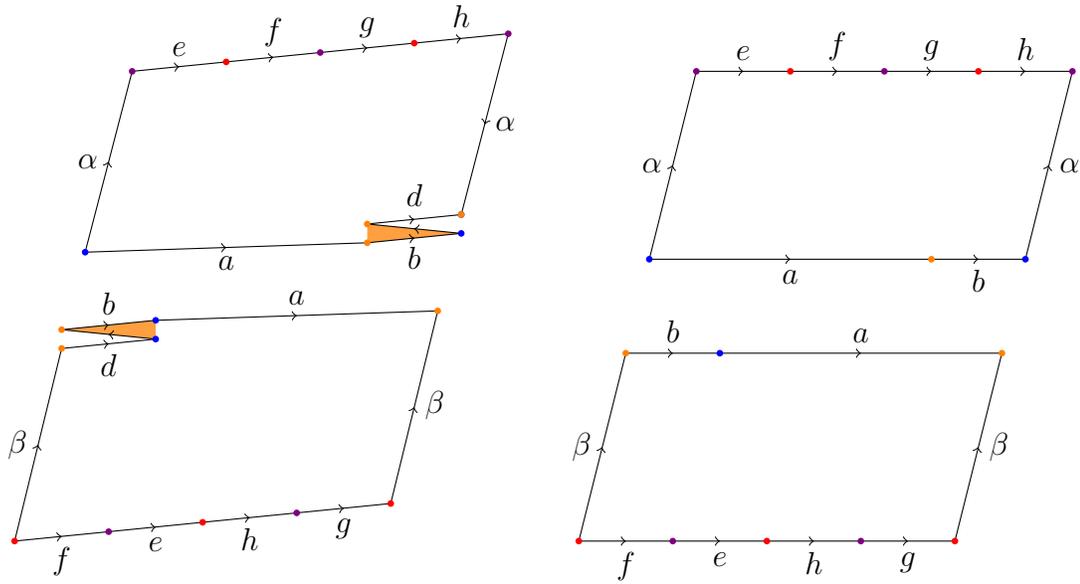
\begin{figure}
\begin{center}
\begin{tikzpicture}[scale=1.25]
\begin{scope} [shift={(-6,0)}]

\foreach \i in {0,...,3}{%
    \foreach \j in {0,5}{%
        \draw[middlearrow={>}] (\i+.25*\j,\j+.1*\i)--(\i+1+.25*\j,\j+.1+.1*\i) ;}}

\fill [fill=orange!75, draw = none] (1.5, 2.15) -- (.5, 2.25) -- (1.5, 2.35);
\draw[middlearrow={>}] (.5,2.05)--(1.5,2.15);
\draw[middlearrow={>}] (1.5,2.15)--(.5,2.25);
\draw[middlearrow={>}] (.5,2.25)--(1.5,2.35);
\draw[middlearrow={>}] (1.5,2.35)--(4.5,2.45);

\fill [fill=orange!75, draw = none] (3.75, 3.175) -- (4.75, 3.275) -- (3.75, 3.375);
\draw[middlearrow={>}] (.75,3.075)--(3.75,3.175);
\draw[middlearrow={>}] (3.75,3.175)--(4.75,3.275);
\draw[middlearrow={>}] (4.75,3.275)--(3.75,3.375);
\draw[middlearrow={>}] (3.75,3.375)--(4.75,3.475);

\draw[middlearrow={>}] (0,0)--(.5,2.05);
\draw[middlearrow={>}] (4,.4)--(4.5,2.45);
\draw[middlearrow={>}] (.75,3.075)--(1.25,5);
\draw[middlearrow={>}] (5.25,5.4)--(4.75,3.475);

\foreach \i in {0, 2, 4}{%
    \foreach \j in {0}{%
        \fill[red] (\i+.25*\j,\j+.1*\i) circle (1pt);}}
\foreach \i in {1, 3}{%
    \foreach \j in {0}{%
        \fill[violet] (\i+.25*\j,\j+.1*\i) circle (1pt);}}
\foreach \i in {0, 2, 4}{%
    \foreach \j in {5}{%
        \fill[violet] (\i+.25*\j,\j+.1*\i) circle (1pt);}}
\foreach \i in {1, 3}{%
    \foreach \j in {5}{%
        \fill[red] (\i+.25*\j,\j+.1*\i) circle (1pt);}}

\fill[blue] (1.5,2.15) circle (1pt);
\fill[blue] (1.5,2.35) circle (1pt);
\fill[blue] (.75,3.075) circle (1pt);
\fill[blue] (4.75,3.275) circle (1pt);
\fill[blue] (4.75,3.475) circle (1pt);
\fill[orange] (.5,2.05) circle (1pt);
\fill[orange] (.5,2.25) circle (1pt);
\fill[orange] (4.5,2.45) circle (1pt);
\fill[orange] (3.75,3.175) circle (1pt);
\fill[orange] (3.75,3.375) circle (1pt);
\fill[orange] (4.75,3.475) circle (1pt);

\node[left] at (1,4.0375) {$\alpha$};
\node[right] at (5,4.4375) {$\alpha$};
\node[left] at (.25,1.025) {$\beta$};
\node[right] at (4.25,1.45) {$\beta$};

\node[above] at (1.75,5.05) {$e$};
\node[above] at (2.75,5.15) {$f$};
\node[above] at (3.75,5.25) {$g$};
\node[above] at (4.75,5.35) {$h$};
\node[below] at (2.25,3.15) {$a$};
\node[below] at (4.25,3.25) {$b$};
\node[above] at (4.25,3.45) {$d$};

\node[below] at (1,2.1) {$d$};
\node[above] at (1,2.3) {$b$};
\node[above] at (3,2.4) {$a$};
\node[below] at (0.5,0.05) {$f$};
\node[below] at (1.5,0.15) {$e$};
\node[below] at (2.5,0.25) {$h$};
\node[below] at (3.5,0.35) {$g$};

\end{scope}


\foreach \i in {0,...,3}{%
    \foreach \j in {0,5}{%
        \draw[middlearrow={>}] (\i+.25*\j,\j)--(\i+1+.25*\j,\j) ;}}
\draw[middlearrow={>}] (.5,2)--(1.5,2);
\draw[middlearrow={>}] (1.5,2)--(4.5,2);
\draw[middlearrow={>}] (.75,3)--(3.75,3);
\draw[middlearrow={>}] (3.75,3)--(4.75,3);

\foreach \j in {0,3}{%
    \foreach \i in {0,4}{%
        \draw[middlearrow={>}] (\i+.25*\j,\j)--(\i+.5+.25*\j,\j+2) ;}}
\foreach \i in {0, 2}{%
    \fill[red] (\i,0) circle (1pt);
    \fill[red] (\i+2.25,5) circle (1pt);
    \fill[violet] (\i+1.25,5) circle (1pt);
    \fill[violet] (\i+1,0) circle (1pt);}
\fill[red] (4,0) circle (1pt);
\fill[violet] (5.25,5) circle (1pt);

\fill[blue] (1.5,2) circle (1pt);
\fill[blue] (.75,3) circle (1pt);
\fill[blue] (4.75,3) circle (1pt);
\fill[orange] (.5,2) circle (1pt);
\fill[orange] (4.5,2) circle (1pt);
\fill[orange] (3.75,3) circle (1pt);

\node[left] at (1,4) {$\alpha$};
\node[right] at (5,4) {$\alpha$};
\node[left] at (.25,1) {$\beta$};
\node[right] at (4.25,1) {$\beta$};

\node[above] at (1.75,5) {$e$};
\node[above] at (2.75,5) {$f$};
\node[above] at (3.75,5) {$g$};
\node[above] at (4.75,5) {$h$};
\node[below] at (2.25,3) {$a$};
\node[below] at (4.25,3) {$b$};

\node[above] at (1,2) {$b$};
\node[above] at (3,2) {$a$};
\node[below] at (0.5,0) {$f$};
\node[below] at (1.5,0) {$e$};
\node[below] at (2.5,0) {$h$};
\node[below] at (3.5,0) {$g$};

\end{tikzpicture}
\end{center}
\caption{Left: the surface $N_{\frac12, 0, C}$ with $C = 0.1$. The two orange triangles glued along $b$ form a horizontal cylinder. As $C\to 0$, this cylinder degenerates. Right: the surface $N_{\frac12, 0,0}$. Note that the blue and orange points have cone angle $2\pi$, thus $N_{\frac12,0,0} \in \HH_1(1,1,0,0)$.}
\label{fig:Nhalf}
\end{figure}


\chapter{General case of Veech surfaces}
\label{chap:generalCase}

Given a translation surface $M$, we write $\SL(M)\subset \SL(2,\RR)$ for the group of affine transformations of $M$.
Recall that $M$ is said to be a Veech surface if $\SL(M)$ is a lattice of $\SL(2,\RR)$, and that in this case we have the following dichotomy.
See \cite{hs06} or \cite{v89}.
\begin{thm}[Veech dichotomy]
Let $M$ be a Veech translation surface. 
Then for any direction $\theta$ the flow in direction $\theta$ is either ergodic or $M$ decomposes into a finite number of cylinders in the direction $\theta$ that have pairwise commensurable moduli $\mu_i$, i.e. for all $i$ and $j$, $\mu_i/\mu_j\in \QQ$.
\end{thm}
Let now $M$ be a Veech translation surface. We are going to examine that part of its locus $\MM$ for which surfaces admit decompositions into horizontal cylinders. Up to rotation we may assume that $M$ is such a surface, and we write $\Cyl_i$ for the horizontal cylinders composing $M$.
We fix a real vector $v\in H^1(M, \Sigma, \RR)$ such that $\Dir(M,v)$ is finite.
Thanks to \cite{w15} where the computation of $\Dir(M, v)$ is carefully detailed, we know that infinitely many examples of such pairs of surfaces and vectors can be constructed by looking at abelian covers of the pillow-case.

For each $i$, we choose a decomposition of $\Cyl_i$ as in Figure \ref{fig:C_i}.
In this decomposition, we put the $a^i_j$ at the bottom of each $\Cyl_i$ and the $b^i_j$ at the top. $\alpha_i$ are on the sides (we do not require that the $\alpha_i$ are vertical).
Note the orientation of the edges.
Because $M$ is a Veech surface, the moduli of all cylinders are pairwise commensurable and the finiteness condition on $v$ implies that all $v(a^i_1 + a^i_2 + \ldots)$ and all $v(b^i_1+ b^i_2+ \ldots)$ are 0. 
Call $l_i = a^i_1 + a^i_2 + \ldots >0 $ the length of $\Cyl_i$ and $h_i = \IM(\alpha_i)>0$ its height.
 
\begin{figure}
\begin{center}
\begin{tikzpicture}[scale=1.5]
\draw[middlearrow={>}] (0.4, 1)--(1.1, 1);
  \fill (.4,1) circle (.8pt);
\draw[middlearrow={>}] (1.1, 1)--(1.8, 1);
  \fill (1.1,1) circle (.8pt);
\draw[middlearrow={>}] (1.8, 1)--(2.5, 1);
  \fill (1.8,1) circle (.8pt);
\draw[middlearrow={>}] (2.5, 1)--(3, 1);
  \fill (2.5,1) circle (.8pt);
\draw[middlearrow={>}] (3, 1)--(3.7, 1);
  \fill (3,1) circle (.8pt);
  \fill (3.7,1) circle (.8pt);

\draw[middlearrow={>}] (0, 0)--(1.3, 0);
  \fill (0,0) circle (.8pt);
\draw[middlearrow={>}] (1.3, 0)--(1.6, 0);
  \fill (1.3,0) circle (.8pt);
\draw[middlearrow={>}] (1.6, 0)--(2.9, 0);
  \fill (1.6,0) circle (.8pt);
\draw[middlearrow={>}] (2.9, 0)--(3.3, 0);
  \fill (2.9,0) circle (.8pt);
  \fill (3.3,0) circle (.8pt);

\draw[middlearrow={>}] (0, 0)--(0.4, 1);
\draw[middlearrow={>}] (3.3, 0)--(3.7, 1);

\node at (0,.5) {$\alpha_i$};
\node at (3.7,.5) {$\alpha_i$};

\node at (0.65,1.2) {$b^i_1$};
\node at (1.45,1.2) {$b^i_2$};
\node at (2.15,1.2) {$b^i_3$};
\node at (2.75,1.2) {$b^i_4$};
\node at (3.35,1.2) {$b^i_5$};

\node at (.65,-.2) {$a^i_1$};
\node at (1.45,-.2) {$a^i_2$};
\node at (2.25,-.2) {$a^i_3$};
\node at (3.1,-.2) {$a^i_4$};

\draw[red, densely dashed] (.2,.5)--(3.5,.5);
\fill[red] (0.2,.5) circle (0.8pt);
\fill[red] (3.5,.5) circle (0.8pt);

\end{tikzpicture}
\end{center}
\caption{One cylinder $\Cyl_i$ in the decomposition of $M$ in horizontal cylinders.}
\label{fig:C_i}
\end{figure}
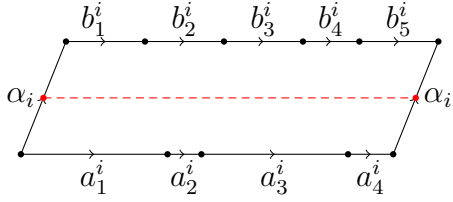

As above, set 
\[
    M_{A, B, C} =\left( \begin{array}{cc}
A & B \\
C & \frac{1+BC}{A} \end{array} \right) .M
\]
and when defined 
\[
N_{A, B, C} = \Push(M_{A, B,C},v_0).
\]

\begin{rem}
\label{rem:small_As}
Let $A$ be big enough that for all $i$ and $k$, $A \RE(\dev(M,a^i_k))>|v(a^i_k)|$. Then $N_{A, 0,0 }$ exists and lies in the same stratum as $M$, so as in Remark \ref{rem:thm1}, the interesting behavior happens as $A$ becomes smaller.
\end{rem}

Lemma \ref{lem:cyclic} below will allow us to find representations of $N_{A,B,C}$ as an union of polygons with gluing when $C\neq 0$ is small (see Proposition \ref{prop:normal}).
\begin{lem}
\label{lem:cyclic}
Up to relabeling the $a^i_k$, we can assume that the $v(a^i_1)$, $v(a^i_1 + a^i_2)$, $\ldots$ are all non-negative (resp. non-positive), and that the $v(b^i_1)$, $v(b^i_1 + b^i_2)$, $\ldots$ are all non-positive (resp. non-negative). 
\end{lem}
\begin{proof}
Let us only prove the claim that all that all the $v(a^i_1+\ldots +a^i_j)$ are non-negative. The other cases are similar.
Note that for any map $\varphi: \ZZ/k \to \RR$ such that $\sum_{x\in \ZZ/k} \varphi(x)=0$, the following map  
\[
\begin{cases} \ZZ/k \times \NN \to \RR\\
(x,l) \mapsto \sum_{i=0}^l \varphi(x + i) \end{cases} 
\]
has a finite image, so its maximum is obtained at some pair $(x_0,l_0)$.
Applying this to $\varphi: x \mapsto v(a^i_x)$ with $k$ the number of bottom horizontal edges in $C^-_i$, we obtain some index $x_0$ that up to a standard cut and paste on $\Cyl_i$ we can assume is 1.
Now we see that no $v(a^i_1 + \ldots + a^i_l)$ can be negative since then $v(a^i_{l+1} + a^i_{l+2} + \ldots + a^i_{1+l_0})$ would be bigger that $v(a^i_1 + \ldots + a^i_{1+l_0})$.
This would contradict the definition of $(x_0, l_0)$.
\end{proof}

\begin{prop}
\label{prop:normal}
There exists a number $\eta>0$ such that for any $(A,B,C) \in \RR^{>0} \times \RR \times (-\eta, \eta) $, the push $N_{A,B,C}$ is well-defined and in the same stratum as $M$.
\end{prop}
\begin{proof}
Let us first investigate the case $C>0$. We apply Lemma \ref{lem:cyclic} so that $N_{A,B,C}$ admits the polygonal representation depicted in Figure \ref{fig:IET} on the left.
Indeed, because each $v(a^i_1 + \ldots + a^i_l)$ is non-negative, each endpoint of $a^i_l$ in the polygon in Figure \ref{fig:IET} is on or at the right of the segment in dashed orange, thus it does not enter the center parallelogram. Similarly, the endpoints of $b^i_l$ never enter the center parallelogram.
The only additional condition we need to check is that the center parallelogram is non-degenerate, i.e. that $\left(\dev(M_{A,B,C}, \alpha_i)+v(\alpha_i) \right)\cdot \vec{n} >0$ where $\vec{n}$ is the orthonormal vector making an angle of $+\pi/2$ with $\dev(M_{A,B,C}, a^i_1 + a^i_2 + \ldots) + v(a^i_1 + a^i_2 + \ldots) = l_i (A,C)$. By direct computation, this is equivalent to $-Cv(\alpha_i)+h_i>0$. 
Therefore the push $N_{A,B,C}$ exists for any $C>0$ verifying $C<\max_i(\frac{h_i}{|v(\alpha_i)|})$.

The case $C<0$ is similar. Note that we might have to relabel the $a^i_k$, resp. $b^i_k$. Writing $\widetilde{\alpha_i}$ for the new side edges of the parallelogram in Figure \ref{fig:IET}, we find again that $N_{A,B,C}$ exists for any $C<0$ verifying $C>-\max_i(\frac{h_i}{|v(\widetilde{\alpha_i})|})$, and the proof is complete.
\end{proof}

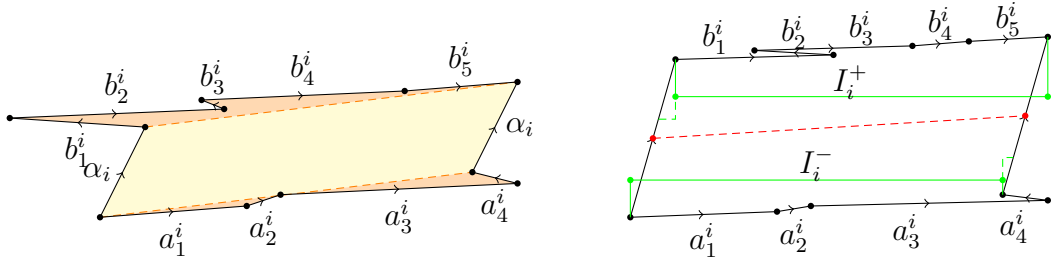
\begin{figure}
\begin{center}
\begin{tikzpicture}[scale=1.5]
\begin{scope}[shift={(-4.7,0)}]

\fill[yellow!20!white] (0,0)--(3.3,.4)--(3.7,1.2)--(.4,.8)--(0,0);
\fill[orange!30!white] (.4,.8)--(-.8,.88)--(1.1,.96)--(.9,1.04)--(2.7,1.12)--(3.7,1.2)--(.4,.8);
\fill[orange!30!white] (0,0)--(1.3,.1)--(1.6,.2)--(3.7,.3)--(3.3,.4)--(0,0);
\draw[orange, densely dashed] (0,0)--(3.3,.4);
\draw[orange, densely dashed] (.4,0.8)--(3.7,1.2);

\draw[middlearrow={>}] (0.4, .8)--(-.8, .88);
  \fill (.4,.8) circle (.8pt);
\draw[middlearrow={>}] (-.8, .88)--(1.1, .96);
  \fill (-.8,.88) circle (.8pt);
\draw[middlearrow={>}] (1.1, .96)--(.9, 1.04);
  \fill (1.1,.96) circle (.8pt);
\draw[middlearrow={>}] (.9, 1.04)--(2.7, 1.12);
  \fill (.9,1.04) circle (.8pt);
\draw[middlearrow={>}] (2.7, 1.12)--(3.7, 1.2);
  \fill (2.7,1.12) circle (.8pt);
  \fill (3.7,1.2) circle (.8pt);

\draw[middlearrow={>}] (0, 0)--(1.3, 0.1);
  \fill (0,0) circle (.8pt);
\draw[middlearrow={>}] (1.3, 0.1)--(1.6, 0.2);
  \fill (1.3,0.1) circle (.8pt);
\draw[middlearrow={>}] (1.6, 0.2)--(3.7, 0.3);
  \fill (1.6,.2) circle (.8pt);
\draw[middlearrow={>}] (3.7, 0.3)--(3.3, 0.4);
  \fill (3.7,.3) circle (.8pt);
  \fill (3.3,.4) circle (.8pt);

\draw[middlearrow={>}] (0, 0)--(0.4, .8);
\draw[middlearrow={>}] (3.3, 0.4)--(3.7, 1.2);


\node[below] at (-.2,.84) {$b^i_1$};
\node[above] at (.15,.92) {$b^i_2$};
\node[above] at (1,1) {$b^i_3$};
\node[above] at (1.8,1.08) {$b^i_4$};
\node[above] at (3.2,1.16) {$b^i_5$};

\node[below] at (.65,.05) {$a^i_1$};
\node[below] at (1.45,.15) {$a^i_2$};
\node[below] at (2.65,.25) {$a^i_3$};
\node[below] at (3.5,.35) {$a^i_4$};

\node[left] at (.2,.4) {$\alpha_i$};
\node[right] at (3.5,.8) {$\alpha_i$};

\end{scope}


\draw[middlearrow={>}] (0.4, 1.4)--(1.8, 1.44);
  \fill (.4,1.4) circle (.8pt);
\draw[middlearrow={>}] (1.8, 1.44)--(1.1, 1.48);
  \fill (1.8,1.44) circle (.8pt);
\draw[middlearrow={>}] (1.1, 1.48)--(2.5, 1.52);
  \fill (1.1,1.48) circle (.8pt);
\draw[middlearrow={>}] (2.5, 1.52)--(3, 1.56);
  \fill (2.5,1.52) circle (.8pt);
\draw[middlearrow={>}] (3, 1.56)--(3.7, 1.6);
  \fill (3,1.56) circle (.8pt);
  \fill (3.7,1.6) circle (.8pt);

\draw[middlearrow={>}] (0, 0)--(1.3, 0.05);
  \fill (0,0) circle (.8pt);
\draw[middlearrow={>}] (1.3, 0.05)--(1.6, 0.1);
  \fill (1.3,0.05) circle (.8pt);
\draw[middlearrow={>}] (1.6, 0.1)--(3.7, 0.15);
  \fill (1.6,.1) circle (.8pt);
\draw[middlearrow={>}] (3.7, 0.15)--(3.3, 0.2);
  \fill (3.7,.15) circle (.8pt);
  \fill (3.3,.2) circle (.8pt);

\draw[middlearrow={>}] (0, 0)--(0.4, 1.4);
\draw[middlearrow={>}] (3.3, 0.2)--(3.7, 1.6);


\node at (0.75,1.6) {$b^i_1$};
\node at (1.45,1.64) {$b^i_2$};
\node at (2.05,1.68) {$b^i_3$};
\node at (2.75,1.72) {$b^i_4$};
\node at (3.35,1.76) {$b^i_5$};

\node at (.65,-.2) {$a^i_1$};
\node at (1.45,-.16) {$a^i_2$};
\node at (2.45,-.12) {$a^i_3$};
\node at (3.4,-.08) {$a^i_4$};

\draw[red, densely dashed] (.2,.7)--(3.5,.9);
\fill[red] (0.2,.7) circle (0.8pt);
\fill[red] (3.5,.9) circle (0.8pt);

\draw[green] (0, 0)--(0,.33)--(3.3,.33)--(3.3,0.2);
\draw[green, densely dashed] (3.3,.33)--(3.3,.53)--(3.394,.53);
\fill[green] (0,.33) circle (0.8pt);
\fill[green] (3.3,.33) circle (0.8pt);
\node[] at (1.65,.47) {$I_i^-$}; 

\draw[green] (0.4, 1.4)--(0.4,1.07)--(3.7,1.07)--(3.7,1.6);
\draw[green, densely dashed] (0.249, .87)--(0.4,.87)--(0.4,1.07);
\fill[green] (0.4,1.07) circle (0.8pt);
\fill[green] (3.7,1.07) circle (0.8pt);
\node[] at (1.95,1.2) {$I_i^+$}; 

\end{tikzpicture}
\end{center}
\caption{Left: $N_{A,B,C}$ is always well-defined for $C$ small enough. Such a representation of $N_{A,B,C}$ always exist by Lemma \ref{lem:cyclic}. Right: the intervals $I_i^+$ and $I_i^-$ in $N_{A,B,C}$ defining the Interval Exchange Transformation $\varphi_{A,B}$.}
\label{fig:IET}
\end{figure}

As above, define $N^\pm_{A,B} = \lim_{C\to 0^\pm} N_{A,B,C}$ when this limit is well-defined.
In order to describe these limits, let us define an IET in the following way.
Consider the union of intervals $I = \sqcup_i I_i$ where $I_i = [0, l_i]$. 
For each cylinder $\Cyl_i$, we write $I_i^-$ for the horizontal segment $I_i^-$ of length $l_i$ starting $\IM(\dev(N_{A,B,C},\alpha_i))/3$ vertically above the leftmost vertex of $a_1^i$, and $I_i^+$ for the horizontal segment starting $\IM(\dev(N_{A,B,C},\alpha_i))/3$ vertically below the leftmost vertex of $b_1^i$.
See Figure \ref{fig:IET} on the right, in green.
In order to compute $N^+_{A,B}$, we define an Interval Exchange Transformation $\varphi_{A,B}$ on $I$ by the first-visit map (following downwards vertical geodesics) from the union of $I_i^-$ to the union of $I_i^+$, both identified with $I$ by identifying $I_i^{\pm}$ and $I_i$ isometrically.
Note that $\varphi_{A,B}$ only depends on the sign $C$ when $C$ is small enough.
Under $\varphi_{A,B}$, each $I_i^-$ (resp. $I_i^+$) is subdivided in subintervals ${a'}_j^i$ (resp. ${b'}_j^i$), $j$ running from left to right, and $\varphi_{A,B}$ acts as a permutation of these subintervals.
One can do a similar construction for the case of $N^-_{A,B}$, yielding an IET $\psi$ that is in general different from $\varphi$.

Before we investigate the existence of the limits $N^\pm_{A,B}$ let us define the sets $\Xi_{A,B,C}$, $\Gamma_{A,B,C}$ and the notion of ``edge paths around a loop''. 

We also define $\TM_{A,B,C}$ (resp. $\TN_{A,B,C}$) to be the surface obtained from $M_{A,B,C}$ (resp. $N_{A,B,C}$) by removing the mid-height segment of each horizontal cylinder $\Cyl_i$ of $M_{A,B,C}$ (resp. $N_{A,B,C}$). These segments are drawn in dashed red in Figures \ref{fig:C_i} and \ref{fig:IET} on the right. 
We will call $\Cyl^+_i$ (resp. $\Cyl^-_i$) the upper (resp. lower) component of $\Cyl_i$ with the mid-height segment removed.

Define now $\Xi_{A,B,C}$ as the set of vertical segments $\gamma: [0, 1] \to \TN_{A,B,C}$ that start and end in $\Sigma$, and are one-to-one from $\gamma^{-1}(\TN_{A,B,C}\setminus \Sigma)$ to its image. 
Otherwise put, it is the set of vertical saddle connections in $\TN_{A,B,C}$ together with concatenations of distinct vertical saddle connections $\gamma_1,\ldots, \gamma_k$ such that $\gamma_i(1) = \gamma_{i+1}(0)$.
We start by defining $\Gamma_{A,B,C}$ to be the set of one-to-one vertical loops $\gamma: S^1 \to \TN_{A,B,T}\setminus \Sigma$.
Note that each $\gamma\in \Gamma_{A,B,C}$ defines a horizontal cylinder bounded by end-circles that are elements of $\Xi_{A,B,C}$. 
In particular, $\Gamma_{A,B,C} \neq \varnothing \implies \Xi_{A,B,T} \neq \varnothing$.

We now define ``edge paths around a loop'', a tool that will be used in the proofs of Proposition \ref{prop:lim_descr}, Lemma \ref{lem:Aplus} and Lemma \ref{lem:nonzero}. 
For each cylinder $\Cyl_i^\pm$, write $B_i^\pm$ for the part of the boundary of $\Cyl_i^\pm$ that only consists of edges $a_k^i$ (resp. $b_k^i$).
For any $\gamma\in \Xi_{A,B,C}$, we define the edge path $\overline{\gamma}$ around $\gamma$ in the following way. 
If $\gamma$ is in $B_i^\pm$, then set $\overline{\gamma} = \gamma$. 
Else, let $x\in (0, 1]$ be the smallest such that $\gamma(x) \in \Sigma \cup B_i^\pm$.
Then $\gamma([0,x])$ separates $\Cyl_i$ in two connected components, one with the mid-height waist of $\Cyl_i^\pm$ on its boundary, and the other bordered by a path in $B_i^\pm$.
Define $\overline{\gamma}$ on $[0, x]$ to be this latter path. 
Then repeat this procedure starting at $x$, and so on until you reach 1. See Figure \ref{fig:edgepath} for an explicit example.
Note that $\overline{\gamma}$ is the concatenation of edges $e_j = a_k^i$ (resp. $b_k^i$), and that $\gamma$ and $\overline{\gamma}$ define the same element of $H_1(\TN_{A,B,C}, \Sigma, \ZZ)$.
Furthermore we have the following equality of vectors in $\RR^2$: 
\[
\dev(N_{A,B,C}, \gamma) = \sum_j \dev(N_{A,B,C}, e_j).
\]

\begin{prop}
Let $I_i$, $\varphi$ and ${a'}_j^i$, ${b'}_j^i$ defined as above. Then the limit $N^+_{A,B}$ exists and is described as the surface with cylinders ${\Cyl'}_i$ as in Figure \ref{fig:IET'} and with gluings defined by $\varphi_{A,B}$.
Specifically, ${\Cyl'}_i$ is a paralellogram with bottom segment $I_i^-$ (resp. top segment $I_i^+$) subdivided in ${a'}_j^i$ (resp. ${b'}_j^i$) and identified via $\varphi$, and with side segments $\alpha_i$ with $\dev(N^+_{A,B},\alpha_i) = \dev(M_{A,B,0}) +  v(\alpha_i)$. 
The same result hold for $N^-_{A,B}$ and $\psi$, with the appropriate relabeling of the edges $a^i_l$ and $b^i_l$.
\label{prop:lim_descr}
\end{prop}
\begin{proof}

Fix $C>0$ and cut downward vertical slits on $\Cyl_i$, starting at points $I^-$ corresponding to singular points of $I$, and ending at points of $I^+$ (these also correspond to singular points of $I$ by definition of $I^\pm$). 
On Figure \ref{fig:convex_quad}, these are the slits $\alpha_1$, $\alpha_2$ and $\alpha_3$.
Then reglue the pieces of vertical strips on top of each other as in Figures \ref{fig:omega_s_plus_tv_test} to \ref{fig:convex_quad}.
As in the right-most picture of Figure \ref{fig:convex_quad}, we now let $C \to 0$.
As we saw in the case of $M =  M_{\WMS}$ and $v=v_0$, $N_{\frac12, 0, C}$ presents a horizontal cylinder that is not glued back to any of the $\Cyl_i$. First assume that all the pieces we cut are glued to some cylinder $\Cyl_i$. 
If there are no vertical saddle connection that is a union of the vertical slits we cut, then we see as in Lemma \ref{lem:thm1part2} that the limit is in the same stratum as $M$. 
Else, then a vertical saddle connection gets shrunk to a point, and the limit is the one you expect. This is similar to Lemma \ref{lem:ex1}.
In order to complete the proof, we need to investigate what happens in the case that some pieces we cut cannot be glued below any of the cylinders $\Cyl_i$. This is where the fact that $M$ is Veech is used in the most critical way. Indeed, by Veech dichotomy these pieces form an union of horizontal cylinders $\Cyl$. Each of the mid-loops of these cylinders is an element $\gamma$ of $\Gamma_{A,B,C}$. By looking at the edge path $\overline{\gamma}$ around $\gamma$, we see that $\dev(N_{A,B,C},\gamma)=i\lambda C$ for $\lambda$ some positive integer. Thus as $C\to 0$ the injectivity radii of elements of each cylinder $\Cyl$ tend to 0 and we are in a situation very similar to Lemma \ref{lem:ex2}: $\Cyl$ does not contribute to the limit of $N_{A,B,C}$ in the sense of Mirzakhani-Wright. This completes the proof.
\end{proof}

\begin{figure}
\begin{center}
\begin{tikzpicture}[scale=2.1]
\draw[green,middlearrow={>}] (0.4, 1.4)--(1.4, 1.4);
  \fill (.4,1.4) circle (.8pt);
\draw[green,middlearrow={>}] (1.4, 1.4)--(1.8, 1.4);
  \fill (1.4,1.4) circle (.8pt);
\draw[green,middlearrow={>}] (1.8, 1.4)--(2.6, 1.4);
  \fill (1.8,1.4) circle (.8pt);
\draw[green,middlearrow={>}] (2.6, 1.4)--(3, 1.4);
  \fill (2.6,1.4) circle (.8pt);
\draw[green,middlearrow={>}] (3, 1.4)--(3.7, 1.4);
  \fill (3,1.4) circle (.8pt);
  \fill (3.7,1.4) circle (.8pt);

\draw[green,middlearrow={>}] (0, 0)--(1, 0);
  \fill (0,0) circle (.8pt);
\draw[green,middlearrow={>}] (1, 0)--(1.4, 0);
  \fill (1,0) circle (.8pt);
\draw[green,middlearrow={>}] (1.4, 0)--(1.8, 0);
  \fill (1.4,0) circle (.8pt);
\draw[green,middlearrow={>}] (1.8, 0)--(3.3, 0);
  \fill (1.8,0) circle (.8pt);
  \fill (3.3,0) circle (.8pt);

\draw[middlearrow={>}] (0, 0)--(0.4, 1.4);
\draw[middlearrow={>}] (3.3, 0)--(3.7, 1.4);

\node at (0,.7) {$\alpha_i$};
\node at (3.7,.7) {$\alpha_i$};

\node at (0.9,1.6) {${b'}^i_1$};
\node at (1.6,1.6) {${b'}^i_2$};
\node at (2.05,1.6) {${b'}^i_3$};
\node at (2.75,1.6) {${b'}^i_4$};
\node at (3.35,1.6) {${b'}^i_5$};

\node at (.5,-.2) {${a'}^i_1$};
\node at (1.2,-.2) {${a'}^i_2$};
\node at (1.6,-.2) {${a'}^i_3$};
\node at (2.55,-.2) {${a'}^i_4$};

\draw[red, densely dashed] (.2,.7)--(3.5,.7);
\fill[red] (0.2,.7) circle (0.8pt);
\fill[red] (3.5,.7) circle (0.8pt);

\node[] at (1.65,.22) {$I_i^-$}; 

\node[] at (1.95,1.2) {$I_i^+$}; 

\end{tikzpicture}
\end{center}
\caption{The cylinder ${\Cyl'}_i$ with gluings. We use the IET $I$ to describe the limit $N^+_{A,B}$.}
\label{fig:IET'}
\end{figure}
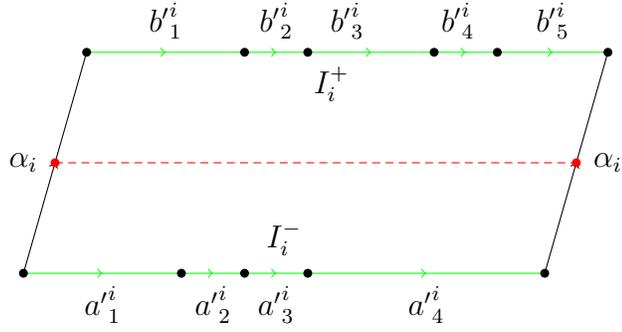


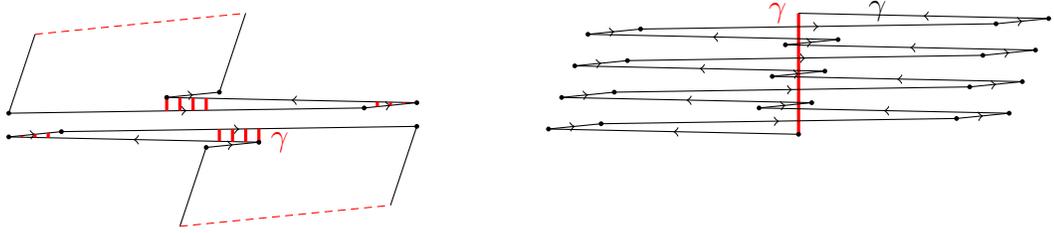
\begin{figure}
\begin{center}
\begin{tikzpicture}[scale=.7]
\draw[middlearrow={>}] (0, 0)--(6.75,.1);
  \fill (0,0) circle (1.3pt);
\draw[middlearrow={>}] (6.75, 0.1)--(7.75, 0.2);
  \fill (6.75, 0.1) circle (1.3pt);
\draw[middlearrow={>}] (7.75, 0.2)--(3, 0.3);
  \fill (7.75, 0.2) circle (1.3pt);
\draw[middlearrow={>}] (3, 0.3)--(4, 0.4);
  \fill (3, 0.3) circle (1.3pt);
  \fill (4,.4) circle (1.3pt);

\draw[-] (0, 0)--(.5, 1.5);
\draw[-] (4, 0.4)--(4.5, 1.9);
\draw[red, densely dashed] (0.5,1.5)--(4.5,1.9);

\draw[middlearrow={>}] (3.75, -.65)--(4.75,-.55);
  \fill (3.75, -.65) circle (1.3pt);
\draw[middlearrow={>}] (4.75, -.55)--(0, -.45);
  \fill (4.75, -.55) circle (1.3pt);
\draw[middlearrow={>}] (0, -.45)--(1, -.35);
  \fill (0, -.45) circle (1.3pt);
\draw[middlearrow={>}] (1, -.35)--(7.75, -.25);
  \fill (1, -.35) circle (1.3pt);
  \fill (7.75, -.25) circle (1.3pt);

\draw[-] (3.25, -2.15)--(3.75, -.65);
\draw[-] (7.25, -1.75)--(7.75, -.25);
\draw[red, densely dashed] (3.25, -2.15)--(7.25, -1.75);
\draw[red, very thick] (3,3/67.5)--(3,.3);
\draw[red, very thick] (3.25,3.25/67.5)--(3.25,.294736842);
\draw[red, very thick] (3.5,3.5/67.5)--(3.5,.289473684);
\draw[red, very thick] (3.75,3.75/67.5)--(3.75,.284210526);

\draw[red, very thick] (7,.125)--(7,.215789474);
\draw[red, very thick] (7.25,.15)--(7.25,.210526316);
\draw[red, very thick] (7.5,.175)--(7.5,.205263158);

\draw[red, very thick] (.25,-.45526)--(.25,-.425);
\draw[red, very thick] (.5,-.460526)--(.5,-.4);
\draw[red, very thick] (.75,-.465789)--(.75,-.375);

\draw[red, very thick] (4,-.534211)--(4,-.305555);
\draw[red, very thick] (4.25,-.539474)--(4.25,-.301852);
\draw[red, very thick] (4.5,-.544737)--(4.5,-.298148);
\draw[red, very thick] (4.75,-.55)--(4.75,-.294444);

\fill[] (3, .3) circle (1.2pt);
\fill[] (4.75, -.55) circle (1.2pt);

\node[red, right, very thick] at (4.75, -.55) {$\gamma$};
\draw[red, very thick] (15, -.4)--(15, 1.9);
\node[red, left, very thick] at (15, 1.9) {$\gamma$};
\node[very thick] at (16.5, 2) {$\overline{\gamma}$};

\draw[middlearrow={>}] (15, -.4)--(10.25, -.3);
  \fill (15, -.4) circle (1.3pt);
\draw[middlearrow={>}] (10.25, -.3)--(11.25, -.2);
  \fill (10.25, -.3) circle (1.3pt);
\draw[middlearrow={>}] (11.25, -.2)--(18, -.1);
  \fill (11.25, -.2) circle (1.3pt);
\draw[middlearrow={>}] (18, -.1)--(19, 0);
  \fill (18, -.1) circle (1.3pt);
\draw[middlearrow={>}] (19, 0)--(14.25, .1);
  \fill (19, 0) circle (1.3pt);
\draw[middlearrow={>}] (14.25, .1)--(15.25, .2);
  \fill (14.25, .1) circle (1.3pt);

\draw[middlearrow={>}] (15.25, .2)--(10.5, .3);
  \fill (15.25, .2) circle (1.3pt);
\draw[middlearrow={>}] (10.5, .3)--(11.5, .4);
  \fill (10.5, .3) circle (1.3pt);
\draw[middlearrow={>}] (11.5, .4)--(18.25, .5);
  \fill (11.5, .4) circle (1.3pt);
\draw[middlearrow={>}] (18.25, .5)--(19.25, .6);
  \fill (18.25, .5) circle (1.3pt);
\draw[middlearrow={>}] (19.25, .6)--(14.5, .7);
  \fill (19.25, .6) circle (1.3pt);
\draw[middlearrow={>}] (14.5, .7)--(15.5, .8);
  \fill (14.5, .7) circle (1.3pt);

\draw[middlearrow={>}] (15.5, .8)--(10.75, .9);
  \fill (15.5, .8) circle (1.3pt);
\draw[middlearrow={>}] (10.75, .9)--(11.75, 1);
  \fill (10.75, .9) circle (1.3pt);
\draw[middlearrow={>}] (11.75, 1)--(18.5, 1.1);
  \fill (11.75, 1) circle (1.3pt);
\draw[middlearrow={>}] (18.5, 1.1)--(19.5, 1.2);
  \fill (18.5, 1.1) circle (1.3pt);
\draw[middlearrow={>}] (19.5, 1.2)--(14.75, 1.3);
  \fill (19.5, 1.2) circle (1.3pt);
\draw[middlearrow={>}] (14.75, 1.3)--(15.75, 1.4);
  \fill (14.75, 1.3) circle (1.3pt);

\draw[middlearrow={>}] (15.75, 1.4)--(11, 1.5);
  \fill (15.75, 1.4) circle (1.3pt);
\draw[middlearrow={>}] (11, 1.5)--(12, 1.6);
  \fill (11, 1.5) circle (1.3pt);
\draw[middlearrow={>}] (12, 1.6)--(18.75, 1.7);
  \fill (12, 1.6) circle (1.3pt);
\draw[middlearrow={>}] (18.75, 1.7)--(19.75, 1.8);
  \fill (18.75, 1.7) circle (1.3pt);
\draw[middlearrow={>}] (19.75, 1.8)--(15, 1.9);
  \fill (19.75, 1.8) circle (1.3pt);

\end{tikzpicture}
\end{center}
\caption{$N_{A,B,C}$ for $M=M_{\WMS}$ and $v= v_0$ as described in the previous section, and with $A$ close to $\frac{1}{3}$. We drew some $\gamma\in \Xi_{A,B,C}$ in red. On the right we drew both $\gamma$ and $\overline{\gamma}$.}
\label{fig:edgepath}
\end{figure}

We say that there is an accident at the positive number $A$ if the limit $N^+_{A,B}$ (resp. $N^-_{A,B}$) in the sense of Mirzakhani-Wright is in a different stratum as $M$.
Let $\Acc_+$ (resp. $\Acc_-$) be the set of $A$s such that $N^+_{A,B}$ (resp. $N^-_{A,B}$) is not in the same stratum as $M$, and let $\Acc = \Acc_- \cup \Acc_+$.
The letter $\Acc$ stands for ``accident''.
We are now going to describe a criterion to find the real numbers $A$ that are accumulation points in $\Acc \subset \RR^{>0}$.
We will see in Lemma \ref{lem:Aplus} and Theorem \ref{thm:Acc} that $\Xi_{A,B,C}$ and $\Gamma_{A,B,C}$ are both related to accidents at $A$.

\begin{lem}
\label{lem:Aplus}
The following hold:
\begin{enumerate}
\item $\Acc_+ = \Acc_- = \Acc$,
\item $\Xi_{A,B,C} = \Xi_{A,B,C'}$ for all $C, C' \neq 0$ small enough and we write simply $\Xi_{A,B}$,
\item $t\in \Acc$ if and only if $\Xi_{A,B} \neq \varnothing$.
\end{enumerate}
\end{lem}
\begin{proof}
This proof is similar to the one of Proposition \ref{prop:lim_descr}
We first prove (ii) when $C$ and $C'$ have the same sign. 
Let $\gamma \in \Xi_{A,B,C}$ for some $C\neq 0$ (say $\gamma$ travels upwards), and consider $\overline{\gamma}$ the edge path around $\gamma$. 
Since $\overline{\gamma}$ is made of edges $e_k^i$ that correspond to horizontal edges of $M$, then for $A$ fixed $\RE(\dev(N_{A,B,C}, \gamma)) = \RE(\dev(N_{A,B,C}, \overline{\gamma}))$ is a constant function of $C$ as long as $C$ is small enough. 
In particular, if it is 0 for some $C$ small enough (i.e. $\gamma$ is vertical) then it is 0 for all $C'\neq 0$ sufficiently close to 0 and of same sign, and the proof of (ii) follows by symmetry.

We now prove (iii).
If $\Xi_{A,B} \neq \varnothing$, then we can find $\gamma\in \Xi_{A,B,C}$ travelling upwards. 
Again considering the edge path around $\gamma$ we see that $\dev(N_{A,B,C}, \gamma)$ is equal to $i \lambda C\in i\RR\subset \CC$ with $\lambda >0$. 
As $C\to 0$, $\gamma$ is pinched to a point thus $N_{A,B,C}$ does not have a limit in the same stratum as $M$, thus $t\in \Acc$.
If $\Xi_{A,B}=\varnothing$, then replicating the proof of Proposition \ref{prop:lim_descr} we have that the limits $N_{A,B}^\pm$ exist and are in the same stratum as $M$, thus $t\notin \Acc$, and the proof of (iii) is complete.

Note that minor changes in the proof of (iii) imply that $t\in \Acc_+$ if and only if $\Xi_{A,B,C} \neq \varnothing$ for some small $C>0$ and $t\in \Acc_-$ if and only if $\Xi_{A,B,C} \neq \varnothing$ for some small $C<0$. 
This together with (ii) imply (i).
\end{proof}

We would like to describe the topology of the set $\Acc$. 
More precisely, we want to be able to predict when an accident $A\in \Acc$ is an accumulation point. 
We start with the following Lemma that will be useful in the proof of Theorem \ref{thm:Acc}.
\begin{lem}
\label{lem:nonzero}
Let $\gamma \in \Xi_{A,B,C}$. Then $v(\gamma)\neq 0$.
\end{lem}
\begin{proof}
Suppose $v(\gamma) = 0$. Then, since $\gamma$ and $\overline{\gamma}$ are equal in homology, $v(\overline{\gamma}) = 0$ as well.
But then, as in the proof of Lemma \ref{lem:Aplus}, we write $\overline{\gamma}$ as a concatenation of edges $e_j = a_k^i $ for some $i,k$. 
Now, using that $v(\overline{\gamma}) = 0$ for the second equality, the following equalities hold.
\[
\dev(N_{A,B,C}, \gamma) = \dev(N_{A,B,C}, \overline{\gamma}) = \dev(M_{A,B,C}, \overline{\gamma}) = \sum_j \dev(M_{A,B,C}, e_j) 
\]
The sum on the right is a sum of vectors with positive real parts. This contradicts the fact that $\gamma\in \Xi_{A,B,C}$. Hence $v(\gamma) \neq 0$.
\end{proof}

Note that Lemma \ref{lem:Aplus} says that for a Veech surface pushed along a suitable direction, accidents appear at $A$ exactly when there are vertical saddle connections in $\TN_{A,B,C}$. Theorem \ref{thm:Acc} belows implies that such accidents are accumulation points exactly when there are vertical loops in $\TN_{A,B,C}$.
\begin{thm}
\label{thm:Acc}
Suppose $A_0\in \Acc$. Then the following dichotomy hold:
\begin{itemize}
	\item If $\Gamma_{A_0,B}\neq \varnothing$, then there are infinitely many accidents around any neighborhood of $A_0$. More precisely, there is a homeomorphism preserving the natural order on $\RR$ between $\{0\} \cup \bigcup\{\pm \dfrac{1}{n}\}$ and $[A_0-\epsilon, A_0 + \epsilon]\cap \Acc$.
	\item If $\Gamma_{A_0,B}= \varnothing$, then $A_0$ is isolated in $\Acc$. 
\end{itemize}
\end{thm}
\begin{proof}
Proof of the second claim.
For any $L>0$ the set of saddle connections (resp. concatenations of successive saddle connections) $\gamma$ of length less than $L$ and with direction close to vertical is finite. 
By Lemma \ref{lem:nonzero}, $v(\gamma) \neq 0$ for any such saddle connection, so by choosing $\epsilon$ small enough we ensure that any vertical saddle connection in $\Xi_{A,B}$ for $A\neq A_0 \in [A_0-\epsilon, A_0+\epsilon]$ has to be of length at least $L$. Chosing $L$ big enough that an vertical saddle connection (resp. concatenation of successive saddle connections) starting in a half-cylinder $\Cyl_i^\pm$ intersects its mid-waist, we see that $\Xi_{A,B} $ must be empty.

Proof of the first claim. Start with some $\gamma\in \Gamma_{A_0,B}$.
For any $C>0$ small enough, $\gamma$ defines a horizontal cylinder $\widetilde{\Cyl}_{A_0, B, C}$ on $N_{A_0, B, C}$ which is bounded by two paths $\gamma_1, \gamma_2 \in \Xi_{A_0,B}$ homotopic to $\gamma$. 
Lemma \ref{lem:nonzero} implies that $v(\gamma) \neq 0$, thus for $\epsilon$ small and any $A\in [A_0 - \epsilon, A_0 + \epsilon]$ with $A\neq A_0$, the cylinder $\widetilde{\Cyl}_{A, B, C}$ is not horizontal anymore in $N_{A,B,C}$, i.e. $\RE(\dev(N_{A,B,C}, \gamma_i)) \neq 0$.
Therefore, we can define a map $\phi_{A,B,C}: [0,1] \to [0, 1]$ that maps $t\in [0,1]$ to the number $s$ such that the downward vertical line in $N_{A,B,C}$ that starts at $\gamma_1(t)$ intersects for the first time $\gamma_2([0,1])$ at $\gamma_2(s)$. Then there is a vertical saddle connection contained entirely in $\widetilde{\Cyl}_{A,B,C}$ exactly when there is a $t$ such that both $\gamma_1(t)$ and $\gamma_2(\phi_{A,B,C}(t))$ are singular points.
As in the case of $M = M_{\WMS}$ and $v = v_0$ studied in Section \ref{sec:WMSpush}, this happens exactly on an infinite discrete subspace of $[A_0-\epsilon, A_0) \cup (A_0, A_0 + \epsilon]$.
Note that there might be several (but finitely many) horizontal cylinders similar to $\widetilde{\Cyl}_{A_0,B,C}$, and that any other accident can be ruled out by choosing a smaller $\epsilon$ if necessary, as in the proof of the second claim. Therefore the first claim follows.
\end{proof}

%
%

\chapter{Convergence to boundary of strata, in a simple case}
\label{chap:BS}
The Borel-Serre partial compactification of homogeneous spaces was defined in \cite{bs73} and similar constructions in strata have been introduced and used in \cite{emz03}, \cite{bsw15}, \cite{g15} and \cite{c15}. See also \cite{w15BS}, in preparation.

As we saw in the previous section, the ``infinite catastrophes'' phenomenon occurs when a cylinder degenerates, i.e. when its circumference approaches 0. In this section we look back at the example from Chapter \ref{sec:background} and more generally describe the convergence to the a Borel-Serre type boundary of strata, in the simple case where a single ``simple'' cylinder degenerate.
We call simple cylinder a cylinder for which the two outer loops are single saddle connections connecting different singular points. Otherwise put, the outer loops contain exactly one singular point each, and these two singular points are distinct. 
$A$ and $C$ being fixed, let $\Cyl_{A,C}$ be the cylinder described in Figure \ref{fig:convex_quad} (it does not depend on the parameter $B$). Then $\Cyl_{A,C}$ is simple, and infinitely many other examples of simple cylinders can be constructed using for instance \cite{w15}.

Let us now define the three types of shapes of small simple cylinders: infinitesimal cylinders, IETs on infinitesimal slits and infinitely thin cylinders.
\begin{itemize} 
	\item Infinitesimal cylinders are (genuine) cylinders normalized to have area 1. They are parametrized by the direction $\theta\in S^1$ of their outer circles, their modulus $m\in (0, \infty)$ (height divided by circumpherence of outer circles) and their twist, i.e. the translation $t \in S^1$ between the two outer circles. We write $\Bound_1$ for the set of all infinitesimal cylinders.
	\item IETs on infinitesimal slits are IETs on the unit slit of $\RR^2$ centered at 0. Since we are only considering simple cylinders, this IET is just a translation $t\in S^1$. They can be thought of as modulus-0 infinitesimal cylinder, and they are parametrized by the direction $\theta\in S^1$ of the slit, and the translation $t\in S^1$. We write $\Bound_0$ for the set of IETs on infinitesimal slits.
	\item Infinitely thin cylinders can be thought of as modulus-$\infty$ cylinders, and are parametrized by a direction $\theta\in S^1$, a length (or height) $l\in [0, \infty]$ and a twist. We write $\Bound_\infty$ for the set of Infinitely thin cylinders.
\end{itemize}
See Figure \ref{fig:cylinders} for examples of infinitesimal cylinders, IETs on infinitesimal slits and infinitely thin cylinders.
See also Figure \ref{fig:Bound} for how the pieces $\Bound_0$, $\Bound_1$ and $\Bound_\infty$ are glued to each other into the topological space $\Bound$.

Let $S\in \Bound$ be an infinitesimal shape with direction $\theta$, height $h$, modulus $m$ and twist $t$. Let $M_n$ be a sequence of surfaces with simple cylinders $\Cyl_n$ of direction $\theta_n$, height $h_n$, modulus $m_n$, twist $t_n$ and circumpherence going to zero.
We denote by $M'_n$ the surface obtained from $M_n$ by removing $\Cyl_n$ and gluing the two remaining slits together.  
For the scope of this paper, we will say that $M_n$ converges to the pair $(M, S)$ if $M'_n \to M$ in the Mirzakhani-Wright topology, and $(\theta_n, h_n, m_n, t_n)\to (\theta, h, m, t)$.

In order to understand how a subspace of strata can approach a pair $(M, S)$, we define spiraling in the following way. 
This definition is inspired by the previous work \cite{bc16} co-authored with Hyungryul Baik.
\begin{itemize}
	\item We say that a subspace $\Push$ of some stratum approaches a pair $(M,S)$ if there is a sequence $M_n\in \Push$ that converges to the pair $(M,S)$.
	\item We say that a $\Push$ approaches $(M,S)$ in a spiraling way if it approaches $(M,S)$ and if there is no continuous path $t\mapsto M_t \in\Push$ with simple cylinders $\Cyl_t$ that converges to $(M, S)$.
	\item We say that $\Push$ approaches $(M,S)$ in a non-spiraling way if such a path exists.
\end{itemize}

%
	
%

\begin{figure}
\begin{center}
\begin{tikzpicture}[scale=.45]

\draw[-] (-.06, .08)--(3.94, 3.08);
\draw[-] (4.06, 2.92)--(.06, -.08);

\fill (0, 0) circle (.13);
\fill (0.94, .86) circle (.13);
\fill (3.06, 2.17) circle (.13);
\fill (4, 3) circle (.13);
\draw[-,densely dashed,red] (1.34,2.38)--(2.66,.63);
 
\draw[-] (4.7, .4)--(8.7,3.4);
\draw[thick,-] (8.7,3.4)--(9.3,2.6);
\draw[-] (9.3,2.6)--(5.3, -.4);
\draw[thick,-] (5.3, -.4)--(4.7,.4);
\fill[orange] (4.7, .4) circle (.13);
\fill[orange] (8.7, 3.4) circle (.13);
\fill[blue] (6.3, .35) circle (.13);
\draw[-,densely dashed,red] (6.1,2.7)--(7.9,.3);
\draw[-, purple] (9,3.3)--(9, 2.7);
\draw[-, purple] (5,-.3)--(5, .3);

\draw[thick,-] (11, 1.45)--(12.2,-.15);
\draw[-] (12.2,-.15)--(14.2,1.35);
\draw[thick,-] (14.2,1.35)--(13, 2.95);
\draw[-] (13, 2.95)--(11,1.45);
\fill[orange] (11, 1.45) circle (.13);
\fill[orange] (13, 2.95) circle (.13);
\fill[blue] (12.7, .225) circle (.13);
\draw[-,densely dashed,red] (13.8,-.2)--(11.4,3);
\draw[-, purple] (11.6,.35)--(11.6, .95);
\draw[-, purple] (13.6,1.85)--(13.6, 2.45);

\draw[thick,-] (15, 5.35)--(21,-2.65);
\draw[-] (21,-2.65)--(21.4,-2.35);
\draw[thick,-] (21.4,-2.35)--(15.4, 5.65);
\draw[-] (15.4, 5.65)--(15,5.35);
\fill[orange] (15, 5.35) circle (.13);
\fill[orange] (15.4, 5.65) circle (.13);
\fill[blue] (21.1, -2.575) circle (.13);
\draw[-,densely dashed,red] (21.8,-3.2)--(14.6,6.3);
\draw[-, purple] (18,1.05)--(18, 1.65);
\draw[-, purple] (18.4,1.35)--(18.4, 1.95);


\begin{scope} [shift={(-.8,0)}]
\draw[thick] (25.7, 3.8)--(29.3, -1);
\draw[-,densely dashed,red] (25.1, 4.6)--(25.7,3.8);
\draw[-,densely dashed,red] (29.3, -1)--(29.9,-1.8);
\fill (25.7, 3.8) circle (.13);
\fill (29.3, -1) circle (.13);
\node at (27.5, -2.5) {twist: $\frac{1}{4}$};
\end{scope}

\end{tikzpicture}
\end{center}
\caption{5 infinitesimal shapes. From left to right: an IET on a infinitesimal slit, three infinitesimal cylinders with moduli $\frac{1}{5}$, $\frac{4}{5}$ and $20$ respectively, and an infinitely thin cylinder with length 5 (note that any length in $[0, \infty]$ are allowed for thin cylinders). All of these shapes have orientation $\frac{6\pi}{5}$ as highlighted in dashed red, and have twist $\frac{1}{4} \in [0,1]/0 \sim 1 $ as can be seen from the shift between orange and blue singular points. For thin cylinders, we explicitly write the twist below the infinitesimal shape for convenience.}
\label{fig:cylinders}
\end{figure}
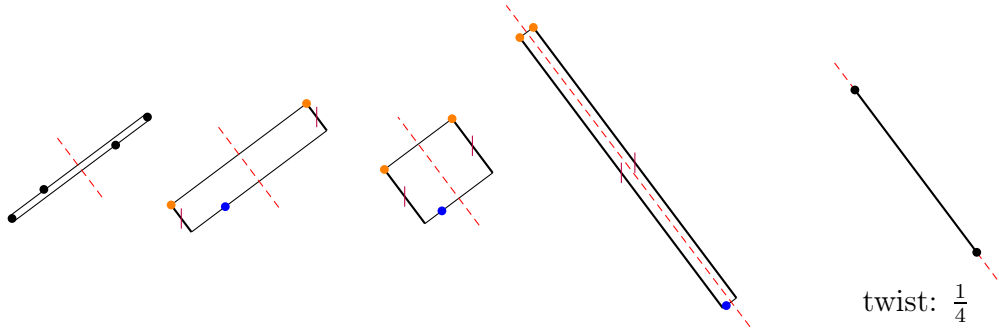


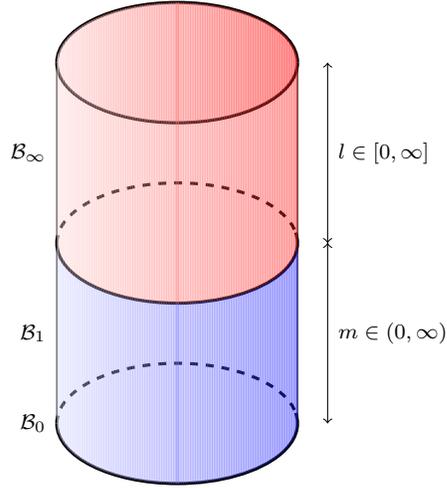
\begin{figure}
\begin{center}
\begin{tikzpicture}[scale=.8]
    \draw[very thick] (-2,6) arc (180:360:2cm and 1cm);
    \draw[very thick] (-2,6) arc (180:0:2cm and 1cm);
    \draw[very thick] (-2,3) arc (180:370:2cm and 1cm);
    \draw[very thick, dashed] (-2,3) arc (180:10:2cm and 1cm);
    \draw[very thick] (-2,0) arc (180:370:2cm and 1cm);
    \draw[very thick, dashed] (-2,0) arc (180:10:2cm and 1cm);

    \draw (-2,0) node[draw=none,fill=none,font=\scriptsize, left] {$\Bound_0$};
    \draw(-2,0)--(-2,3) node[draw=none,fill=none,font=\scriptsize,left,midway] {$\Bound_1$};
    \draw(2,0)--(2,3);
    \draw(-2,3)--(-2,6) node[draw=none,fill=none,font=\scriptsize,left,midway] {$\Bound_\infty$};
    \draw(2,3)--(2,6);

    \shade[left color=red!15!white,right color=red!90!white,opacity=0.4] (0, 6) circle (2cm and 1cm);
    \shade[left color=red!15!white,right color=red!90!white,opacity=0.3] (-2, 6) arc (180:360:2cm and 1cm) -- (2,3) arc (360:180:2cm and 1cm) -- cycle;
    \shade[left color=blue!15!white,right color=blue!90!white,opacity=0.3] (-2, 3) arc (180:360:2cm and 1cm) -- (2,0) arc (360:180:2cm and 1cm) -- cycle;

    \draw[<->] (2.5,0)--(2.5,3) node[draw=none,fill=none,font=\scriptsize,right,midway] {$m\in (0,\infty)$};
    \draw[<->] (2.5,3)--(2.5,6) node[draw=none,fill=none,font=\scriptsize,right,midway] {$l\in [0,\infty]$};

\end{tikzpicture}
\end{center}
\caption{The slice of $\Bound =\Bound_0 \sqcup \Bound_1 \sqcup \Bound_\infty$ for which the orientation $\theta$ is fixed is a topological closed cylinder. In this slice, elements of $\Bound_0$ form a circle, corresponding to all choice of a twist $t\in S^1$. Similarly, elements of $\Bound_1$ form a cylinder with choices of a modulus $m\in (0,\infty)$ and of a twist $t\in S^1$, and elements of $\Bound_\infty$ form a cylinder with choices of a length $l\in [0,\infty]$ and of a twist $t$. $\Bound$ is the product of this figure with the circle, corresponding with all choices of $\theta\in S^1$.}
\label{fig:Bound}
\end{figure}


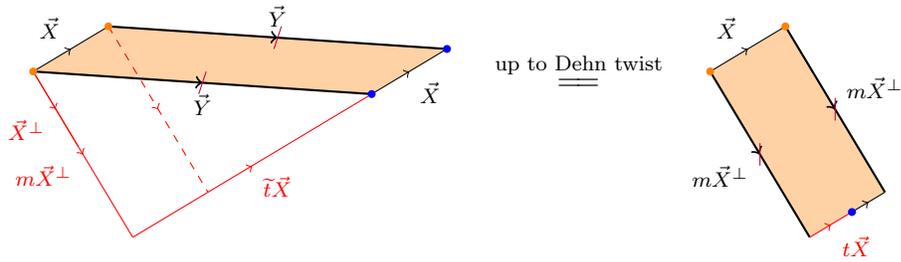
\begin{figure}
\begin{center}
\begin{tikzpicture}[scale=10]

\fill [fill=orange!35, draw = none] (0, 0) -- (.1, .06) -- (.55, .03) -- (.45, -.03) --(0, 0);

\draw[red,middlearrow={>}, dashed] (.1,.06)--(.23235,-.16059); 

\draw[middlearrow={>}] (0, 0)--(.1,.06) node[draw=none,fill=none,font=\scriptsize,above left,midway] {$\vec{X}$};
\draw[middlearrow={>}] (.45, -.03)--(.55,.03) node[draw=none,fill=none,font=\scriptsize,below right,midway] {$\vec{X}$};
\draw[thick,middlearrow={>}] (0, 0)--(.45,-.03) node[draw=none,fill=none,font=\scriptsize,below,midway] {$\vec{Y}$};
\draw[thick,middlearrow={>}] (.1, .06)--(.55,.03) node[draw=none,fill=none,font=\scriptsize,above,midway] {$\vec{Y}$};
\draw[-, purple] (.22,-.03)--(.23, 0);
\draw[-, purple] (.33,.06)--(.32, .03);

\draw[red,middlearrow={>}] (0,0)--(.13235,-.22059) node[red,draw=none,fill=none,font=\scriptsize,below left,midway] {$m \vec{X}^\perp$};
\draw[red,middlearrow={>}] (.13235,-.22059)--(.45,-.03) node[red,draw=none,fill=none,font=\scriptsize,below right,midway] {$\widetilde{t} \vec{X}$};
\draw[red,middlearrow={>}] (0, 0)--(.06,-.1) node[red,draw=none,fill=none,font=\scriptsize,below left,midway] {$\vec{X}^\perp$};

\fill[orange] (0, 0) circle (.15pt);
\fill[orange] (.1, .06) circle (.15pt);
\fill[blue] (.45, -.03) circle (.15pt);
\fill[blue] (.55, .03) circle (.15pt);

\begin{scope} [shift={(.2,0)}]

\fill [fill=orange!35, draw = none] (.7, 0) -- (.83235, -.22059) -- (.93235, -.16059) -- (.8, .06) --(.7, 0);
\draw[middlearrow={>}] (.7, 0)--(.8,.06) node[draw=none,fill=none,font=\scriptsize,above left,midway] {$\vec{X}$};
\draw[middlearrow={>}] (.88841, -.18696)--(.93235,.-.16059); 
\draw[thick,middlearrow={>}] (.7,0)--(.83235,-.22059) node[draw=none,fill=none,font=\scriptsize,below left,midway] {$m \vec{X}^\perp$};
\draw[thick,middlearrow={>}] (.8,.06)--(.93235,-.16059) node[draw=none,fill=none,font=\scriptsize,above right,midway] {$m \vec{X}^\perp$};

\draw[red,middlearrow={>}] (.83235, -.22059)--(.88841,-.18696) node[red,draw=none,fill=none,thick,font=\scriptsize,below right,midway] {$t\vec{X}$};
\fill[orange] (.7, 0) circle (.15pt);
\fill[orange] (.8, .06) circle (.15pt);
\fill[blue] (.88841, -.18696) circle (.15pt);
\draw[-, purple] (.766175,-.0953)--(.766175, -.1253);
\draw[-, purple] (.866175,-.0353)--(.866175, -.0653);

\end{scope}

\node at (.725, 0) {$\myeq$};
\end{tikzpicture}
\end{center}
\caption{Left: the cylinder $\Cyl_{A, C}$ is defined by a parallelogram with vectors $\vec{X} =(2A-1,2C)$ and $\vec{Y} = (1-A, -C)$. The two edges labeled $\vec{Y}$ are glued together, whereas the edges labeled $\vec{X}$ are glued to different segments of the surface $N_{A,C}$. Right: we put $\Cyl$ into its canonical form with $t =\widetilde{t} \mod 1$. We have $\vec{X^{\perp}}=(2C, -(2A-1))$, $m =-\frac{X\wedge Y}{\|X\|^2}$ and $\widetilde{t} = \frac{X\cdot Y}{\|X\|^2}$.}
\label{fig:parallelogram}
\end{figure}

Going back to the example of the WMS pushed along $v_0$ (see Chapter \ref{chap:push}), we will look at the ``phase-space'' of $A, B, C$ with A close to $\frac{1}{2}$, $B, C$ close to 0 and $C>0$. 
In fact, $B$ will play absolutely no role in the discussion so we set it to 0.
The cylinder $\Cyl_{A,C}$ defined by a parallelogram as in Figure \ref{fig:convex_quad} is simple, and its circumpherence goes to 0 as $A\to \frac{1}{2}$ and $C\to 0$.
In order to see how $N_{A,C}$ approaches the boundary of stratum as $(A,C)\to (\frac{1}{2}, 0)$, let us look at the height $h_{A,C}$ of $\Cyl_{A,C}$, its modulus $m_{A,C}$ and its twist $t_{A,C}$. Elementary computations give the following lemma.

\begin{lem}
\label{lem:spirals}
\begin{align}
X_{A,C} = 2 \cdot (A-\frac{1}{2},C) \\
h_{A,C} =-\frac{X_{A,C}\wedge Y_{A,C}}{\|X_{A,C}\|} = \frac{C}{\sqrt{(2A-1)^2+4C^2}} \\
m_{A,C}  =-\frac{X_{A,C}\wedge Y_{A,C}}{\|X_{A,C}\|^2} = \frac{C}{(2A-1)^2+4C^2} \\
t_{A,C} =\widetilde{t_{A,C}} \mod 1 \text{, with } \widetilde{t_{A,C}} = \frac{X_{A,C}\cdot Y_{A,C}}{\|X_{A,C}\|^2}=\frac{(2A-1)(1-A)-2C^2}{(2A-1)^2+4C^2}
\end{align} 

(1) implies that $\Cyl_{A,C}$ has same direction as $(A-\frac{1}{2}, C)$ and circumpherence twice the norm $\|(A-\frac{1}{2}, C)\|$.

(2) implies that the locus where $h_{A,C}$ is constant and equal to $h$ is the union of the half-lines passing through $(A,C) = (\frac{1}{2}, 0)$ and with slope $s=\pm\frac{2h}{\sqrt{1-4h^2}}$. 
Note that as $|s|\to \infty$, $h= \frac{|s|}{2\sqrt{1+s^2}} \to \frac{1}{2}$ and as $s\to 0$, $h\to 0$. 

(3) implies that the locus where $m_{A,C}$ is constant and equal to $m$ is the circle passing through $(A, C) = (\frac{1}{2}, 0)$ and with center $(0,\frac{1}{8m})$.

(4) implies that the locus where $\widetilde{t_{A,C}}$ is constant and equal to $\widetilde{t}$ is the half-circle passing through $(\frac{1}{2},0)$ and $(\frac{1+\widetilde{t}}{1+2\widetilde{t}},0)$ and with center on the $A$-axis if $\widetilde{t}>-\frac{1}{2}$, the vertical half-line passing through $(\frac{1}{2},0)$ if $\widetilde{t}=-\frac{1}{2}$ and the half-circle passing through $(\frac{1}{2},0)$ and $(\frac{1+\widetilde{t}}{1+2\widetilde{t}},0)$ and with center on the $A$-axis if $\widetilde{t}<-\frac{1}{2}$. 
Note that as $p=\frac{1+\widetilde{t}}{1+2\widetilde{t}} \to \frac{1}{2}^+$, $\widetilde{t}=\frac{1-p}{2p-1}\to \infty$ and as $p\to \frac{1}{2}^-$, $\widetilde{t}\to -\infty$.

See Figures \ref{fig:phasespace1}, \ref{fig:phasespace2} and \ref{fig:phasespace3} for these loci.
\end{lem}


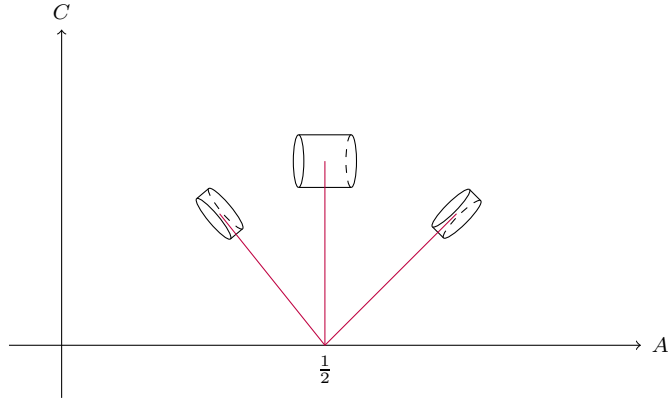
\begin{figure}
\begin{center}
\begin{tikzpicture}[scale=7]

\begin{scope} [shift={(.3,.25)}, rotate = 130, scale = .05]
    \draw (-1,.3) arc (180:360:1 and .2);
    \draw (-1,.3) arc (180:0:1 and .2);
    \draw (-1,-.3) arc (180:370:1 and .2);
    \draw[dashed] (-1,-.3) arc (180:10:1 and .2);

    \draw(-1,-.3)--(-1,.3);
    \draw(1,-.3)--(1,.3);
\end{scope}

\begin{scope} [shift={(.5, .35)}, rotate = 90, scale = .05]
    \draw (-1,1) arc (180:360:1 and .2);
    \draw (-1,1) arc (180:0:1 and .2);
    \draw (-1,-1) arc (180:370:1 and .2);
    \draw[dashed] (-1,-1) arc (180:10:1 and .2);

    \draw(-1,-1)--(-1,1);
    \draw(1,-1)--(1,1);
\end{scope}

\begin{scope} [shift={(.75, .25)}, rotate = 45, scale = .05]
    \draw (-1,.3) arc (180:360:1 and .2);
    \draw (-1,.3) arc (180:0:1 and .2);
    \draw (-1,-.3) arc (180:370:1 and .2);
    \draw[dashed] (-1,-.3) arc (180:10:1 and .2);

    \draw(-1,-.3)--(-1,.3);
    \draw(1,-.3)--(1,.3);
\end{scope}

    \draw[->] (-.1,0)--(1.1,0);
    \draw (1.1,0) node[draw=none,fill=none,font=\scriptsize,right] {$A$};
    \draw[->] (0,-.1)--(0,.6);
    \draw (0,.6) node[draw=none,fill=none,font=\scriptsize,above] {$C$};

    \draw (.5,0) node[draw=none,fill=none,font=\scriptsize,below] {$\frac{1}{2}$};

%

    \draw[purple] (.5,0)--(.5, .35);

    \draw[purple] (.5,0)--(.75,.25);

    \draw[purple] (.5,0)--(.3,.25);

\end{tikzpicture}
\end{center}
\caption{A vector $(A-1/2, C)$ is equal to $\frac{1}{2} \dev(N_{A,B,C}, \gamma)$ where $\gamma $ is one of the saddle connections bounding the cylinder in Figure \ref{fig:convex_quad}. In particular, as $A \to \frac{1}{2}$ this cylinder becomes horizontal ($\gamma$ becomes vertical), and for $A < \frac{1}{2}$ the cylinder is flipped, i.e. $\RE(\dev(N_{A,B,C}, \gamma))<0$. We superimposed a sketch of the cylinder on top of $(A, C)$ for some values of $A$ and $C$.}
\label{fig:phasespace0}
\end{figure}


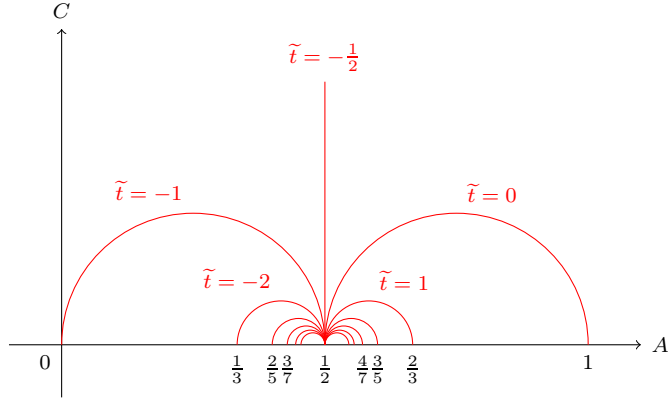
\begin{figure}
\begin{center}
\begin{tikzpicture}[scale=7]

    \draw[->] (-.1,0)--(1.1,0);
    \draw (1.1,0) node[draw=none,fill=none,font=\scriptsize,right] {$A$};
    \draw[->] (0,-.1)--(0,.6);
    \draw (0,.6) node[draw=none,fill=none,font=\scriptsize,above] {$C$};

    \draw (.5,0) node[draw=none,fill=none,font=\scriptsize,below] {$\frac{1}{2}$};
    \draw (.5,.5) node[red,draw=none,fill=none,font=\scriptsize,above] {$\widetilde{t}=-\frac{1}{2}$};

    \draw[red] (.5,0) arc (180:0:.25);
    \draw (1,0) node[draw=none,fill=none,font=\scriptsize,below] {1};
    \draw (.75,.25) node[red,draw=none,fill=none,font=\scriptsize,above right] {$\widetilde{t}=0$};

    \draw[red] (.5,0) arc (180:0:.083333);
    \draw (.666667,0) node[draw=none,fill=none,font=\scriptsize,below] {$\frac{2}{3}$};
    \draw (.583333,.083333) node[red,draw=none,fill=none,font=\scriptsize,above right] {$\widetilde{t}=1$};

    \draw[red] (.5,0) arc (180:0:.05);
    \draw (.6,0) node[draw=none,fill=none,font=\scriptsize,below] {$\frac{3}{5}$};
    \draw[red] (.5,0) arc (180:0:.035714);
    \draw (.571429,0) node[draw=none,fill=none,font=\scriptsize,below] {$\frac{4}{7}$};
    \draw[red] (.5,0) arc (180:0:.0277778);
    \draw[red] (.5,0) arc (180:0:.0227273);

    \draw[red] (.5,0) arc (0:180:.25);
    \draw (0,0) node[draw=none,fill=none,font=\scriptsize,below left] {0};
    \draw (.25,.25) node[red,draw=none,fill=none,font=\scriptsize,above left] {$\widetilde{t}=-1$};
    
    \draw[red] (.5,0) arc (0:180:.083333);
    \draw (.333333,0) node[draw=none,fill=none,font=\scriptsize,below] {$\frac{1}{3}$};
    \draw (.416667,.083333) node[red,draw=none,fill=none,font=\scriptsize,above left] {$\widetilde{t}=-2$};
    
    \draw[red] (.5,0) arc (0:180:.05);
    \draw (.4,0) node[draw=none,fill=none,font=\scriptsize,below] {$\frac{2}{5}$};
    \draw[red] (.5,0) arc (0:180:.035714);
    \draw (.428571,0) node[draw=none,fill=none,font=\scriptsize,below] {$\frac{3}{7}$};
    \draw[red] (.5,0) arc (0:180:.0277778);
    \draw[red] (.5,0) arc (0:180:.0227273);

    \draw[red] (.5,0)--(.5, .5);
\end{tikzpicture}
\end{center}
\caption{Loci for which $\widetilde{t_{A,C}}$ is constant are half-circles centered on the $A$-axis.}
\label{fig:phasespace1}
\end{figure}


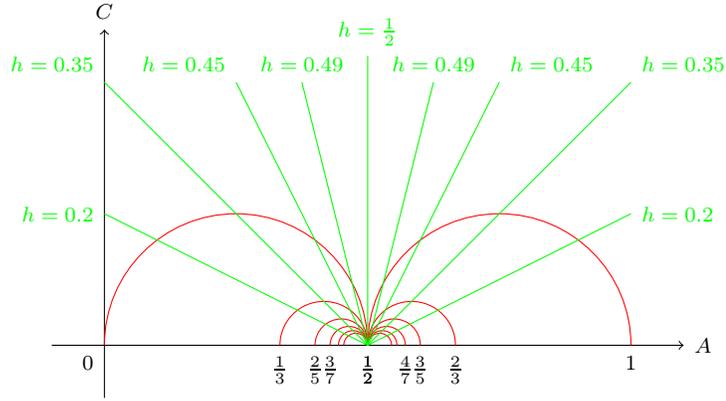
\begin{figure}
\begin{center}
\begin{tikzpicture}[scale=7]

    \draw[->] (-.1,0)--(1.1,0);
    \draw (1.1,0) node[draw=none,fill=none,font=\scriptsize,right] {$A$};
    \draw[->] (0,-.1)--(0,.6);
    \draw (0,.6) node[draw=none,fill=none,font=\scriptsize,above] {$C$};

    \draw (.5,0) node[draw=none,fill=none,font=\scriptsize,below] {$\frac{1}{2}$};
    \draw (.5,.55) node[green,draw=none,fill=none,font=\scriptsize,above] {$h=\frac{1}{2}$};

    \draw (.5,0) node[draw=none,fill=none,font=\scriptsize,below] {$\frac{1}{2}$};
    \draw[red] (.5,0) arc (180:0:.25);
    \draw (1,0) node[draw=none,fill=none,font=\scriptsize,below] {1};
    \draw[red] (.5,0) arc (180:0:.083333);
    \draw (.666667,0) node[draw=none,fill=none,font=\scriptsize,below] {$\frac{2}{3}$};
    \draw[red] (.5,0) arc (180:0:.05);
    \draw (.6,0) node[draw=none,fill=none,font=\scriptsize,below] {$\frac{3}{5}$};
    \draw[red] (.5,0) arc (180:0:.035714);
    \draw (.571429,0) node[draw=none,fill=none,font=\scriptsize,below] {$\frac{4}{7}$};
    \draw[red] (.5,0) arc (180:0:.0277778);
    \draw[red] (.5,0) arc (180:0:.0227273);

    \draw[red] (.5,0) arc (0:180:.25);
    \draw (0,0) node[draw=none,fill=none,font=\scriptsize,below left] {0};
    \draw[red] (.5,0) arc (0:180:.083333);
    \draw (.333333,0) node[draw=none,fill=none,font=\scriptsize,below] {$\frac{1}{3}$};
    \draw[red] (.5,0) arc (0:180:.05);
    \draw (.4,0) node[draw=none,fill=none,font=\scriptsize,below] {$\frac{2}{5}$};
    \draw[red] (.5,0) arc (0:180:.035714);
    \draw (.428571,0) node[draw=none,fill=none,font=\scriptsize,below] {$\frac{3}{7}$};
    \draw[red] (.5,0) arc (0:180:.0277778);
    \draw[red] (.5,0) arc (0:180:.0227273);

    \draw[green] (.5,0)--(.5, .55);

    \draw[green] (.5,0)--(1,.25);
    \draw (1,.25) node[green,draw=none,fill=none,font=\scriptsize,right] {$h=0.2$};
    \draw[green] (.5,0)--(1,.5);
    \draw (1,.5) node[green,draw=none,fill=none,font=\scriptsize,above right] {$h=0.35$};
    \draw[green] (.5,0)--(.75,.5);
    \draw (.75,.5) node[green,draw=none,fill=none,font=\scriptsize,above right] {$h=0.45$};
    \draw[green] (.5,0)--(.625,.5);
    \draw (.625,.5) node[green,draw=none,fill=none,font=\scriptsize,above] {$h=0.49$};

    \draw[green] (.5,0)--(0,.25);
    \draw (0,.25) node[green,draw=none,fill=none,font=\scriptsize,left] {$h=0.2$};
    \draw[green] (.5,0)--(0,.5);
    \draw (0,.5) node[green,draw=none,fill=none,font=\scriptsize,above left] {$h=0.35$};
    \draw[green] (.5,0)--(.25,.5);
    \draw (.25,.5) node[green,draw=none,fill=none,font=\scriptsize,above left] {$h=0.45$};
    \draw[green] (.5,0)--(.375,.5);
    \draw (.375,.5) node[green,draw=none,fill=none,font=\scriptsize,above] {$h=0.49$};

\end{tikzpicture}
\end{center}
\caption{Loci for which $h_{A,C}$ is constant are half-lines ending at $(\frac{1}{2},0)$.}
\label{fig:phasespace2}
\end{figure}


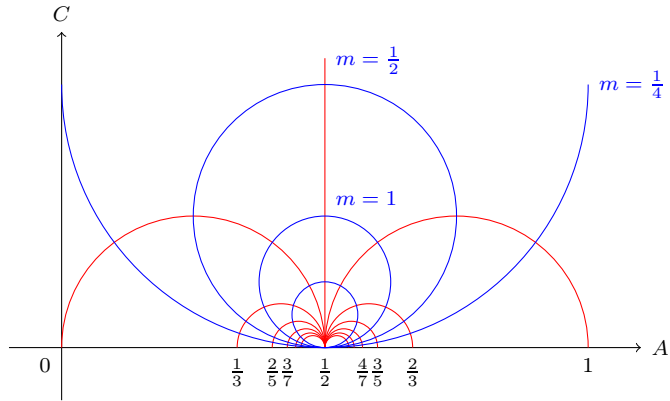
\begin{figure}
\begin{center}
\begin{tikzpicture}[scale=7]

    \draw[->] (-.1,0)--(1.1,0);
    \draw (1.1,0) node[draw=none,fill=none,font=\scriptsize,right] {$A$};
    \draw[->] (0,-.1)--(0,.6);
    \draw (0,.6) node[draw=none,fill=none,font=\scriptsize,above] {$C$};

    \draw (.5,0) node[draw=none,fill=none,font=\scriptsize,below] {$\frac{1}{2}$};
    \draw[red] (.5,0) arc (180:0:.25);
    \draw (1,0) node[draw=none,fill=none,font=\scriptsize,below] {1};
    \draw[red] (.5,0) arc (180:0:.083333);
    \draw (.666667,0) node[draw=none,fill=none,font=\scriptsize,below] {$\frac{2}{3}$};
    \draw[red] (.5,0) arc (180:0:.05);
    \draw (.6,0) node[draw=none,fill=none,font=\scriptsize,below] {$\frac{3}{5}$};
    \draw[red] (.5,0) arc (180:0:.035714);
    \draw (.571429,0) node[draw=none,fill=none,font=\scriptsize,below] {$\frac{4}{7}$};
    \draw[red] (.5,0) arc (180:0:.0277778);
    \draw[red] (.5,0) arc (180:0:.0227273);

    \draw[red] (.5,0) arc (0:180:.25);
    \draw (0,0) node[draw=none,fill=none,font=\scriptsize,below left] {0};
    \draw[red] (.5,0) arc (0:180:.083333);
    \draw (.333333,0) node[draw=none,fill=none,font=\scriptsize,below] {$\frac{1}{3}$};
    \draw[red] (.5,0) arc (0:180:.05);
    \draw (.4,0) node[draw=none,fill=none,font=\scriptsize,below] {$\frac{2}{5}$};
    \draw[red] (.5,0) arc (0:180:.035714);
    \draw (.428571,0) node[draw=none,fill=none,font=\scriptsize,below] {$\frac{3}{7}$};
    \draw[red] (.5,0) arc (0:180:.0277778);
    \draw[red] (.5,0) arc (0:180:.0227273);

    \draw[red] (.5,0)--(.5, .55);

    \draw[blue] (0, .5) arc (180:360:.5);
    \draw (1,.5) node[blue,draw=none,fill=none,font=\scriptsize,right] {$m=\frac{1}{4}$};
    \draw[blue] (.5, .25) circle (.25);
    \draw (.5,.5) node[blue,draw=none,fill=none,font=\scriptsize,above right] {$m=\frac{1}{2}$};
    \draw[blue] (.5, .125) circle (.125);
    \draw (.5,.25) node[blue,draw=none,fill=none,font=\scriptsize,above right] {$m=1$};
    \draw[blue] (.5, .0625) circle (.0625);

\end{tikzpicture}
\end{center}
\caption{Loci for which $m_{A,C}$ is constant are circles centered on the $C$-axis.}
\label{fig:phasespace3}
\end{figure}

\begin{thm}
\label{thm:spiraling}
Let $\Push = \{ N_{A,B,C}; \; \frac15 <A<1,\, -\frac12<B<\frac12,\, 0<C<\frac12 \} \subset \Push(\MM_{\WMS})$. Then the following holds.
\begin{itemize}
	\item For any height $h\in [0, \frac12)$, angle $\theta \in [0, \pi]$ such that $\tan(\theta)=\pm\frac{2h}{\sqrt{1-4h^2}}$ and twist $t\in S^1$, $\Push$ approaches $(N_{\frac12,0,0}, S)$ in a spiraling way, with $S$ the thin cylinder with angle $\theta$, height $h$ and twist $t$.
	\item For any twist $t \in S^1$, $\Push$ approaches $(N_{\frac12,0,0}, S)$ in a non-spiraling way, with $S$ the thin cylider with angle $\pi/2$, height $\frac12$ and twist $t$.
	\item For any modulus $m\in [0, \infty]$, angle $\theta\in \{0, \pi\}$ and twist $t\in S^1$, $\Push$ approaches $(N_{\frac12,0,0}, S)$ in a spiraling way, with $S$ the infinitesimal cylinder with angle $\theta$, modulus $m$ and twist $t$. Note that the case $m=0$ corresponds to infinitesimal IETs.
\end{itemize}
 
\end{thm}
\begin{proof}
We start with the second claim which is easier to prove. Given a twist $t\in [0,1)$, let $\gamma: [0,1] \to \{(A,C); \; \frac15<A<1, \, 0<C<\frac12\}$ be an arc of the upper half circle centered on the line $C=0$ and passing through $(\frac12,0)$ and $(\frac{1+t}{1+2t})$, and with $\gamma(1) = (\frac12,0)$. Then letting $(A_x, C_x)=\gamma(x)$, Lemma \ref{lem:spirals} implies that the path $x \mapsto N_{A_x,0, C_x}$ converges to $(N_{\frac12, 0,0}, S)$ as $x\to 1$, as was to be shown.

We now prove the first claim, assuming without loss of generality that $t=0$. 
We first show that there is a sequence $(A_n, C_n)$ such that $N_{A_n, 0, C_n}$ converges to $(N_{\frac12,0,0}, S)$.
But from Lemma \ref{lem:spirals} and as can be seen in Figure \ref{fig:phasespace2}, we can achieve this by defining $(A_n, C_n)$ as the intersection of the half-line of slope $s = \frac{2h}{\sqrt{1-4h^2}}$ passing through $(\frac12, 0)$ and the half-circles as in the proof of the second claim. 
Now, let us assume by contradiction that there is a continuous path $\gamma: [0,1] \to \{(A,C); \; \frac15<A<1, \, 0<C<\frac12\}$ such that $N_{A_x, 0, C_x}$ converges to $(N_{\frac12,0,0}, S)$ as $x\to 1$. 
Since $N_{A_x,0,C_x}\to N_{\frac12, 0,0}$ in the sense of Mirzakhani-Wright, Theorem \ref{thm:balls} implies that $(A_x,C_x) \to (\frac12,0)$. 
Since it converges to $(N_{\frac12,0,0}, S)$, we see from Lemma \ref{lem:spirals} and Figure \ref{fig:phasespace2} that the slopes $s_x=C_x/A_x$ must converge to $\pm \frac{2h}{\sqrt{1-4h^2}}$, and that this prohibits $t_x = t_{A_x, C_x}$ from converging in $S^1$. Thus this claim is proven as well. The third claim is very similar.
\end{proof}

\chapter{Conclusion}

We have developed several techniques to study the geometry of the pushes $\Push(M, v)$, where $M$ is a Veech surface and $v$ is a suitable direction. 
These provide a better understanding of horocycle orbit closures in strata, in particular due to the description of the new ``infinite catastrophes in finite time'' phenomenon. 
Note that the example of $M_{\WMS}$ and $v_0$ studied in Section \ref{sec:WMSpush} was carefully chosen so that there is no local regluing around the accident $A=\frac12$ (see Theorem \ref{thm:balls}). 
In order to completely describe the geometry of $\overline{\Push(M_{\WMS}, v_0)}$, it would be necessary to understand if the local boundaries around the $A=\frac12$ and $C=0$ are global, i.e. that they are not identified with other local boundaries. 
As a next step, we would like to investigate, in the case of arbitrary Veech translations surfaces, how complicated local and global regluing could be.
Moreover, Remark \ref{rem:small_As} and Proposition \ref{prop:normal} together imply that for any Veech surface $M$ and direction $v$ such that $\Dir(M,v)$ is finite, $\Push(M, v)$ is an immersed manifold away from the push of the locus $\{ M_{A,B,C}; \; 0<A<A_0, B\in \RR, C=0 \}$ with some number $A_0$.
Because we are dealing with vectors that are not globally defined, some self-intersection phenomenon may appear, and in the future we would like to find a systematic method to predict if and when self-intersection occurs.

Note that it is crucial in Chapter \ref{chap:generalCase} that $M$ is a Veech surface, as this condition implies that if infinitely many catastrophes occur in finite time, the surface that shrinks to a point is made up of horizontal cylinders. 
Is it possible to find an example of a non-Veech surface $M$ and a suitable direction $v$ such that a non-periodic minimal subsurface shrinks to a point?
Is such a case, what would be the topology of $\Acc$ and the geometry of the push $\Push(M,v)$? Would it always be an immersed manifold with boundary?
More broadly, is it always true that horocycle orbit closures are manifold with boundary?

As suggested by Chapter \ref{chap:BS}, the study of how pushes $\Push(M,v)$ approach the Borel-Serre boundary is natural, and we expect it will yield interesting insight on how horocycle orbit closures approach the Borel-Serre boundary of strata. 

Another potential application of our techniques is related to the following question of Forni: for any horocycle-invariant ergodic measure $\mu$ on a stratum $\HH$, are the limits ${g_t}_\ast \mu$ as $t\to \pm \infty$ either $g_t$-invariant or do not converge in the space of probability measures of $\HH$? 
This question was answered affirmatively in \cite{bsw15} in the case $\HH=\HH^{\lbl}_1(1,1)$. The additional complexity studied in this thesis could potentially yield a negative answer in higher-dimensional strata.
Similarly, we are interested in the study of the limits $\Push(M, tv)$ as $t\to \pm \infty$. When these limits exist, do they admit extra invariance? If not, what can be said of their geometries?

\bibliographystyle{plain}
\bibliography{ewms}

\begin{thebibliography}{10}

\bibitem{bc16}
H.~Baik and L.~Clavier.
\newblock The space of geometric limits of {A}belian subgroups of
  {$\text{PLS}_2(\mathbb{C})$}.
\newblock {\em Hiroshima Math. J.}, 46(1):to appear, 2016.

\bibitem{bsw15}
M.~Bainbridge, J.~Smillie, and B.~Weiss.
\newblock Dynamics of the horocycle flow on the eigenform loci in
  $\mathcal{H}(1,1)$.
\newblock 2015.
\newblock Preprint.

\bibitem{bs73}
A.~Borel and J.-P. Serre.
\newblock Corners and arithmetic groups.
\newblock {\em Comm. Math. Helv.}, 48:436--491, 1973.

\bibitem{cw10}
K.~Calta and K.~Wortman.
\newblock On unipotent flows in $\mathcal{H}(1,1)$.
\newblock {\em Ergodic Theory Dynam. Systems}, 30(2):379--398, 2010.

\bibitem{c15}
D.~Chen.
\newblock Degenerations of {A}belian differentials.
\newblock Preprint on webpage at \url{http://arxiv.org/abs/1504.01983}, 2015.

\bibitem{dh75}
A.~Douady and J.~Hubbard.
\newblock On the density of {S}trebel differentials.
\newblock {\em Invent. Math.}, 30(2):175--179, 1975.

\bibitem{emwm06}
A.~Eskin, J.~Marklof, and D.~Witte-Morris.
\newblock Unipotent flows on the space of branched covers of veech surfaces.
\newblock {\em Ergodic Theory Dynam. Systems}, 26(1):129--162, 2006.

\bibitem{ems03}
A.~Eskin, H.~Masur, and M.~Schmoll.
\newblock Billiards in rectangles with barriers.
\newblock {\em Duke Math. J.}, 118(3):427--463, 2003.

\bibitem{emz03}
A.~Eskin, H.~Masur, and A.~Zorich.
\newblock Moduli spaces of {A}belian differentials: the principal boundary,
  counting problems and the {Siegel-Veech} constants.
\newblock {\em Publ. Math. IHÉS}, 97:61--179, 2003.

\bibitem{EM}
A.~Eskin and M.~Mirzakhani.
\newblock Invariant and stationary measures for the ${\SL(2,\RR)}$ action on
  moduli space.
\newblock 2014.
\newblock Preprint.

\bibitem{EMM}
A.~Eskin, M.~Mirzakhani, and A.~Mohammadi.
\newblock Isolation, equidistribution, and orbit closures for the ${\SL(2,
  \RR)}$ action on moduli space.
\newblock 2015.
\newblock Preprint.

\bibitem{fm14}
G.~Forni and C.~Matheus.
\newblock Introduction to {T}eichm{\"u}ller theory and its applications to
  dynamics of interval exchange transformations, flows on surfaces and
  billiards.
\newblock {\em J. Mod. Dyn.}, 8(3--4):271--436, 2014.

\bibitem{g15}
Q.~Gendron.
\newblock The {D}eligne-{M}umford and the incidence variety compactifications
  of the strata of {$\Omega \mathcal{M}_g$}.
\newblock Preprint on webpage at \url{http://arxiv.org/abs/1503.03338}, 2015.

\bibitem{hs08}
F.~Herrlich and G.~Schmith{\"u}sen.
\newblock An extraordinary origami curve.
\newblock {\em Math. Nachr.}, 281(2):219--237, 2008.

\bibitem{hm79}
J.~Hubbard and H.~Masur.
\newblock Quadratic differentials and foliations.
\newblock {\em Acta Math.}, 142(3-4):221--274, 1979.

\bibitem{hs06}
P.~Hubert and T.~A. Schmidt.
\newblock {\em An introduction to {V}eech surfaces}, volume~1B.
\newblock Elsevier B.V., 2006.

\bibitem{my10}
C.~Matheus and J.~C. Yoccoz.
\newblock The action of the affine diffeomorphisms on the relative homology
  group of certain exceptionally symmetric origamis.
\newblock {\em J. Mod. Dyn.}, 4(3):453--486, 2010.

\bibitem{mw14}
Y.~Minsky and B.~Weiss.
\newblock Cohomology classes represented by measured foliations, and {M}ahler's
  question for interval exchanges.
\newblock {\em Ann. Sci. {É}c. Norm. Sup{é}r.}, 47(2):245--284, 2014.

\bibitem{mzw15}
M.~Mirzakhani and A.~Wright.
\newblock The boundary of an affine invariant submanifold.
\newblock 2015.
\newblock Preprint on webpage at \url{http://arxiv.org/abs/1508.01446}.

\bibitem{dwm}
D.~W. Morris.
\newblock {\em Ratner's theorems on unipotent flows}.
\newblock University of Chicago Press, 2005.

\bibitem{ratner123}
M.~Ratner.
\newblock Ergodic theory in hyperbolic space.
\newblock {\em Contemporary Math.}, 26:309--334, 1984.

\bibitem{ratner5}
M.~Ratner.
\newblock On measure rigidity of unipotent subgroups of semisimple groups.
\newblock {\em Acta Math.}, 165:229--309, 1990.

\bibitem{ratner4}
M.~Ratner.
\newblock Strict measure rigidity for unipotent subgroups of solvable groups.
\newblock {\em Invent. Math.}, 101:449--482, 1990.

\bibitem{ratner6}
M.~Ratner.
\newblock On {R}aghunathan's measure conjecture.
\newblock {\em Ann. of Math.}, 134:545--607, 1991.

\bibitem{ratner7}
M.~Ratner.
\newblock Raghunathan's topological conjecture and distributions of unipotent
  flows.
\newblock {\em Duke Math. J.}, 63(1):235--280, 1991.

\bibitem{ss13}
J.-C. Schlage-Puchta and G.~Weitze-Schmith{\"u}sen.
\newblock Finite translation surfaces with maximal number of translations.
\newblock Preprint on webpage at \url{http://arxiv.org/abs/1311.7446}, 2013.

\bibitem{sw04}
J.~Smillie and B.~Weiss.
\newblock Minimal sets for flows on moduli space.
\newblock {\em Israel J. Math.}, 142:249--260, 2004.

\bibitem{sw15}
J.~Smillie and B.~Weiss.
\newblock Examples of horocycle invariant measures on the moduli space of
  translation surfaces.
\newblock 2015.
\newblock In preparation.

\bibitem{v89}
W.~Veech.
\newblock {T}eichm{\"u}ller curves in modular space, {E}isenstein series, and
  an application to triangular billiards.
\newblock {\em Invent. Math.}, 97:553--583, 1989.

\bibitem{expo:Wright}
A.~Wright.
\newblock Translation surfaces and their orbit closures: an introduction for a
  broad audience.
\newblock {\em EMS Surv. Math. Sci.}, 2(1):63--108, 2015.

\bibitem{w15BS}
C.~Wu.
\newblock Collision of cone points and {B}orel-{S}erre type compactification.
\newblock 2015.
\newblock In preparation.

\bibitem{w15}
C.~Wu.
\newblock The relative cohomology of {A}belian covers of the flat pillowcase.
\newblock {\em J. Mod. Dyn.}, 9(1):123--140, 2015.

\bibitem{expo:Zorich}
A.~Zorich.
\newblock {\em Flat surfaces}, volume~1 of {\em Frontiers in Number Theory,
  Physics and Geometry}.
\newblock Springer Verlag, 2006.

\end{thebibliography}

\end{document}